\documentclass[11pt]{article}

\usepackage{amsmath}
\usepackage{amsthm}
\usepackage{amssymb}
\usepackage{float}
\usepackage{enumerate}
\usepackage{graphicx}
\usepackage{setspace}
\usepackage{hyperref}
\hypersetup{colorlinks=true}

\usepackage[margin=1in]{geometry}

\usepackage{subfigure}
\usepackage{float}
\usepackage{color}

\usepackage{textcomp}
\usepackage{algorithm}
\usepackage{algorithmic}

\usepackage{epstopdf}

\newcommand{\oper}[1]{\mathcal{#1}}

\newcommand{\norm}[1]{\left\|#1\right\|}
\newcommand{\inner}[1]{\left<#1\right>}

\newcommand{\E}{\mathbb{E}}

\newcommand{\R}{\mathbb{R}}

\newtheorem{lem}{Lemma}
\newtheorem{cor}{Corollary}
\newtheorem{thm}{Theorem}

\newtheorem{define}{Definition}

\DeclareMathOperator*{\argmin}{argmin}

\DeclareMathOperator*{\rank}{rank}

\DeclareMathOperator{\supp}{supp}

\DeclareMathOperator{\opt}{opt}

\begin{document}
\title{Linear Convergence of Stochastic Iterative Greedy Algorithms with Sparse Constraints}

\author{Nam Nguyen, Deanna Needell and Tina Woolf}

\date{\today}
\maketitle

\begin{abstract}
Motivated by recent work on stochastic gradient descent methods, we develop two stochastic variants of greedy algorithms for possibly non-convex optimization problems with sparsity constraints. We prove linear convergence\footnote{Linear convergence is sometime called exponential convergence} in expectation to the solution within a specified tolerance.  This generalized framework applies to problems such as sparse signal recovery in compressed sensing, low-rank matrix recovery, and covariance matrix estimation, giving methods with provable convergence guarantees that often outperform their deterministic counterparts.  We also analyze the settings where gradients and projections can only be computed approximately, and prove the methods are robust to these approximations.  We include many numerical experiments which align with the theoretical analysis and demonstrate these improvements in several different settings.
\end{abstract}

\section{Introduction}
\label{sec::Intro}

Over the last decade, the problem of high-dimensional data inference from limited observations has received significant consideration, with many applications arising from signal processing, computer vision, and machine learning. In these problems, it is not unusual that the data often lies in hundreds of thousands or even million dimensional spaces while the number of collected samples is sufficiently smaller. Exploiting the fact that data arising in real world applications often has very low intrinsic complexity and dimensionality, such as sparsity and low-rank structure, recently developed statistical models have been shown to perform accurate estimation and inference. These models often require solving the following optimization with the constraint that the model parameter is sparse: 
\begin{equation}
\label{opt::general loss with sparse vector constraint}
\min_w  F(w)  \quad\quad \text{subject to} \quad \norm{w}_0 \leq k.
\end{equation}
Here, $F(w)$ is the objective function that measures the model discrepancy, $\norm{w}_0$ is the $\ell_0$-norm that counts the number of non-zero elements of $w$, and $k$ is a parameter that controls the sparsity of $w$.

In this paper, we study a more unified optimization that can be applied to a broader class of sparse models. First, we define a more general notion of sparsity. Given the set $\oper D = \{d_1,d_2,... \}$ consisting of vectors or matrices $d_i$, which we call atoms, we say that the model parameter is sparse if it can be described as a combination of only a few elements from the atomic set $\oper D$. Specifically, let $w \in \R^n$ be represented as 
\begin{equation}
w = \sum_{i=1}^k \alpha_i d_i, \quad\quad d_i \in \oper D,
\end{equation}
where $\alpha_i$ are called coefficients of $w$; then, $w$ is called sparse with respect to $\oper D$ if $k$ is relatively small compared to the ambient dimension $n$. Here, $\oper D$ could be a finite set (e.g. $\oper D = \{ e_i\}_{i=1}^n$ where $e_i$'s are basic vectors in Euclidean space), or $\oper D$ could be infinite (e.g. $\oper D = \{ u_i v_i^* \}_{i=1}^{\infty}$ where $u_i v_i^*$'s are unit-norm rank-one matrices). This notion is general enough to handle many important sparse models such as group sparsity and low rankness (see \cite{CRPW_2012_J}, \cite{NCT_gradMP_2013_J} for some examples).

Our focus in this paper is to develop algorithms for the following optimization:
\begin{equation}
\label{opt::general min}
\min_w  \underbrace{\frac{1}{M}\sum_{i=1}^M f_i (w) }_{F(w)} \quad\text{subject to} \quad \norm{w}_{0,\oper D} \leq k,
\end{equation}
where $f_i(w)$'s, $w \in \R^n$, are smooth functions which can be \textit{non-convex}; $\norm{w}_{0,\oper D}$ is defined as the norm that captures the sparsity level of $w$. In particular, $\norm{w}_{0,\oper D}$ is the smallest number of atoms in $\oper D$ such that $w$ can be represented by them: 
\begin{equation}
\norm{w}_{0,\oper D} = \min_k \{k:  w = \sum_{i \in T} \alpha_i d_i \quad\text{with} \quad |T| = k \}.
\end{equation}
Also in (\ref{opt::general min}), $k$ is a user-defined parameter that controls the sparsity of the model. The formulation (\ref{opt::general min}) arises in many signal processing and machine learning problems, for instance, compressed sensing (e.g. \cite{Donoho_CS_2006_J}, \cite{CRT_CS_2004_J}), Lasso (\cite{Tibshirani_Lasso_1996_J}), sparse logistic regression, and sparse graphical model estimation (e.g. \cite{YL_2006_J}). In the following, we provide some examples to demonstrate the generality of the optimization (\ref{opt::general min}).

1) \textit{Compressed sensing:} The goal is to recover a signal $w^{\star}$ from the set of observations $y_i = \inner{a_i,w^{\star}} + \epsilon_i$ for $i=1,...,m$. Assuming that the unknown signal $w^{\star}$ is sparse, we minimize the following to recover $w^{\star}$:
$$
\min_{w \in \R^n} \frac{1}{m} \sum_{i=1}^m (y_i - \inner{a_i, w})^2  \quad\text{subject to} \quad \norm{w}_0 \leq k.
$$
In this problem, the set $\oper D$ consists of $n$ basic vectors, each of size $n$ in Euclidean space. This problem can be seen as a special case of (\ref{opt::general min}) with $f_i(w) = (y_i - \inner{a_i,w})^2$ and $M = m$. An alternative way to write the above objective function is 
$$
\frac{1}{m} \sum_{i=1}^m (y_i - \inner{a_i, w})^2  = \frac{1}{M} \sum_{j=1}^M \frac{1}{b} \left( \sum_{i=(j-1)b+1}^{jb}  (y_i - \inner{a_i,w})^2 \right),
$$ 
where $M = m/b$. Thus, we can treat each function $f_j(w)$ as $f_j(w) = \frac{1}{b} \sum_{i=(j-1)b+1}^{jb} (y_i - \inner{a_i,w})^2$. In this setting, each $f_j(w)$ accounts for a collection (or {\it block}) of observations of size $b$, rather than only one observation. This setting will be useful later for our proposed stochastic algorithms. 

2) \textit{Matrix recovery:} Given $m$ observations $y_i = \inner{A_i,W^{\star}} + \epsilon_i$ for $i=1,..,m$ where the unknown matrix $W^{\star} \in \R^{d_1 \times d_2}$ is assumed low-rank, we need to recover the original matrix $W^{\star}$. To do so, we perform the following minimization:
$$
\min_{W \in \R^{d_1 \times d_2}} \frac{1}{m} \sum_{i=1}^m (y_i - \inner{A_i, W})^2  \quad\text{subject to} \quad \rank(W) \leq k.
$$
In this problem, the set $\oper D$ consists of infinitely many unit-normed rank-one matrices and the functions $f_i(W) = (y_i - \inner{A_i, W})^2$. We can also write functions in the block form $f_i(W) = \frac{1}{b} \sum (y_i - \inner{A_i, W})^2$ as above.

3) \textit{Covariance matrix estimation:} Let $x$ be a Gaussian random vector of size $n$ with covariance matrix $W^{\star}$. The goal is to estimate $W^{\star}$ from $m$ independent copies $x_1,...,x_m$ of $x$. A useful way to find $W^{\star}$ is via solving the maximum log-likelihood function with respect to a sparse constraint on the precision matrix $\Sigma^{\star} = (W^{\star})^{-1}$. The sparsity of $\Sigma^{\star}$ encourages the independence between entries of $x$. The minimization formula is as follows:
$$
\min_{\Sigma} \frac{1}{m} \sum_{i=1}^m \inner{x_i x_i^T, \Sigma} - \log \det \Sigma \quad\text{subject to} \quad \norm{\Sigma_{\text{off}}}_0 \leq k,
$$
where $\Sigma_{\text{off}}$ is the matrix $\Sigma$ with diagonal elements set to zero. In this problem, $\oper D$ is the finite collection of unit-normed $n \times n$ matrices $\{ e_i e^*_j \}$ and the functions $f_i(\Sigma) = \inner{x_ix_i^T,\Sigma}-\frac{1}{m}\log \det \Sigma$.  

Our paper is organized as follows. In the remainder of Section \ref{sec::Intro} we discuss related work in the literature and highlight our contributions; we also describe notations used throughout the paper and assumptions employed to analyze the algorithms. We present our stochastic algorithms in Sections \ref{sec::StoIHT} and \ref{sec::StoGradMP}, where we theoretically show the linear convergence rate of the algorithms and include a detailed discussion. In Section \ref{sec::inexact StoIHT and StoGradMP}, we explore various extensions of the two proposed algorithms and also provide the theoretical result regarding the convergence rate. We apply our main theoretical results in Section \ref{sec::Estimates}, in the context of sparse linear regression and low-rank matrix recovery. In Section \ref{sec::experiments}, we demonstrate several numerical simulations to validate the efficiency of the proposed methods and compare them with existing deterministic algorithms. Our conclusions are given in Section \ref{sec::conclusion}. We reserve Section \ref{sec::proofs} for our theoretical analysis.

\subsection{Related work and our contribution}

Sparse estimation has a long history and during its development there have been many great ideas along with efficient algorithms to solve (not exactly) the optimization problem (\ref{opt::general min}). We sketch here some main lines which are by no means exhaustive.

\textbf{Convex relaxation.}  Optimization based techniques arose as a natural convex relaxation to the problem of sparse recovery (\ref{opt::general min}).  There is now a massive amount of work in the field of Compressive Sensing and statistics \cite{candes2006compressive,CSwebpage} that demonstrates these methods can accurately recover sparse signals from a small number of noisy linear measurements.  Given noisy measurements $y = Aw^{\star} + e$, one can solve the $\ell_1$-minimization problem
$$
\hat{w} = \argmin_w \|w\|_1 \quad\text{such that}\quad \|Aw - y\|_2 \leq \varepsilon,
$$
where $\varepsilon$ is an upper bound on the noise $\|e\|_2 \leq \varepsilon$.  Cand\`es, Romberg and Tao~\cite{CT_CSdecoding_2005_J,CRT_Stability_2006a_J} prove that under a deterministic condition on the matrix $A$, this method accurately recovers the signal,
\begin{equation}\label{recov}
\|w^\star - \hat{w}\|_2 \leq \varepsilon + \frac{\|w^\star - w^\star_k\|_2}{\sqrt{k}},
\end{equation}
where $w^\star_k$ denotes the $k$ largest entries in magnitude of the signal $w^\star$.  The deterministic condition is called the Restricted Isometry Property (RIP)~\cite{CT_CSdecoding_2005_J} and requires that the matrix $A$ behave nicely on sparse vectors:
$$
(1-\delta)\|x\|_2^2 \leq \|Ax\|_2^2 \leq (1+\delta)\|x\|_2^2 \quad\text{for all $k$-sparse vectors $x$},
$$
for some small enough $\delta < 1$. 

The convex approach is also extended beyond the quadratic objectives. In particular, the convex relaxation of the optimization (\ref{opt::general min}) is the following:
\begin{equation}
\label{opt::convex relaxation}
\min_w F(w) + \lambda \norm{w},
\end{equation}
where the regularization $\norm{w}$ is used to promote sparsity, for instance, it can be the $\ell_1$ norm (vector case) or the nuclear norm (matrix case). Many methods have been developed to solve these problems including interior point methods and other first-order iterative methods such as (proximal) gradient descent and coordinate gradient descent~(e.g. \cite{kim2007interior,FNW07:Gradient-Projection,DDM04:Iterative-Thresholding}). The theoretical analyses of these algorithms have also been studied with either linear or sublinear rate of convergence, depending on the assumption imposed on the function $F(w)$. In particular, sublinear convergence rate is obtained if $F(w)$ exhibits a convex and smooth function, whereas the linear convergence rate is achieved when $F(w)$ is the smooth and strongly convex function. For problems such as compressed sensing, although the loss function $F(w)$ does not possess the strong convexity property, experiments still show the linear convergence behavior of the gradient descent method. In the recent work \cite{ANW_2012_J}, the authors develop theory to explain this behavior. They prove that as long as the function $F(w)$ obeys the restricted strong convexity and restricted smoothness, a property similar to the RIP, then gradient descent algorithm can obtain the linear rate.

\textbf{Greedy pursuits.} More in line with our work are greedy approaches. These algorithms reconstruct the signal by identifying elements of the support iteratively.  Once an accurate support set is located, a simple least-squares problem recovers the signal accurately.  Greedy algorithms like Orthogonal Matching Pursuit (OMP) \cite{Tropp_OMP_2004_J} and Regularized OMP (ROMP) \cite{NV_2010_J} offer a much faster runtime than the convex relaxation approaches but lack comparable strong recovery guarantees. Recent work on greedy methods like Compressive Sampling Matching Pursuit (CoSaMP) and Iterative Hard Thresholding (IHT) offer both the advantage of a fast runtime and essentially the same recovery guarantees as~\eqref{recov}~(e.g. \cite{NT_CoSaMP_2010_J,BD_IHT_2009_J,zhang2011sparse,Fourcard_2012_RIP}). However, these algorithms are only applied for problems in compressed sensing where the least square loss is used to measure the discrepancy. There certainly exists many loss functions that are commonly used in statistical machine learning and do not exhibit quadratic structure such as log-likelihood loss. Therefore, it is necessary to develop efficient algorithms to solve (\ref{opt::general min}).

There are several methods proposed to solve special instances of (\ref{opt::general min}). \cite{SSZ_2010_J} and \cite{SGS_2011_C} propose the forward selection method for sparse vector and low-rank matrix recovery. The method selects each nonzero entry or each rank-one matrix in an iterative fashion. \cite{YY_2012_C} generalizes this algorithm to the more general dictionary $\oper D$. \cite{Zhang_2011_J} proposes the forward-backward method in which an atom can be added or removed from the set, depending on how much it contributes to decrease the loss function. \cite{JJR_2011_C} extends this algorithms beyond the quadratic loss studied in \cite{Zhang_2011_J}. \cite{BRB_GraSP_2013_J} extends the CoSaMP algorithm for a more general loss function. Very recently, \cite{NCT_gradMP_2013_J} further generalizes CoSaMP and proposes the Gradient Matching Pursuit (GradMP) algorithm to solve (\ref{opt::general min}). This is perhaps the first greedy algorithm for (\ref{opt::general min}) - the very general form of sparse recovery. They show that under a restricted convexity assumption of the objective function, the algorithm linearly converges to the optimal solution. This desirable property is also possessed by CoSaMP. We note that there are other algorithms having also been extended to the setting of sparsity in arbitrary $\oper D$ but only limited to the quadratic loss setting, see e.g.~\cite{davenport2012signal,giryes2012rip,giryes2014greedy,giryes2013greedy}.

We outline the GradMP method here, since it will be used as motivation for the work we propose.  GradMP~\cite{NCT_gradMP_2013_J} is a generalization of the CoSaMP~\cite{NT_CoSaMP_2010_J} that solves a wider class of sparse reconstruction problems.  Like OMP, these methods consist of four main steps:
i) form a signal proxy, ii) select a set of large entries of the proxy, iii) use those as the support estimation and estimate the signal via least-squares, and iv) prune the estimation and repeat.  Methods like OMP and CoSaMP use the proxy $A^*(y-Aw^t)$; more general methods like GradMP use the gradient $\nabla F(w^t)$ (see~\cite{NCT_gradMP_2013_J} for details).  The analysis of GradMP depends on the restricted strong convexity and restricted strong smoothness properties as in Definitions~\ref{def:rsc} and~\ref{def:rss} below, the first of which is motivated by a similar property introduced in~\cite{NRWY_2010_C}.  Under these assumptions, the authors prove linear convergence to the noise floor.

The IHT, another algorithm that motivates our work, is a simple method that begins with an estimation $w^0 = 0$ and computes the next estimation using the recursion
$$
w^{t+1} = H_k(w^t + A^*(y - Aw^t)),
$$ 
where $H_k$ is the thresholding operator that sets all but the largest (in magnitude) $k$ coefficients of its argument to zero.  Blumensath and Davies~\cite{BD_IHT_2009_J} prove that under the RIP, IHT provides a recovery bound comparable to~\eqref{recov}. \cite{JMD_SVP_2010_C} extends IHT to the matrix recovery. Very recently, \cite{YLZ_2014_C} proposes the Gradient Hard Thresholding Pursuit (GraHTP), an extension of IHT to solve a special vector case of (\ref{opt::general min}).

\textbf{Stochastic convex optimization.}
Methods for stochastic convex optimization have been developed in a very related but somewhat independent large body of work.  We discuss only a few here which motivated our work, and refer the reader to e.g.~\cite{spall2005introduction,boyd2004convex} for a more complete survey.  Stochastic Gradient Descent (SGD) aims to minimize a convex objective function using unbiased stochastic gradient estimates, typically of the form $\nabla f_i(w)$ where $i$ is chosen stochastically. For the optimization (\ref{opt::general min}) with no constraint, this can be summarized concisely by the update rule
$$
w^{t+1} = w^t - \alpha \nabla f_i(w^t),
$$
for some step size $\alpha$. For smooth objective functions $F(w)$, classical results demonstrate a $1/t$ convergence rate with respect to the objective difference $F(w^t) - F(w^\star)$.  In the strongly convex case, Bach and Moulines \cite{bach2011} improve this convergence to a linear rate, depending on the average squared condition number of the system.  Recently, Needell et. al. draw on connections to the Kaczmarz method (see~\cite{Kac37:Angenaeherte-Aufloesung,SV09:Randomized-Kaczmarz} and references therein), and improve this to a linear dependence on the uniform condition number~\cite{needell2013stochastic}. Another line of work is the Stochastic Coordinate Descent (SCD) beginning with the work of \cite{Nesterov_SCD_2012_J}. Extension to minimization of composite functions in (\ref{opt::convex relaxation}) is described in \cite{RT_BSCD_2011_J}.

\textbf{Contribution.}  
In this paper, we exploit ideas from IHT \cite{BD_IHT_2009_J}, CoSaMP \cite{NT_CoSaMP_2010_J} and GradMP \cite{NCT_gradMP_2013_J} as well as the recent results in stochastic optimization~\cite{SV09:Randomized-Kaczmarz,needell2013stochastic}, and propose two new algorithms to solve (\ref{opt::general min}). The IHT and CoSaMP algorithms have been remarkably popular in the signal processing community due to their simplicity and computational efficiency in recovering sparse signals from incomplete linear measurements. However, these algorithms are mostly used to solve problems in which the objective function is quadratic and it would be beneficial to extend the algorithmic ideas to the more general objective function.   

We propose in this paper stochastic versions of the IHT and GradMP algorithms, which we term Stochastic IHT (StoIHT) and Stochastic GradMP (StoGradMP). These algorithms possess favorable properties toward large scale problems: 
\begin{itemize}
	\item The algorithms do not need to compute the full gradient of $F(w)$. Instead, at each iteration, they only sample one index $i \in [M]=\{1,2,...,M \}$ and compute its associated gradient of $f_i(w)$. This property is particularly efficient in large scale settings in which the gradient computation is often prohibitively expensive.
	\item The algorithms do not need to perform an optimal projection at each iteration as required by the IHT and CoSaMP. Approximated projection is generally sufficient to guarantee linear convergence while the algorithms enjoy significant computational improvement.
	\item Under the restricted strong convexity assumption of $F(w)$ and the restricted strong smoothness assumption of $f_i(w)$ (defined below), the two proposed algorithms are guaranteed to converge linearly to the optimal solution.
	\item The algorithms and proofs can be extended further to consider other variants such as inexact gradient computations and inexact estimation. 
\end{itemize}   

\subsection{Notations and assumptions}

\noindent \textbf{Notation:} For a set $\Omega$, let $|\Omega|$ denote its cardinality and $\Omega^c$ denote its complement. We will write $D$ as the matrix whose columns consist of elements of $\oper D$, and denote $D_\Omega$ as the submatrix obtained by extracting the columns of $D$ corresponding to the indices in $\Omega$. We denote by $\oper R(D_{\Omega})$ the space spanned by columns of the matrix $D_{\Omega}$. Also denote by $\oper P_{\Omega} w$ the orthogonal projection of $w$ onto $\oper R(D_{\Omega})$. Given a vector $w \in \R^n$ that can decomposed as $w = \sum_{i \in \Omega} \alpha_i d_i$, we say that the support of $w$ with respect to $\oper D$ is $\Omega$, denoted by $\supp_{\oper D}(w) = \Omega$. We denote by $[M]$ the set $\{1,2,...,M \}$. We also define $\E_i$ as the expectation with respect to $i$ where $i$ is drawn randomly from the set $[M]$. For a matrix $A$, we use conventional notations: $\norm{A}$ and $\norm{A}_F$ are the spectral norm and Frobenius norms of the matrix $A$. For the linear operator $\oper A: W \in \R^{n_1 \times n_2} \rightarrow y\in \R^m$, $\oper A^*: y \in \R^m \rightarrow W \in \R^{n_1 \times n_2}$ is the transpose of $\oper A$. 

Denote $F(w) \triangleq \frac{1}{M}\sum_{i=1}^M f_i (w)$ and let $p(1), ..., p(M)$ be the probability distribution of an index $i$ selected at random from the set $[M]$. Note that $\sum_{i=1}^M p(i) =1$. Another important observation is that if we select an index $i$ from the set $[M]$ with probability $p(i)$, then
\begin{equation}
\E_i \frac{1}{Mp(i)}  f_i(w) = F(w) \quad\text{and} \quad \E_i \frac{1}{Mp(i)} \nabla f_i(w) = \nabla F(w),
\end{equation}
where the expectation is with respect to the index $i$.

Define $\text{approx}_k (w,\eta)$ as the operator that constructs a set $\Gamma$ of cardinality $k$ such that 
\begin{equation}
\label{eqt::approximation definition}
\norm{\oper P_{\Gamma}w- w}_2 \leq \eta\norm{w - w_k}_2,
\end{equation}
where $w_k$ is the best $k$-sparse approximation of $w$ with respect to the dictionary $\oper D$, that is, $w_k = \argmin_{y \in D_{\Gamma}, |\Gamma| \leq k} \norm{w-y}_2$. Put another way, denote
$$
\Gamma^* = \argmin_{|\Gamma| \leq k} \norm{w - \oper P_{\Gamma} w}_2.
$$
Then, we require that
\begin{equation}
\label{eqt::approximation 1st requirement}
\norm{w - \oper P_{\Gamma} w}_2 \leq \eta \norm{w - \oper P_{\Gamma^*} w}_2.
\end{equation}
An immediate consequence is the following inequality:
\begin{equation}
\label{eqt::approximation consequence}
\norm{w - \oper P_{\Gamma} w}_2 \leq \eta \norm{w - \oper P_{R} w}_2.
\end{equation}
for 
any set of $|R| \leq k$ atoms of $\oper D$. This follows because $\norm{w - \oper P_{\Gamma^*} w}_2 \leq \norm{w - \oper P_{R} w}_2$. In addition, taking the square on both sides of the above inequality and manipulating yields
$$
\norm{\oper P_R w}_2^2 \leq \frac{1}{\eta^2} \norm{\oper P_{\Gamma} w}_2^2 + \frac{\eta^2-1}{\eta^2} \norm{w}^2_2 = \norm{\oper P_{\Gamma} w}_2^2 + \frac{\eta^2-1}{\eta^2} \norm{\oper P_{\Gamma^c} w}^2_2.
$$
Taking the square root gives us an important inequality for our analysis later. For any set of $|R| \leq k$ atoms of $\oper D$,
\begin{equation}
\label{eqt::approximation consequence 2}
\norm{\oper P_R w}_2 \leq \norm{\oper P_{\Gamma} w}_2 + \sqrt{\frac{\eta^2-1}{\eta^2}} \norm{\oper P_{\Gamma^c} w}_2.
\end{equation}

\vspace{3mm}
\noindent \textbf{Assumptions:}
Before describing the two algorithms in the next section, we provide assumptions for the functions $f_i(w)$ as well as $F(w)$. The first assumption requires that $F(w)$ is restricted strongly convex with respect to the set $\oper D$. Although we do not require $F(w)$ to be globally convex, it is necessary that $F(w)$ is convex in certain directions to guarantees the linear convergence of our proposed algorithms. The intuition is that our greedy algorithms only drive along certain directions and seek for the optimal solution. Thus, a global convexity assumption is not necessary.

\begin{define} [$\oper D$-restricted strong convexity ($\oper D$-RSC)]\label{def:rsc}
The function $F(w)$ satisfies the $\oper D$-RSC if there exists a positive constant $\rho^-_{k}$ such that
\begin{equation}
\label{eqt::D-RSC}
F(w') - F(w) - \inner{\nabla F(w), w' - w} \geq \frac{\rho^-_k}{2} \norm{w'-w}_2^2,
\end{equation}
for all vectors $w$ and $w'$ of size $n$ such that $|\supp_{\oper D}(w) \cup \supp_{\oper D}(w')| \leq k$.
\end{define}
We notice that the left-hand side of the above inequality relates to the Hessian matrix of $F(w)$ (provided $F(w)$ is smooth) and the assumption essentially implies the positive definiteness of the $k \times k$ Hessian submatrices. We emphasize that this assumption is much weaker than the strong convexity assumption imposed on the full $n$ dimensional space where the latter assumption implies the positive definiteness of the full Hessian matrix. In fact, when $k=n$, $F(w)$ exhibits a strongly convex function with parameter $\rho^-_{k}$, and when $\rho^-_{k} = 0$, $F(w)$ is a convex function. We also highlight that the $\oper D$-RSC assumption is particularly relevant when studying statistical estimation problems in the high-dimensional setting. In this setting, the number of observations is often much less than the dimension of the model parameter and therefore, the Hessian matrix of the loss function $F(w)$ used to measure the data fidelity is highly ill-posed.

\vspace{2mm}
\noindent In addition, we require that $f_i(w)$ satisfies the so-called $\oper D$-restricted strong smoothness  which is defined as follows: 
\begin{define} [$\oper D$-restricted strong smoothness ($\oper D$-RSS)]\label{def:rss} The function $f_i(w)$ satisfies the $\oper D$-RSS if there exists a positive constant $\rho^+_{k}(i)$ such that
\begin{equation}
\norm{\nabla f_i(w') - \nabla f_i(w)}_2 \leq \rho^+_{k} (i) \norm{w'-w}_2
\end{equation}
for all vectors $w$ and $w'$ of size $n$ such that $|\supp_{\oper D}(w) \cup \supp_{\oper D}(w')| \leq k$.
\end{define}

Variants of these two assumptions have been used to study the convergence of the projected gradient descent algorithm \cite{ANW_2012_J}. In fact, the names restricted strong convexity and restricted strong smoothness are adopted from \cite{ANW_2012_J}.

In this paper, we assume that the functions $f_i(w)$ satisfy $\oper D$-RSS with constants $\rho^+_k(i)$ for all $i = 1,...,M$ and $F(w)$ satisfies $\oper D$-RSC with constant $\rho^-_k$. The following quantities will be used extensively throughout the paper:
\begin{equation}
\label{eqt::parameters}
\alpha_{k} \triangleq \max_i \frac{\rho^+_{k}(i)}{Mp(i)}, \quad \rho^+_{k} \triangleq \max_i \rho^+_{k}(i), \quad \text{and} \quad\overline{\rho}^+_k \triangleq \frac{1}{M} \sum_{i=1}^M \rho^+_k(i).
\end{equation}

\section{Stochastic Iterative Hard Thresholding (StoIHT)}
\label{sec::StoIHT}

In this section, we describe the Stochastic Iterative Hard Thresholding (StoIHT) algorithm to solve (\ref{opt::general min}). The algorithm is provided in Algorithm \ref{alg:StoIHT}. At each iteration, the algorithm performs the following standard steps: 
\begin{itemize}
	\item Select an index $i$ from the set $[M]$ with probability $p(i)$.
	\item Compute the gradient associated with the index just selected and move the solution along the gradient direction.
	\item Project the solution onto the constraint space via the approx operator defined in (\ref{eqt::approximation definition}). 
\end{itemize}
Ideally, we would like to compute the exact projection onto the constraint space or equivalently the best $k$-sparse approximation of $b^t$ with respect to $\oper D$. However, the exact projection is often hard to evaluate or is computationally expensive in many problems. Take an example of the large scale matrix recovery problem, where computing the best matrix approximation would require an intensive Singular Value Decomposition (SVD) which often costs $\oper O(k m n)$, where $m$ and $n$ are the matrix dimensions. On the other hand, recent linear algebraic advances allow computing an approximate SVD in only $\oper O(k^2 \max\{m,n\})$. Thus, approximate projections could have a significant computational gain in each iteration. Of course, the price paid for fast approximate projections is a slower convergence rate.  In Theorem \ref{thm::StoIHT} we will show this trade-off.

\begin{algorithm}[ht]
\caption{StoIHT algorithm} 
\label{alg:StoIHT}
\begin{algorithmic}
\STATE \textbf{input:} $k$, $\gamma$, $\eta$, $p(i)$, and stopping criterion
\STATE \textbf{initialize:} $w^0$ and $t=0$
\REPEAT
\STATE
\begin{tabular}{ll}
\textbf{randomize:} & select an index $i_t$ from $[M]$ with probability $p(i_t)$ \\
\textbf{proxy:} & $b^t = w^t - \frac{\gamma}{Mp(i_t)}  \nabla f_{i_t} (w^t)$ \\
\textbf{identify:} & $\Gamma^t = \text{approx}_k (b^t, \eta)$ \\
\textbf{estimate:} & $w^{t+1} = \oper P_{\Gamma^t} (b^t)$ \\
\ & $t = t+1$	
\end{tabular}
\UNTIL{halting criterion \textit{true}}
\STATE \textbf{output:} $\hat{w} = w^t$
\end{algorithmic}
\end{algorithm}

Denote $w^{\star}$ as a feasible solution of (\ref{opt::general min}). Our main result provides the convergence rate of the StoIHT algorithm via characterizing the $\ell_2$-norm error of $t$-th iterate $w^t$ with respect to $w^{\star}$. We first define some quantities necessary for a precise statement of the theorem. First, we denote the \textit{contraction coefficient}
\begin{equation}
\label{eqt::kappa of StoIHT}
\kappa \triangleq 2 \sqrt{\left(1- \gamma (2 - \gamma \alpha_{3k}) \rho^-_{3k} \right)} + \sqrt{(\eta^2-1)\left(1 + \gamma^2 \alpha_{3k} \overline{\rho}^+_{3k} - 2\gamma \rho^-_{3k}  \right)},
\end{equation} 
where the quantities $\alpha_{3k}$, $\overline{\rho}^+_{3k}$, $\rho^-_{3k}$ and $\eta$ are defined in (\ref{eqt::parameters}), (\ref{eqt::D-RSC}), and  (\ref{eqt::approximation definition}). As will become clear later, the contraction coefficient $\kappa$ controls the algorithm's rate of convergence and is required to be less than unity. This $\kappa$ is intuitively dependent on the characteristics of the objective function (via $\oper D$-RSC and $\oper D$-RSS constants $\rho^+_{3k}$ and $\rho^-_{3k}$), the user-defined step size, the probability distribution, and the approximation error. 
The price paid for allowing a larger approximation error $\eta$ is a slower convergence rate, since $\kappa$ will also become large; however, $\eta$ should not be allowed too large since $\kappa$ must still be less than one.

We also define the \textit{tolerance parameter}
\begin{equation}
\label{eqt::sigma of StoIHT}
\sigma_{w^{\star}} \triangleq \frac{\gamma}{\min_i Mp(i)} \left(2  \E_i \max_{|\Omega|\leq 3k} \norm{\oper P_{\Omega} \nabla f_i(w^{\star})}_2 + \sqrt{\eta^2 - 1} \E_i\norm{\nabla f_i(w^{\star})}_2\right),
\end{equation}
where $i$ is an index selected from $[M]$ with probability $p(i)$.  Of course when $w^\star$ minimizes all components $f_i$, we have $\sigma_{w^{\star}}=0$, and otherwise $\sigma_{w^{\star}}$ measures (a modified version) of the usual noise variance in stochastic optimization. 

\vspace{3mm}
In terms of these two ingredients, we now state our first main result. The proof is deferred to Section \ref{sub::proof of theorem 1}.

\begin{thm}
\label{thm::StoIHT}
Let $w^{\star}$ be a feasible solution of (\ref{opt::general min}) and $w^0$ be the initial solution. At the $(t+1)$-th iteration of Algorithm \ref{alg:StoIHT}, the expectation of the recovery error is bounded by
\begin{equation}
\E \norm{w^{t+1} - w^{\star}}_2 \leq \kappa^{t+1} \norm{w^0 - w^{\star}}_2  + \frac{\sigma_{w^{\star}} }{(1-\kappa)} 
\end{equation}
where $\sigma_{w^\star}$ is defined by~\eqref{eqt::sigma of StoIHT}, $\kappa$ is defined by~\eqref{eqt::kappa of StoIHT} and is assumed to be strictly less than unity, and expectation is taken over all choices of random variables $i_0,...,i_t$.
\end{thm}

The theorem demonstrates a linear convergence for the StoIHT even though the full gradient computation is not available. This is a significant computational advantage in large-scale settings where computing the full gradient often requires performing matrix multiplications with matrix dimensions in the millions. In addition, a stochastic approach may also gain advantages from parallel implementation. We emphasize that the result of Theorem \ref{thm::StoIHT} holds for any feasible solution $w^{\star}$ and the error of the $(t+1)$-th iterate is mainly governed by the second term involving the gradient of $\{ f_i(w^{\star}) \}_{i=1,...,M}$. For certain optimization problems, we expect that the energy of these gradients associated with the global optimum is small. For statistical estimation problems, the gradient of the true model parameter often involves only the statistical noise, which is small. Thus, after a sufficient number of iterations, the error between $w^{t+1}$ and the true statistical parameter is only controlled by the model noise.

The result is significantly simpler when the optimal projection is available at each iteration. That is, the algorithm is always able to find the set $\Gamma^t$ such that $w^{t+1}$ is the best $k$-sparse approximation of $b^t$. In this case, $\eta = 1$ and the contraction coefficient $\kappa$ in (\ref{eqt::kappa of StoIHT}) is simplified to
$$
\kappa =  2 \sqrt{\left(1- \gamma (2 - \gamma \alpha_{3k}) \rho^-_{3k} \right)},
$$
with $\alpha_{3k} = \max_i \frac{\rho^+_{3k}(i)}{Mp(i)}$ and $\sigma_{w^{\star}} = \frac{2\gamma}{\min_i Mp(i)}  \E_i \max_{|\Omega| \leq 3k} \norm{\oper P_{\Omega} \nabla f_i(w^{\star})}_2$. In order for $\kappa < 1$, we need $\rho^-_{3k} \geq \frac{3}{4} \alpha_{3k} = \frac{3}{4} \max_i \frac{\rho^+_{3k}(i)}{Mp(i)}$ and 
$$
\gamma < \frac{1 + \sqrt{1 - \frac{3\alpha_{3k}}{4\rho^-_{3k}}}}{\alpha_{3k}}.
$$

The following corollary provides an interesting particular choice of the parameters in which Theorem \ref{thm::StoIHT} is easier to access.

\begin{cor}
Suppose that $\rho^-_{3k} \geq \frac{3}{4}\rho^+_{3k}$.  
Select $\gamma = \frac{1}{\alpha_{3k}}$,   
$\eta = 1$ and the probability distribution $p(i)=\frac{1}{M}$ for all $i=1,...,M$. Then using the quantities defined by~\eqref{eqt::parameters},  
$$
\E \norm{w^{t+1}-w^{\star}}_2 \leq \kappa^{t+1} \norm{w^0-w^{\star}}_2 + \frac{2\gamma}{(1-\kappa)\min_i Mp(i)} \E_i \max_{|\Omega| \leq 3k}\norm{\oper P_{\Omega} \nabla f_i(w^{\star})}_2,
$$
where $\kappa = 2\sqrt{1-\frac{\rho^-_{3k}}{\rho^+_{3k}} }$.
\end{cor}

When the exact projection is not available, we would want to see how big $\eta$ is such that the StoIHT still allows linear convergence. It is clear from (\ref{eqt::kappa of StoIHT}) that for a given step size $\gamma$, bigger $\eta$ leads to bigger $\kappa$, or slower convergence rate. It is required by the algorithm that $\kappa < 1$. Therefore, $\eta^2$ must at least satisfy 
\begin{equation}
\eta^2 \leq 1 + \frac{1}{1 + \gamma^2 \alpha_{3k} \overline{\rho}^+_{3k} - 2\gamma \rho^-_{3k} }.
\end{equation}
As $\gamma = \frac{1}{\rho^+_{3k}}$ and $p(i) = \frac{1}{M}$, $i=1,...,M$, the bound is simplified to $\eta^2 \leq 1 + \frac{1}{2( 1- \rho^-_{3k})}$. This bound implies that the approximation error in (\ref{eqt::approximation 1st requirement}) should be at most ($1+\epsilon$) away from the exact projection error where $\epsilon \in (0,1)$.

In Algorithm \ref{alg:StoIHT}, the projection tolerance $\eta$ is fixed during the iterations. However, there is a flexibility in changing it every iteration. The advantage of this flexibility is that this parameter can be set small during the first few iterations where the convergence is slow and gradually increased for the later iterations. Denoting the projection tolerance at the $j$-th iteration by $\eta^j$, we define the \textit{contraction coefficient} at the $j$-th iteration:
\begin{equation}
\label{eqt::kappa_t of StoIHT}
\kappa_j \triangleq 2 \sqrt{\left(1- \gamma (2 - \gamma \alpha_{3k}) \rho^-_{3k} \right)} + \sqrt{((\eta^j)^2-1)\left(1 + \gamma^2 \alpha_{3k} \overline{\rho}^+_{3k} - 2\gamma \rho^-_{3k}  \right)},
\end{equation} 
and the \textit{tolerance parameter} $\sigma_{w^{\star}} \triangleq \max_{j \in [t]} \sigma^j_{w^{\star}}$ where
\begin{equation}
\label{eqt::sigma_t of StoIHT}
\sigma^j_{w^{\star}} \triangleq \frac{\gamma}{\min_i Mp(i)} \left(2  \E_i \max_{|\Omega|\leq 3k} \norm{\oper P_{\Omega} \nabla f_i(w^{\star})}_2 + \sqrt{(\eta^j)^2 - 1} \max_i \E_i\norm{\nabla f_i(w^{\star})}_2\right).
\end{equation}

\noindent The following corollary shows the convergence of the StoIHT algorithm in the case where the projection tolerance is allowed to vary at each iteration:
\begin{cor}
At the $(t+1)$-th iteration of Algorithm \ref{alg:StoIHT}, the recovery error is bounded by
\begin{equation}
\E \norm{w^{t+1}-w^{\star}}_2 \leq \norm{w^0-w^{\star}}_2 \prod_{j=0}^{t+1} \kappa_j  + \sigma_{w^{\star}} \sum_{i=0}^t \prod_{j=t-i}^t \kappa_j,
\end{equation}
where $\kappa_j$ is defined by~\eqref{eqt::kappa_t of StoIHT}, and $\sigma_{w^{\star}} = \max_{j \in [t]} \sigma^j_{w^{\star}}$ is defined via~\eqref{eqt::sigma_t of StoIHT}.
\end{cor}

\section{Stochastic Gradient Matching Pursuit (StoGradMP)}
\label{sec::StoGradMP}

CoSaMP \cite{NT_CoSaMP_2010_J} has been a very popular algorithm to recover a sparse signal from its linear measurements. In \cite{NCT_gradMP_2013_J}, the authors generalize the idea of CoSaMP and provide the GradMP algorithm that solves a broader class of sparsity-constrained problems. In this paper, we develop a stochastic version of the GradMP, namely StoGradMP, in which at each iteration only the evaluation of the gradient of a function $f_i$ is required. The StoGradMP algorithm is described in Algorithm \ref{alg:StoGradMP} which consists of following steps at each iteration:
\begin{itemize}
	\item Randomly select an index $i$ with probability $p(i)$.
	\item Compute the gradient of $f_i(w)$ with associated index $i$.
	\item Choose the subspace of dimension at most $2k$ to which the gradient vector is closest, then merge with the estimated subspace from previous iteration.
	\item Solve a sub-optimization problem with the search restricted on this subspace. 
	\item Find the subspace of dimension $k$ which is closest to the solution just found. This is the estimated subspace which is hopefully close to the true subspace.
\end{itemize}

At a high level, StoGradMP can be interpreted as at each iteration, the algorithm looks for a subspace based on the previous estimate and then seeks a new solution via solving a low-dimensional sub-optimization problem. Due to the $\oper D$-RSC assumption, the sub-optimization is convex and thus it can be efficiently solved by many off-the-shelf algorithms. StoGradMP stops when a halting criterion is satisfied.

\begin{algorithm}[ht]
\caption{StoGradMP algorithm}
\label{alg:StoGradMP}
\begin{algorithmic}
\STATE \textbf{input:} $k$, $\eta_1$, $\eta_2$, $p(i)$, and stopping criterion
\STATE \textbf{initialize:} $w^0$, $\Lambda = 0$, and $t=0$
\REPEAT
\STATE
\begin{tabular}{ll}
\textbf{randomize:} & select an index $i_t$ from $[M]$ with probability $p(i_t)$ \\
\textbf{proxy:} & $r^t =  \nabla f_{i_t} (w^t)$ \\
\textbf{identify:} & $\Gamma = \text{approx}_{2k} (r^t, \eta_1)$ \\
\textbf{merge:} & $\widehat{\Gamma} = \Gamma \cup \Lambda  $\\
\textbf{estimate:} & $b^t = \argmin_{w} F(w) \quad w \in \text{span} (D_{\widehat{\Gamma}})$\\
 \textbf{prune:} &$\Lambda = \text{approx}_k (b^t, \eta_2) $ \\
\textbf{update:} & $w^{t+1} = \oper P_{\Lambda} (b^t)$ \\
 & t = t+1	
\end{tabular}
\UNTIL{halting criterion \textit{true}}
\STATE \textbf{output:} $\hat{w} = w^t$
\end{algorithmic}
\end{algorithm}

Denote $w^{\star}$ as a feasible solution of the optimization (\ref{opt::general min}). We will present our main result for the StoGradMP algorithm. As before, our result controls the convergence rate of the recovery error at each iteration. We define the \textit{contraction coefficient}
\begin{equation}
\label{eqt::kappa of StoGradMP}
\kappa \triangleq (1+\eta_2) \sqrt{\frac{\alpha_{4k}}{\rho^-_{4k}}} \left( \max_i \sqrt{Mp(i)} \sqrt{\frac{\frac{2\eta_1^2-1}{\eta^2_1} \rho^+_{4k} - \rho^-_{4k}}{\rho^-_{4k}}}  + \frac{\sqrt{\eta^2_1-1}}{\eta_1} \right),
\end{equation}
where the quantities $\alpha_{4k}$, $\rho^+_{4k}$, $\rho^-_{4k}$, $\eta_1$, and $\eta_2$ are defined in (\ref{eqt::parameters}), (\ref{eqt::D-RSC}), and (\ref{eqt::approximation definition}). As will be provided in the following theorem, $\kappa$ characterizes the convergence rate of the algorithm. This quantity depends on many parameters that play a role in the algorithm. 

In addition, we define analogously as before the \textit{tolerance parameter}
\begin{equation}
\label{eqt::sigma of StoGradMP}
\begin{split}
\sigma_{w^{\star}}  &\triangleq C (1+\eta_2) \frac{1}{\min_{i \in [M]} M p(i)} \max_{|\Omega| \leq 4k, i \in [M]} \norm{\oper P_{\Omega} \nabla f_i (w^{\star})}_2,
\end{split}
\end{equation}
where $C$ is defined as $C \triangleq \frac{1}{\rho^-_{4k}} \left(2 \max_{i \in [M]}Mp(i) \sqrt{\frac{\alpha_{4k}}{\rho^-_{4k}}} + 3\right) $. 

\vspace{3mm}
We are now ready to state our result for the StoGradMP algorithm. The error bound has the same structure as that of StoIHT but with a different convergence rate.
\begin{thm}
\label{thm::StoGradMP}
Let $w^{\star}$ be a feasible solution of (\ref{opt::general min}) and $w^0$ be the initial solution. At the $(t+1)$-th iteration of Algorithm \ref{alg:StoGradMP}, the recovery error is bounded by
\begin{equation}
\E \norm{w^{t+1}-w^{\star}}_2 \leq \kappa^{t+1} \norm{w^0-w^{\star}}_2 + \frac{\sigma_{w^{\star}}}{1-\kappa}
\end{equation}
where $\sigma_{w^{\star}}$ is defined by~\eqref{eqt::sigma of StoGradMP}, $\kappa$ is defined by~\eqref{eqt::kappa of StoGradMP} and is assumed to be strictly less than unity, and expectation is taken over all choices of random variables $i_0,...,i_t$.
\end{thm} 

When $p(i) = \frac{1}{M}$, $i=1,...,M$, and $\eta_1 =\eta_2= 1$ (exact projections are obtained), the contraction coefficient $\kappa$ has a very simple representation: $\kappa = 2\sqrt{\frac{\rho^+_{4k}}{\rho^-_{4k}} \left( \frac{\rho^+_{4k}}{\rho^-_{4k}} - 1 \right)}$. This expression of $\kappa$ is the same as that of the GradMP. In this situation, the requirement $\kappa < 1$ leads to the condition $\rho^+_{4k} < \frac{2+\sqrt{6}}{4} \rho^-_{4k}$. The following corollary provides the explicit form of the recovery error.
\begin{cor}
Using the parameters described by~\eqref{eqt::parameters}, suppose that $\rho^-_{4k} > \frac{4}{2+\sqrt{6}} \rho^+_{4k} $. Select $\eta_1=\eta_2 = 1$, and the probability distribution $p(i) = \frac{1}{M}$, $i=1,...,M$. Then,
$$
\E \norm{w^{t+1} - w^{\star}}_2 \leq \left( 2 \sqrt{\frac{\rho^+_{4k}(\rho^+_{4k} -\rho^-_{4k})}{(\rho^-_{4k})^2}} \right)^{t+1} \norm{w^0 - w^{\star}}_2 + \sigma_{w^{\star}},
$$
where $\sigma_{w^{\star}} = \frac{2}{\rho^-_{4k}} \left(2\sqrt{\frac{\rho^+_{4k}}{\rho^-_{4k}}} + 3 \right) \max_{|\Omega| \leq 4k, i \in [M]} \norm{\oper P_{\Omega} \nabla f_i (w^{\star})}_2$.
\end{cor}

Similar to the StoIHT, the theorem demonstrates the linear convergence of the StoGradMP to the feasible solution $w^{\star}$. The expected recovery error naturally consists of two components: one relates to the convergence rate and the other concerns the tolerance factor. As long as the contraction coefficient is small (less than unity), the first component is negligible, whereas the second component can be very large depending on the feasible solution we measure. We expect that the gradients of $f_i$'s associated with the global optimum to be small, as shown true in many statistical estimation problems such as sparse linear estimation and low-rank matrix recovery, so that the StoGradMP converges linearly to the optimum. We note that the linear rate here is precisely consistent with the linear rate of the original CoSaMP algorithm applied to compressed sensing problems \cite{NT_CoSaMP_2010_J}. Furthermore, StoGradMP gains significant computation over CoSaMP and GradMP since the full gradient evaluation is not required at each iteration.

In Algorithm \ref{alg:StoGradMP}, the parameters $\eta_1$ and $\eta_2$ are fixed during the iterations. However, they can be changed at each iteration. Denoting the projection tolerances at the $j$-th iteration by $\eta_1^j$ and $\eta_2^j$, we define the \textit{contraction coefficient} at the $j$-th iteration as 
\begin{equation}
\label{eqt::kappa_t of StoGradMP}
\kappa_j \triangleq (1+\eta^j_2) \sqrt{\frac{\alpha_{4k}}{\rho^-_{4k}}} \left( \max_i \sqrt{Mp(i)} \sqrt{\frac{\frac{2(\eta_1^j)^2-1}{(\eta^j_1)^2} \rho^+_{4k} - \rho^-_{4k}}{\rho^-_{4k}}}  + \frac{\sqrt{(\eta^j_1)^2-1}}{\eta^j_1} \right).
\end{equation}
Also define the \textit{tolerance parameter} $\sigma_{w^{\star}} \triangleq \max_{j\in[t]} \sigma^j_{w^{\star}}$ where
\begin{equation}
\label{eqt::sigma_t of StoGradMP}
\sigma^j_{w^{\star}}  \triangleq C (1+\eta^j_2) \frac{1}{\min_{i \in [M]} M p(i)} \max_{|\Omega| \leq 4k, i \in [M]} \norm{\oper P_{\Omega} \nabla f_i (w^{\star})}_2
\end{equation}
and $C$ is defined as $C \triangleq 2 \max_{i \in [M]}Mp(i) \sqrt{\frac{\alpha_{4k}}{\rho^-_{4k}}} + 3 $. The following corollary shows the convergence of the algorithm.
\begin{cor}
At the $(t+1)$-th iteration of Algorithm \ref{alg:StoGradMP}, the recovery error is bounded by
\begin{equation}
\E \norm{w^{t+1}-w^{\star}}_2 \leq \norm{w^0-w^{\star}}_2 \prod_{j=0}^{t+1} \kappa_j  + \sigma_{w^{\star}} \sum_{i=0}^t \prod_{j=t-i}^t \kappa_j,
\end{equation}
where $\kappa_j$ is defined by~\eqref{eqt::kappa_t of StoGradMP}, and $\sigma_{w^{\star}} = \max_{j\in[t]} \sigma^j_{w^{\star}}$ is defined via~\eqref{eqt::sigma_t of StoGradMP}.
\end{cor}

\section{StoIHT and StoGradMP with inexact gradients}
\label{sec::inexact StoIHT and StoGradMP}

In this section, we investigate the StoIHT and StoGradMP algorithms in which the gradient might not be exactly estimated. This issue occurs in many practical problems such as distributed network optimization in which gradients are corrupted by noise during the communication on the network. In particular, in both algorithms, the gradient selected at each iteration is contaminated by a noise vector $e^t$ where $t$ indicates the iteration number. We assume $\{e^t\}_{t=1,2,...}$ are deterministic noise with bounded energies. 

\subsection{StoIHT with inexact gradients}

In the StoIHT algorithm, the update $b^t$ at the proxy step has to take into account the noise appearing in the gradient. In particular, at the $t$-th iteration,
$$
b^t = w^t - \frac{\gamma}{M p(i_t)} \left( \nabla f_{i_t} (w^t) + e^t \right).
$$

\noindent Denote the quantity
\begin{equation}
\label{eqt::sigma_e of StoIHT}
\sigma_{e} \triangleq \frac{\gamma}{\min_{i} Mp(i)} \max_{j \in [t]} \left( 2 \max_{|\Omega| \leq 3k} \norm{\oper P_{\Omega} e^j}_2 + \sqrt{\eta^2 - 1} \norm{e^j}_2 \right).
\end{equation}
We state our result in the following theorem. The proof is deferred to Section \ref{subsection::proof of StoIHT with inexact gradients}.
\begin{thm}
\label{thm::StoIHT with inexact gradients}
Let $w^{\star}$ be a feasible solution of (\ref{opt::general min}). At the $(t+1)$-th iteration of Algorithm \ref{alg:StoIHT} with inexact gradients, the expectation of the recovery error is bounded by
\begin{equation}
\E \norm{w^{t+1} - w^{\star}}_2 \leq \kappa^{t+1} \norm{w^0 - w^{\star}}_2  + \frac{1 }{(1-\kappa)} (\sigma_{w^{\star}} + \sigma_e),
\end{equation}
where $\kappa$ is defined in (\ref{eqt::kappa of StoIHT}) and is assumed to be strictly less than unity and expectation is taken over all choices of random variables $i_1,...,i_t$. The quantities $\sigma_{w^{\star}}$ and $\sigma_e$ are defined in (\ref{eqt::sigma of StoIHT}) and (\ref{eqt::sigma_e of StoIHT}), respectively.
\end{thm}

Theorem \ref{thm::StoIHT with inexact gradients} provides the linear convergence of StoIHT even in the setting of an inexact gradient computation. The error bound shares a similar structure as that of the StoIHT with only an additional term related to the gradient noise. An interesting property is that the noise does not accumulate over iterations. Rather, it only depends on the largest noise level.

\subsection{StoGradMP with inexact gradients}

In the StoGradMP algorithm, accounting for noise in the gradient appears in the proxy step; the expression of $r^t$, with an additional noise term, becomes
$$
r^t = \nabla f_{i_t} (w^t) + e^t.
$$

\noindent Denote the quantity
\begin{equation}
\label{eqt::sigma_e of StoGradMP}
\sigma_e \triangleq \frac{\max_i p(i)}{\rho^-_{4k} \min_i p(i)} \max_{j \in [t]} \norm{e^j}_2.
\end{equation}

\noindent We have the following theorem.
\begin{thm}
\label{thm::StoGradMP with inexact gradients}
Let $w^{\star}$ be a feasible solution of (\ref{opt::general min}). At the $(t+1)$-th iteration of Algorithm \ref{alg:StoGradMP} with inexact gradients, the expectation of the recovery error is bounded by
\begin{equation}
\E \norm{w^{t+1} - w^{\star}}_2 \leq \kappa^{t+1} \norm{w^0 - w^{\star}}_2  + \frac{1 }{(1-\kappa)} (\sigma_{w^{\star}} + \sigma_e),
\end{equation}
where $\kappa$ is defined in (\ref{eqt::kappa of StoGradMP}) and is assumed to be strictly less than unity and expectation is taken over all choices of random variables $i_1,...,i_t$. The quantities $\sigma_{w^{\star}}$ and $\sigma_e$ are defined in (\ref{eqt::sigma of StoGradMP}) and (\ref{eqt::sigma_e of StoGradMP}), respectively.
\end{thm}

Similar to the StoIHT, StoGradMP is stable under the contamination of gradient noise. Stability means that the algorithm is still able to obtain the linear convergence rate. The gradient noise only affects the tolerance rate and not the contraction factor. Furthermore, the recovery error only depends on the largest gradient noise level, implying that the noise does not accumulate over iterations.

\subsection{StoGradMP with inexact gradients and approximated estimation}

In this section, we extend the theory of the StoGradMP algorithm further to consider the sub-optimality of optimization at the estimation step. Specifically, we assume that at each iteration, the algorithm only obtains an approximated solution of the sub-optimization. Denote
\begin{equation}
\label{opt::get b^t}
\quad b^t_{\opt} = \argmin_w F(w)  \quad \text{subject to} \quad w \in \text{span} (D_{\widehat{\Gamma}}),
\end{equation}
as the optimal solution of this convex optimization, where $\hat{\Gamma} = \Gamma\cup \Lambda$ may also give rise to an approximation at the identification step.  Write $b^t$ as the approximated solution available at the estimation step.  Then $b^t$ is linked to $b^t_{\opt}$ via the relationship: $\norm{b^t - b^t_{\opt}}_2 \leq \epsilon^t$. This consideration is realistic in two aspects: first, the optimization (\ref{opt::get b^t}) can be too slow to converge to the optimal solution, hence we might want to stop the algorithm after a sufficient number of steps or whenever the solution is close to the optimum; second, even if (\ref{opt::get b^t}) has a closed-form solution as the least-squares problem, it is still beneficial to solve it approximately in order to reduce the computational complexity caused by the pseudo-inverse process (see \cite{DMMS_leastSquare_2011_J} for an example of randomized least-squares approximation). Denoting the quantity

\begin{equation}
\label{eqt::sigma_epsilon}
\sigma_{\epsilon} = \max_{j \in [t]} \epsilon^j,
\end{equation}
we have the following theorem.

\begin{thm}
\label{thm::StoGradMP with inexact gradient and approximated estimation}
Let $w^{\star}$ be a feasible solution of (\ref{opt::general min}). At the $(t+1)$-th iteration of Algorithm \ref{alg:StoGradMP} with inexact gradients and approximated estimations, the expectation of the recovery error is bounded by
\begin{equation}
\E \norm{w^{t+1} - w^{\star}}_2 \leq \kappa^{t+1} \norm{w^0 - w^{\star}}_2  + \frac{1 }{(1-\kappa)} (\sigma_{w^{\star}} + \sigma_e + \sigma_{\epsilon}),
\end{equation}
where $\kappa$ is defined in (\ref{eqt::kappa of StoGradMP}) and is assumed to be strictly less than unity and expectation is taken over all choices of random variables $i_1,...,i_t$. The quantities $\sigma_{w^{\star}}$, $\sigma_e$, and $\sigma_{\epsilon}$ are defined in (\ref{eqt::sigma of StoGradMP}), (\ref{eqt::sigma_e of StoGradMP}), and (\ref{eqt::sigma_epsilon}), respectively.
\end{thm}

Theorem \ref{thm::StoGradMP with inexact gradient and approximated estimation} shows the stability of StoGradMP under both the contamination of gradient noise at the proxy step and the approximate optimization at the estimation step. Furthermore, StoGradMP still achieves a linear convergence rate even in the presence of these two sources of noise. Similar to the artifacts of gradient noise, the approximated estimation affects the tolerance rate and not the contraction factor, and the recovery is only impacted by the largest approximated estimation bound (rather than an accumulation over all of the iterations). 

\section{Some estimates}
\label{sec::Estimates}

In this section we investigate some specific problems which require solving an optimization with a sparse constraint and transfer results of Theorems \ref{thm::StoIHT} and \ref{thm::StoGradMP}. 

\subsection{Sparse linear regression}
\label{subsec::Sparse linear regression}

The first problem of interest is the well-studied sparse recovery in which the goal is to recover a $k_0$-sparse vector $w_0$ from noisy observations of the following form:
$$
y = A w_0 + \xi.
$$
Here, the $m \times n$ matrix $A$ is called the design matrix and $\xi$ is the $m$ dimensional vector noise. A natural way to recover $w_0$ from the observation vector $y$ is via solving
\begin{equation}
\label{opt::linear regression}
\min_{w \in \R^n}  \frac{1}{2m}\norm{y - Aw}^2_2 \quad\text{subject to} \quad \norm{w}_0 \leq k,
\end{equation}
where $k$ is the user-defined parameter which is assumed greater than $k_0$. Clearly, this optimization is a special form of (\ref{opt::general min}) 
with $\oper D$ being the collection of standard vectors in $\R^n$ and $F(w) = \frac{1}{2m}\norm{y - Aw}^2_2$. Decompose the vector $y$ into non-overlapping vectors $y_{b_i}$ of size $b$ and denote $A_{b_i}$ as the $b_i \times n$ submatrix of $A$. We can then rewrite $F(w)$ as
$$
F(w) = \frac{1}{M} \sum_{i=1}^M \frac{1}{2b} \norm{y_{b_i} - A_{b_i} w}_2^2 \triangleq \frac{1}{M} \sum_{i=1}^M f_i(w),
$$
where $M = m/b$. In order to apply Theorems \ref{thm::StoIHT} and \ref{thm::StoGradMP} for this problem, we need to compute the contraction coefficient and tolerance parameter which involve the $\oper D$-RSC and $\oper D$-RSS conditions. It is easy to see that these two properties of $F(w)$ and $\{ f_i(w) \}_{i=1}^M$ are equivalent to the RIP studied in \cite{CRT_Stability_2006a_J}. In particular, we require that the matrix $A$ satisfies
$$
\frac{1}{m}\norm{A w}_2^2 \geq (1-\delta_k) \norm{w}_2^2 
$$
for all $k$-sparse vectors $w$. In addition, the matrices $A_{b_i}$, $i=1,...,M$, are also required to obey
$$
\frac{1}{b}\norm{A_{b_i} w}_2^2 \leq (1+\delta_k) \norm{w}_2^2
$$
for all $k$-sparse vectors $w$. Here, $(1+\delta_k)$ and $(1-\delta_k)$ with $\delta_k \in (0,1]$ play the role of $\rho^+_k(i)$ and $\rho^-_k$ in Definitions \ref{def:rsc} and \ref{def:rss}, respectively. For the Gaussian matrix $A$ (entries are i.i.d. $\oper N(0, 1)$), it is well-known that these two assumptions hold as long as $m \geq \frac{Ck \log n}{\delta_k}$ and $b \geq \frac{c k \log n}{\delta_k}$. By setting the block size $b = c k \log n$, the number of blocks $M$ is thus proportional to $\frac{m}{k \log n}$. 

Now using StoIHT to solve (\ref{opt::linear regression}) and applying Theorem \ref{thm::StoIHT}, we set the step size $\gamma = 1$, the approximation error $\eta=1$, and $p(i)=1/M$, $i = 1,...,M$, for simplicity. Thus, the quantities in (\ref{eqt::parameters}) are all the same and equal to $1+\delta_k$. It is easy to verify that the contraction coefficient defined in (\ref{eqt::kappa of StoIHT}) is $\kappa = 2\sqrt{2\delta_{3k} - \delta_{3k}^2}$. One can obtain $\kappa \leq 3/4$ when $\delta_{3k} \leq 0.07$, for example. In addition, since $w_0$ is the feasible solution of (\ref{opt::linear regression}), the tolerance parameter $\sigma_{w_0}$ defined in (\ref{eqt::sigma of StoIHT}) can be rewritten as 
$$
\sigma_{w_0} = 2 \E_i \max_{|\Omega| \leq 3k} \frac{1}{b} \norm{\oper P_{\Omega} A_{b_i}^* \xi_{b_i}}_2 \leq \frac{2}{b} \sqrt{3k} \max_{i \in [M]} \max_{j \in [n]} |\inner{A_{b_i,j},\xi_{b_i}}|,
$$ 
where $A_{b_i,j}$ is the $j$-th column of the matrix $A_{b_i}$. For stochastic noise $\xi \sim \oper N(0,\sigma^2 I_m)$, it is easy to verify that $\sigma_{w_0} \leq c' \sqrt{\frac{\sigma^2 k \log n}{b}}$ with probability at least $1-n^{-1}$. 

Using StoGradMP to solve (\ref{opt::linear regression}) with the same setting as above, we write the contraction coefficient in (\ref{eqt::kappa of StoGradMP}) as $\kappa = 2\sqrt{\frac{2\delta_{4k} (1+\delta_{4k})}{(1-\delta_{4k})^2}}$, which is less than $3/4$ if $\delta_{4k} \leq 0.05$. The tolerance parameter $\sigma_{w_0}$ in (\ref{eqt::sigma of StoGradMP}) can be simplified similarly as in StoIHT. We now provide the following corollary based on what we have discussed.
\begin{cor} [for StoIHT and StoGradMP]
\label{cor::apply to sparse linear regression} 
Assume $A \in \R^{m\times n}$ satisfies the $\oper D$-RSC and $\oper D$-RSS assumptions and $\xi \sim \oper N(0,\sigma^2)$. Then with probability at least $1-n^{-1}$, the error at the $(t+1)$-th iterate of the StoIHT and StoGradMP algorithms is bounded by
$$
\E \norm{w^{t+1} - w_0}_2 \leq (3/4)^{t+1} \norm{w_0}_2 + c \sqrt{\frac{\sigma^2 k_0 \log n}{b}}.
$$
\end{cor}
We describe the convergence result of the two algorithms in one corollary since their results share the same form with the only difference in the constant $c$. One can see that for a sufficient number of iterations, the first term involving $\norm{w_0}_2$ is negligible and the recovery error only depends on the second term. When the noise is absent, both algorithms guarantee recovery of the exact $w_0$. The recovery error also depends on the block size $b$. When $b$ is small, more error is expected, and the error decreases as $b$ increases. This of course matches our intuition. We emphasize that the deterministic IHT and GradMP algorithms deliver the same recovery error with $b$ replaced by $m$.

\subsection{Low-rank matrix recovery}
\label{subsec::Low-rank matrix recovery}

We consider the high-dimensional matrix recovery problem in which the observation model has the form 
$$
y_j = \inner{A_j, W_0} + \xi_j  \quad, \quad j = 1,...,m,
$$
where $W_0$ is the $n_1 \times n_2$ unknown rank-$k_0$ matrix, each measurement $A_j$ is an $m\times n_1$ matrix, and the noise $\xi_j$ is assumed $N(0,\sigma^2)$. Noting that $\frac{1}{2m} \sum_{j=1}^m (y_j - \inner{A_j,W})^2 = \frac{1}{2m}\norm{y - \oper A(W)}_2^2$, the standard approach to recover $W_0$ is to solve the minimization 
\begin{equation}
\label{opt::matrix recovery}
\min_{W \in \R^{n_1\times n_2}} \frac{1}{2m}\norm{y - \oper A(W)}_2^2\quad\text{subject to} \quad\rank(W) \leq k,
\end{equation}
with $k$ assumed greater than $k_0$. Here, $\oper A$ is the linear operator. In this problem, the set $\oper D$ consists of infinitely many unit-normed rank-one matrices 
and the objective function can be written as a summation of sub-functions:
$$
F(W) = \frac{1}{M} \sum_{i=1}^M f_i(W) = \frac{1}{M} \sum_{i=1}^M \left( \frac{1}{2b} \sum_{j=(i-1)b+1}^{ib} (y_j - \inner{A_j,W})^2 \right) \triangleq \frac{1}{M} \sum_{i=1}^M \frac{1}{2b} \norm{y_{b_i} - \oper A_i (W)}_2^2,
$$
where $m = Mb$ (assume $b$ is integer). Each $f_i(W)$ accounts for a collection (or block) of observations $y_{b_i}$ of size $b$. In this case, the $\oper D$-RSC and $\oper D$-RSS properties are equivalent to the matrix-RIP \cite{CP_MatrixRec_2009_J}, which holds for a wide class of random operators $\oper A$. In particular, we require
$$
\frac{1}{m} \norm{\oper A(W)}_2^2 \geq (1-\delta_k) \norm{W}_F^2
$$ 
for all rank-$k$ matrices $W$. In addition, the linear operators $\oper A_{i}$ are required to obey
$$
\frac{1}{b} \norm{\oper A_i(W)}_2^2 \leq (1+\delta_k) \norm{W}_F^2
$$
for all rank-$k$ matrices $W$. Here, $(1+\delta_k)$ and $(1-\delta_k)$ with $\delta_k \in (0,1]$ play the role of $\rho^+_k(i)$ and $\rho^-_k$ in Definitions \ref{def:rsc} and \ref{def:rss}, respectively. For the random Gaussian linear operator $\oper A$ (vectors $A_i$ are i.i.d. $\oper N(0, I)$), it is well-known that these two assumptions hold as long as $m \geq \frac{Ck (n_1+n_2)}{\delta_k}$ and $b \geq \frac{c k (n_1+n_2)}{\delta_k}$. By setting the block size $b = c k (n_1+n_2)$, the number of blocks $M$ is thus proportional to $\frac{m}{k (n_1 + n_2)}$.

In this section, we consider applying results of Theorems \ref{thm::StoIHT with inexact gradients} and \ref{thm::StoGradMP with inexact gradient and approximated estimation}. To do so, we need to compute the contraction coefficients and tolerance parameter. We begin with Theorem \ref{thm::StoIHT with inexact gradients}, which holds for the StoIHT algorithm. Similar to the previous section, $\kappa$ in (\ref{eqt::kappa of StoIHT}) can have a similar form. However, in the matrix recovery problem, SVD computations are required at each iteration which is often computationally expensive. There has been a vast amount of research focusing on approximation methods that perform nearly as good as exact SVD but with much faster computation. Among them are the randomized SVD \cite{halko2011finding} that we will employ in the experimental section. For simplicity, we set the step size $\gamma = 1$ and $p(i) = 1/M$ for all $i$. Thus, quantities in (\ref{eqt::parameters}) are the same and equal to $1+\delta_k$. Rewriting $\kappa$ in (\ref{eqt::kappa of StoIHT}), we have
$$
\kappa = 2 \sqrt{2\delta_{3k} - \delta^2_{3k}} + \sqrt{(\eta^2-1)(\delta_{3k}^2 + 4\delta_{3k} )},
$$
where we recall $\eta$ is the projection error. Setting $\kappa \leq 3/4$ by allowing the first term to be less than $1/2$ and the second term less than $1/4$, we obtain $\delta_{4k} \leq 0.03$ and the approximation error $\eta$ is allowed up to $1.19$.

The next step is to evaluate the tolerance parameter $\sigma_{W_0}$ in (\ref{eqt::sigma of StoIHT}).  The parameter $\sigma_{W_0}$ can be read as
\begin{equation}
\begin{split}
\nonumber
\sigma_{W_0} &= 2 \E_i \max_{|\Omega| \leq 4k} \frac{1}{b} \norm{\oper P_{\Omega} \oper A_i^* (\xi_{b_i})}_F + \sqrt{\eta^2-1} \E_i \frac{1}{b} \norm{\oper A_i^* (\xi_{b_i})}_F \\
&\leq  \frac{2}{b} \sqrt{4k} \max_i \norm{\oper A_i^* (\xi_{b_i})} + \frac{1}{b}\sqrt{\eta^2-1} \sqrt{n} \max_i \norm{\oper A_i^* (\xi_{b_i})}.
\end{split}
\end{equation}
For stochastic noise $\xi \sim \oper N(0,\sigma^2 I)$, it is shown in \cite{CP_MatrixRec_2009_J}, Lemma 1.1 that $\norm{\oper A_i^* (\xi_{b_i})} \leq c \sqrt{\sigma^2 n b}$ with probability at least $1-n^{-1}$ where $n = \max\{n_1,n_2\}$. Therefore, $\sigma_{W_0} \leq c \left( \sqrt{\frac{\sigma^2 k n}{b}} + \sqrt{\frac{(\eta^2-1)\sigma^2 n^2}{b}}\right)$. In addition, the parameter $\sigma_e$ in (\ref{eqt::sigma_e of StoIHT}) is estimated as
$$
\sigma_e \leq \max_j \left(2\max_{|\Omega| \leq 3k} \norm{\oper P_{\Omega} E^j}_F + \sqrt{\eta^2-1} \norm{E^j}_F\right) \leq \max_j \left(2\sqrt{3k} \norm{E^j} + \sqrt{\eta^2-1} \norm{E^j}_F\right),
$$
where we recall that $E^j$ is the noise matrix that might contaminate the gradient at the $j$-th iteration. Applying Theorem \ref{thm::StoIHT with inexact gradients} leads to the following corollary.
  
\begin{cor} [for StoIHT]
Assume the linear operator $\oper A$ satisfies the $\oper D$-RSC and $\oper D$-RSS assumptions and $\xi \sim \oper N(0,\sigma^2)$. Set $p(i)=1/M$ for $i = 1,...,M$ and $\gamma = 1$. Then with probability at least $1-n^{-1}$, the error at the $(t+1)$-th iterate of the StoIHT algorithm is bounded by
\begin{equation}
\begin{split}
\nonumber
\E \norm{W^{t+1} - W_0}_F \leq (3/4)^{t+1} \norm{W_0}_F &+ c \left( \sqrt{\frac{\sigma^2 k n}{b}} + \sqrt{\frac{(\eta^2-1)\sigma^2 n^2}{b}}\right) \\
&+ 4\max_{j \in [t]} \left( 2\sqrt{3k} \norm{E^j} + \sqrt{\eta^2-1} \norm{E^j}_F \right).
\end{split}
\end{equation}
\end{cor}  
The discussion after Corollary \ref{cor::apply to sparse linear regression} for the vector recovery can also be applied here. For a sufficient number of iterations, the recovery error is naturally controlled by three factors: the measurement noise $\sigma$, the approximation projection parameter $\eta$, and the largest gradient noise $E^j$. In the absence of these three parameters, the recovery is exact. When $\eta=1$ and $E^j = 0$, $j=0,...,t$, the error has the same structure as the convex nuclear norm minimization method \cite{CP_MatrixRec_2009_J} which has been shown to be optimal. 

Moving to the StoGradMP algorithm, setting $p(i) = 1/M$ for $i=1,...,M$ again for simplicity, we can write the contraction coefficient in (\ref{eqt::kappa of StoGradMP}) as
$$
\kappa = (1+\eta_2) \sqrt{\frac{1+\delta_{4k}}{1-\delta_{4k}}} \left( \sqrt{\frac{\frac{2\eta_1^2-1}{\eta^2_1} (1+\delta_{4k}) - (1-\delta_{4k})}{1-\delta_{4k}}}  + \frac{\sqrt{\eta^2_1-1}}{\eta_1} \right).
$$
If we allow for example the projection error $\eta_1 = 1.01$ and $\eta_2 = 1.01$ and require $\kappa \leq 0.9$, simple algebra gives us $\delta_{4k} \leq 0.03$ 
In addition, for stochastic noise $\xi \sim \oper N(0,\sigma^2 I)$, the tolerance parameter $\sigma_{W_0}$ in (\ref{eqt::sigma of StoGradMP}) can be read as 
$$
\sigma_{W_0} = c (1+\eta_2) \max_{i \in [M], |\Omega| \leq 4k}\norm{\oper P_{\Omega} \oper A_i^* (\xi_{b_i})}_F \leq c_1  (1+\eta_2) \sqrt{4k} \max_{i \in [M]} \norm{\oper A_i^* (\xi_{b_i})} \leq c_2  (1+\eta_2) \sqrt{\frac{\sigma k n}{b}}
$$
with probability at least $1-n^{-1}$. Again, the last inequality is due to \cite{CP_MatrixRec_2009_J}. Now applying Theorem \ref{thm::StoGradMP with inexact gradient and approximated estimation}, where we recall the parameter $\sigma_e$ in (\ref{eqt::sigma_e of StoGradMP}) is $\sigma_e = \max_j \norm{E^j}_F$ and $\epsilon^j$ is the optimization error at the estimation step, we have the following corollary. 

\begin{cor} [for StoGradMP]
Assume the linear operator $\oper A$ satisfies the $\oper D$-RSC and $\oper D$-RSS assumptions and $\xi \sim \oper N(0,\sigma^2)$. Set $p(i)=1/M$ for $i=1,...,M$. Then with probability at least $1-n^{-1}$, the error at the $(t+1)$-th iterate of the StoGradMP algorithm is bounded by
\begin{equation}
\begin{split}
\nonumber
\E \norm{W^{t+1} - W_0}_F \leq (0.9)^{t+1} \norm{W_0}_F &+ c (1+\eta_2) \sqrt{\frac{\sigma^2 k n}{b}} + 4\max_{j \in [t]} \norm{E^j}_F + 4 \max_{j \in [t]}\epsilon^j.
\end{split}
\end{equation}
\end{cor}

\section{Numerical experiments}
\label{sec::experiments}

In this section we present some experimental results comparing our proposed stochastic methods to their deterministic counterparts.  Our goal is to explore several interesting aspects of improvements and trade-offs; we have not attempted to optimize algorithm parameters.  Unless otherwise specified, all experiments are run with at least $50$ trials, ``exact recovery'' is obtained when the signal recovery error $\|w-\hat{w}\|_2$ drops below $10^{-6}$, and the plots illustrate the $10\%$ trimmed mean. For the approximation error versus epoch (or iteration) plots, the trimmed mean is calculated at each epoch (or iteration) value by excluding the highest and lowest  $5\%$ of the error values (rounding when necessary). For the approximation error versus CPU time plots, the trimmed mean computation is the same, except the CPU time values corresponding to the excluded error values are also excluded from the mean CPU time. We begin with experiments in the compressed sensing setting, and follow with application to the low-rank matrix recovery problem.

\subsection{Sparse vector recovery}

The first setting we explored is standard compressed sensing which has been studied in Subsection \ref{subsec::Sparse linear regression}. Unless otherwise specified, the vector has dimension $256$, and its non-zero entries are i.i.d. standard Gaussian. The signal is measured with an $m \times 256$ i.i.d. standard Gaussian measurement matrix.  First, we compare signal recovery as a function of the number of measurements used, for various sparsity levels $s$.  Each algorithm terminates upon convergence or upon a maximum of $500$ epochs\footnote{We refer to an epoch as the number of iterations needed to use $m$ rows.  Thus for deterministic non-blocked methods, an epoch is one iteration, whereas for a method using blocks of size $b$, an epoch is $m/b$ iterations.}.  We used a block size of $b = \min(s, m)$ for sparsity level $k_0$ and number of measurements $m$, except when $k_0=4$ and $m>5$ we used $b=8$ in order to obtain consistent convergence. Specifically, this means that $b$ measurements were used to form the signal proxy at each iteration of the stochastic algorithms.  For the IHT and StoIHT algorithms, we use a step size of $\gamma =1$.  The results for IHT and StoIHT are shown in Figure ~\ref{fig1} and GradMP and StoGradMP in Figure~\ref{fig2}.  Here we see that with these parameters, StoIHT requires far fewer measurements than IHT to recover the signal, whereas GradMP and StoGradMP are comparable.

\begin{figure}[ht]
\begin{tabular}{cc}
\includegraphics[height=2.25in,width=2.9in]{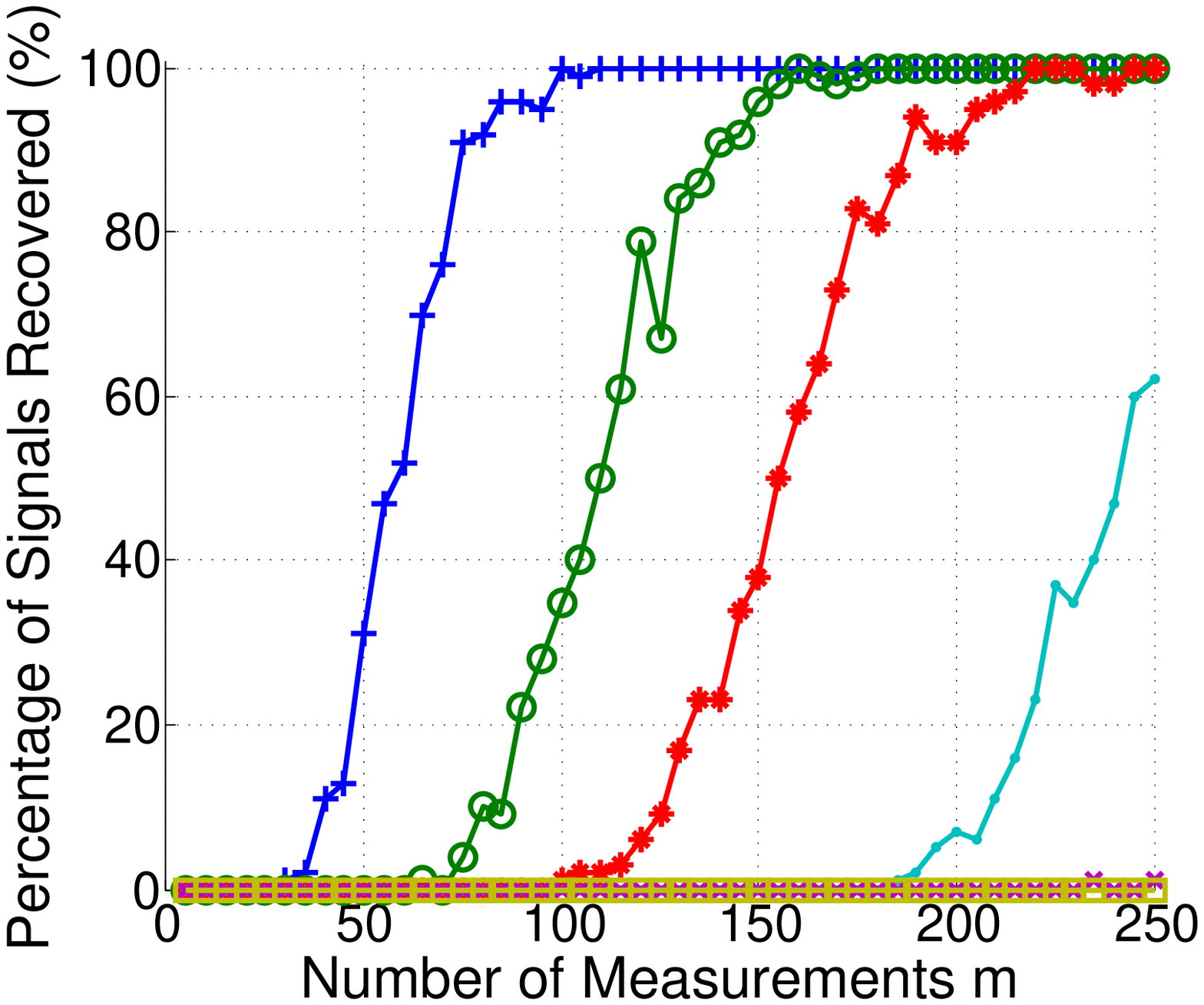} & 
\includegraphics[height=2.25in,width=3.22in]{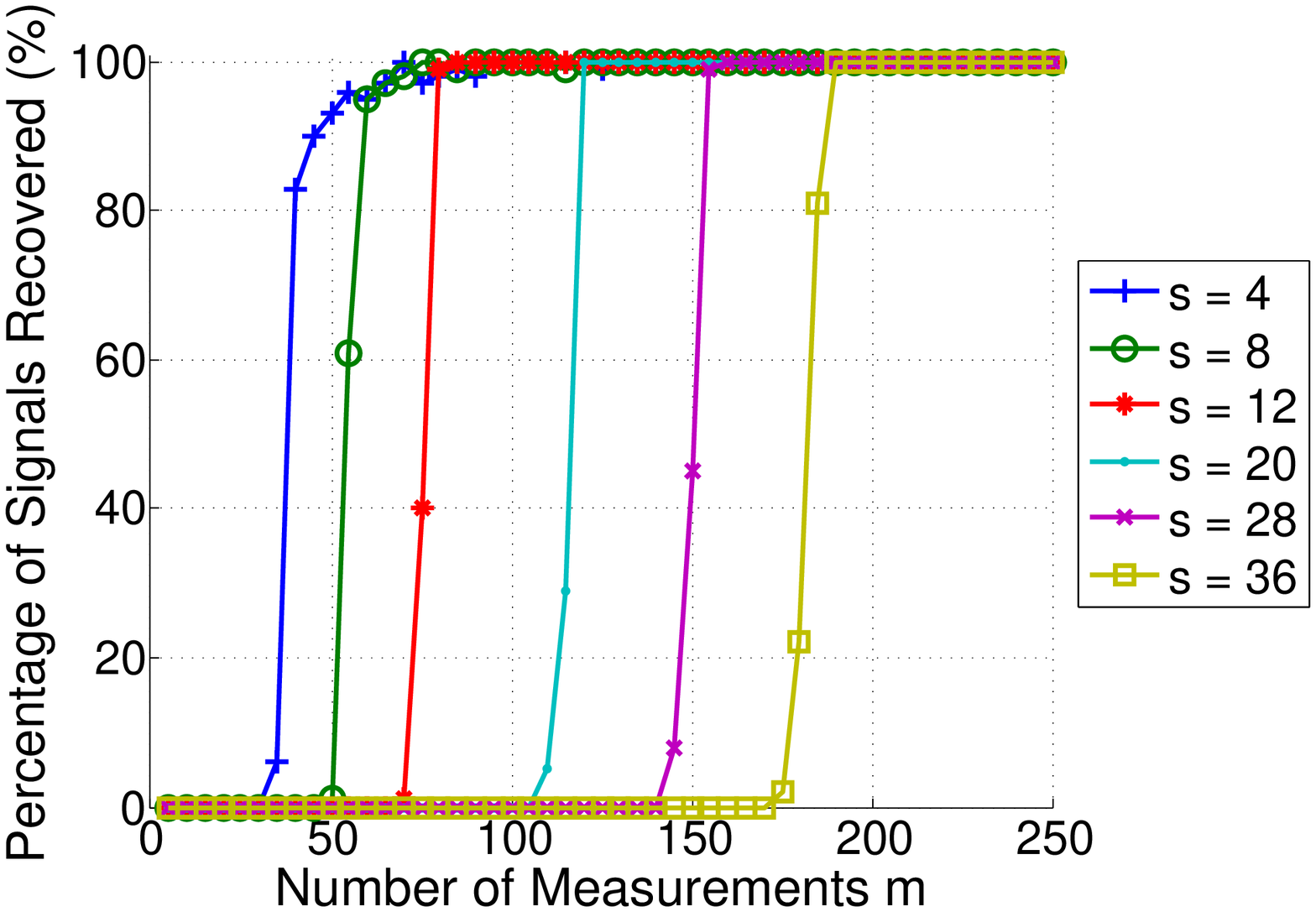}
\end{tabular}
\caption{Sparse Vector Recovery: Percent recovery as a function of the number of measurements for IHT (left) and StoIHT (right) for various sparsity levels $k_0$.\label{fig1}}
\end{figure}

\begin{figure}[ht]
\begin{tabular}{cc}
\includegraphics[height=2.25in,width=2.9in]{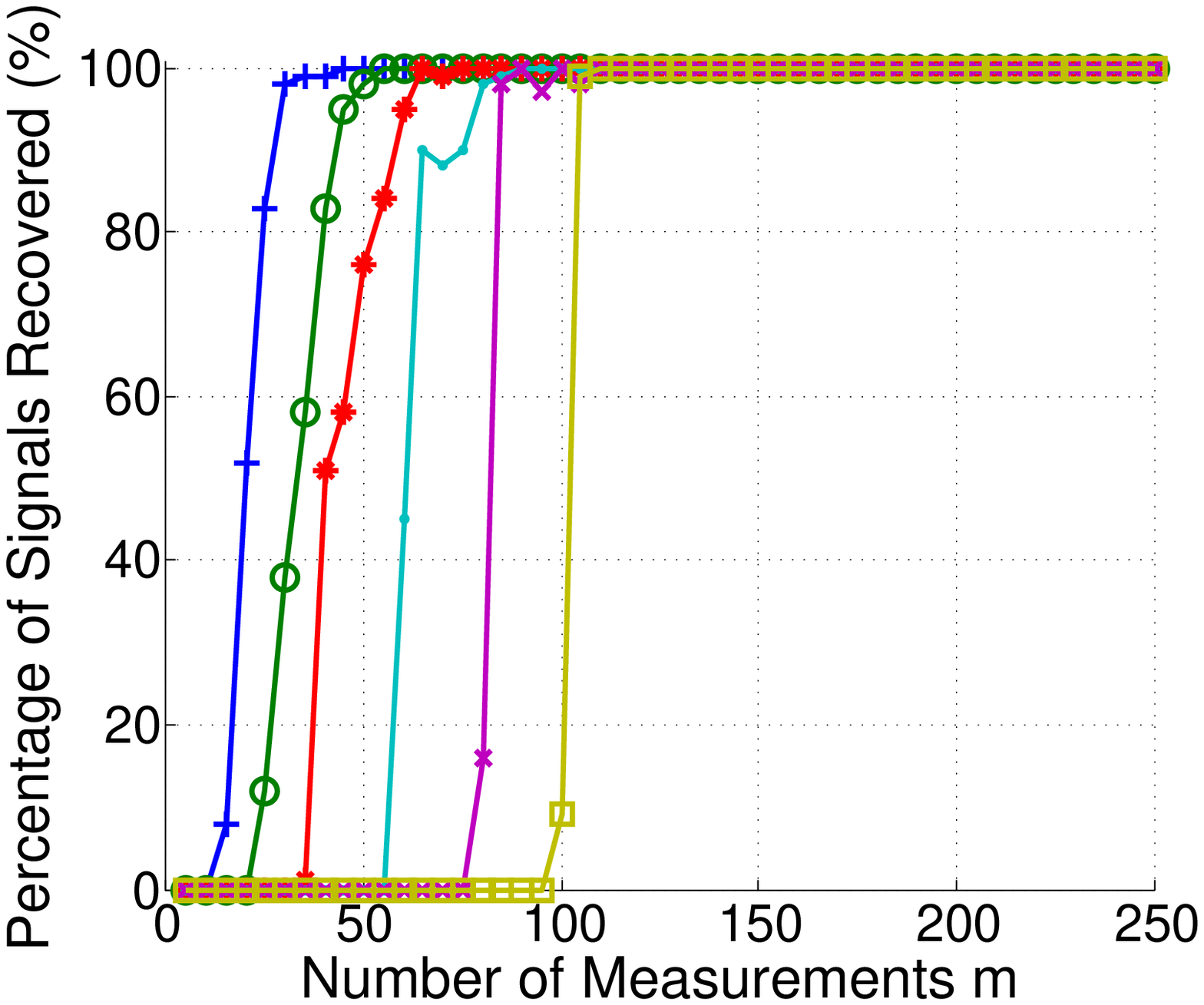} & \includegraphics[height=2.25in,width=3.22in]{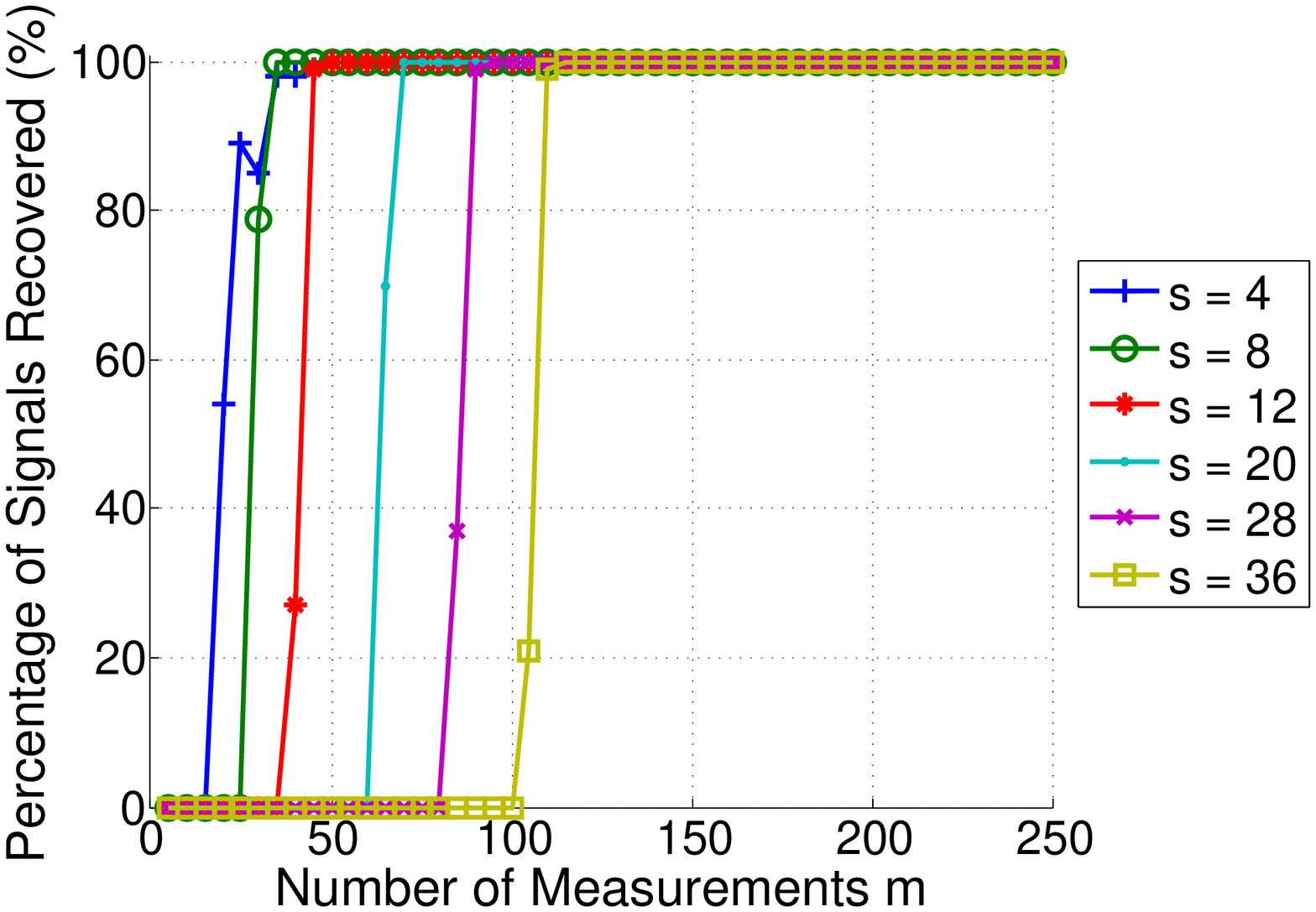}
\end{tabular}
\caption{Sparse Vector Recovery: Percent recovery as a function of the number of measurements for GradMP (left) and StoGradMP (right) for various sparsity levels $k_0$.\label{fig2}}
\end{figure}

Next we explore how the choice of block size affects performance.  We employ the same setup as described above, only now we fix the number of measurements $m$ ($m=180$ for the IHT methods and $m=80$ for the GradMP methods), allow a maximum of 100 epochs, and use various block sizes in the stochastic algorithms.  The sparsity of the signal is $k_0=8$.  The results are depicted for both methods in Figure~\ref{fig3}.  Here we see that in both cases, the deterministic methods seem to offer intermediate performance, outperforming some block sizes and underperforming others.  It is interesting that the StoIHT method seems to prefer larger block sizes whereas the StoGradMP seems to prefer smaller sizes.  This is likely because StoGradMP, even using only a few gradients may still estimate the support accurately, and thus the signal accurately.

\begin{figure}[ht]
\begin{tabular}{cc}
\includegraphics[height=2.25in,width=3.05in]{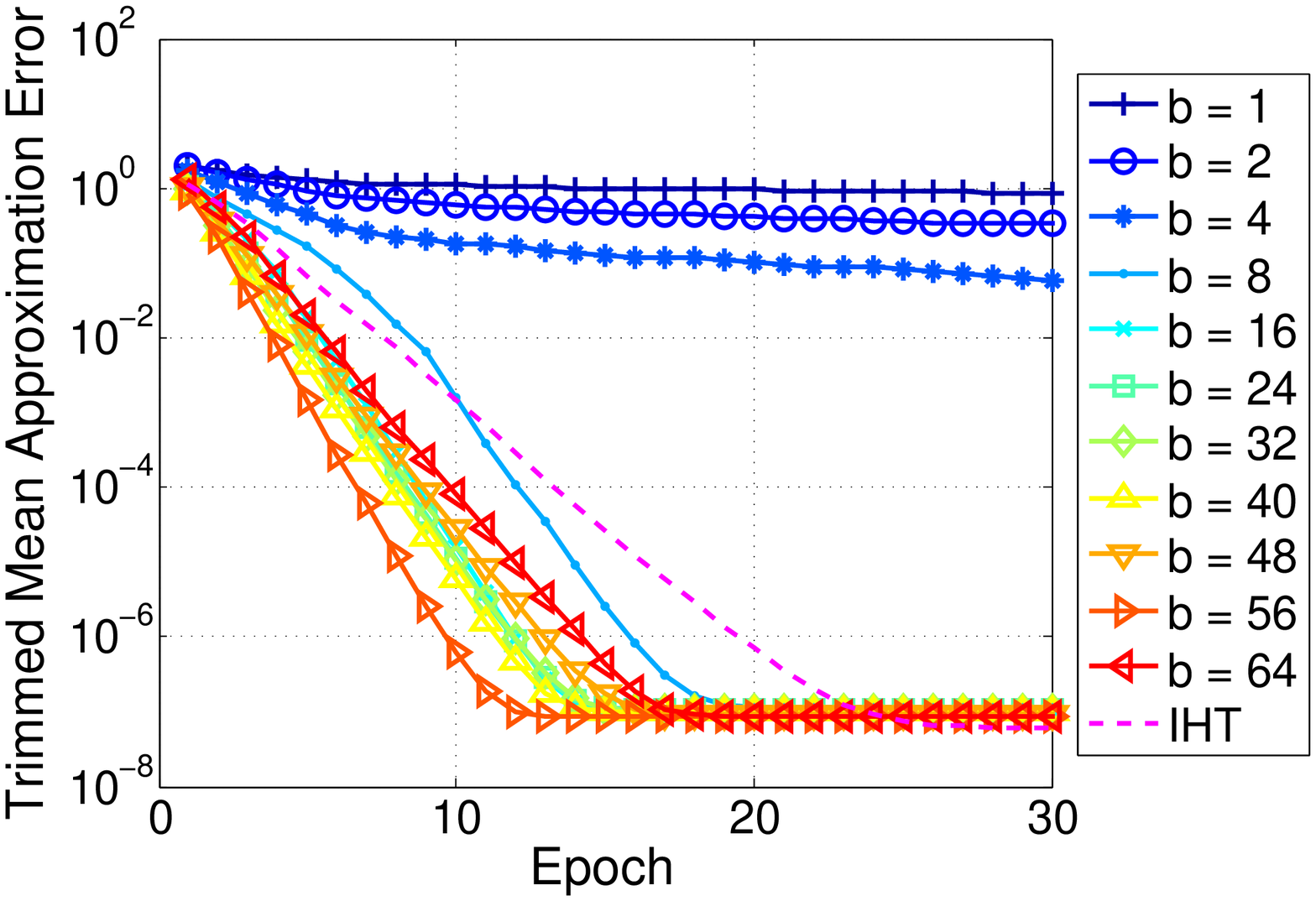} & 
\includegraphics[height=2.25in,width=3.05in]{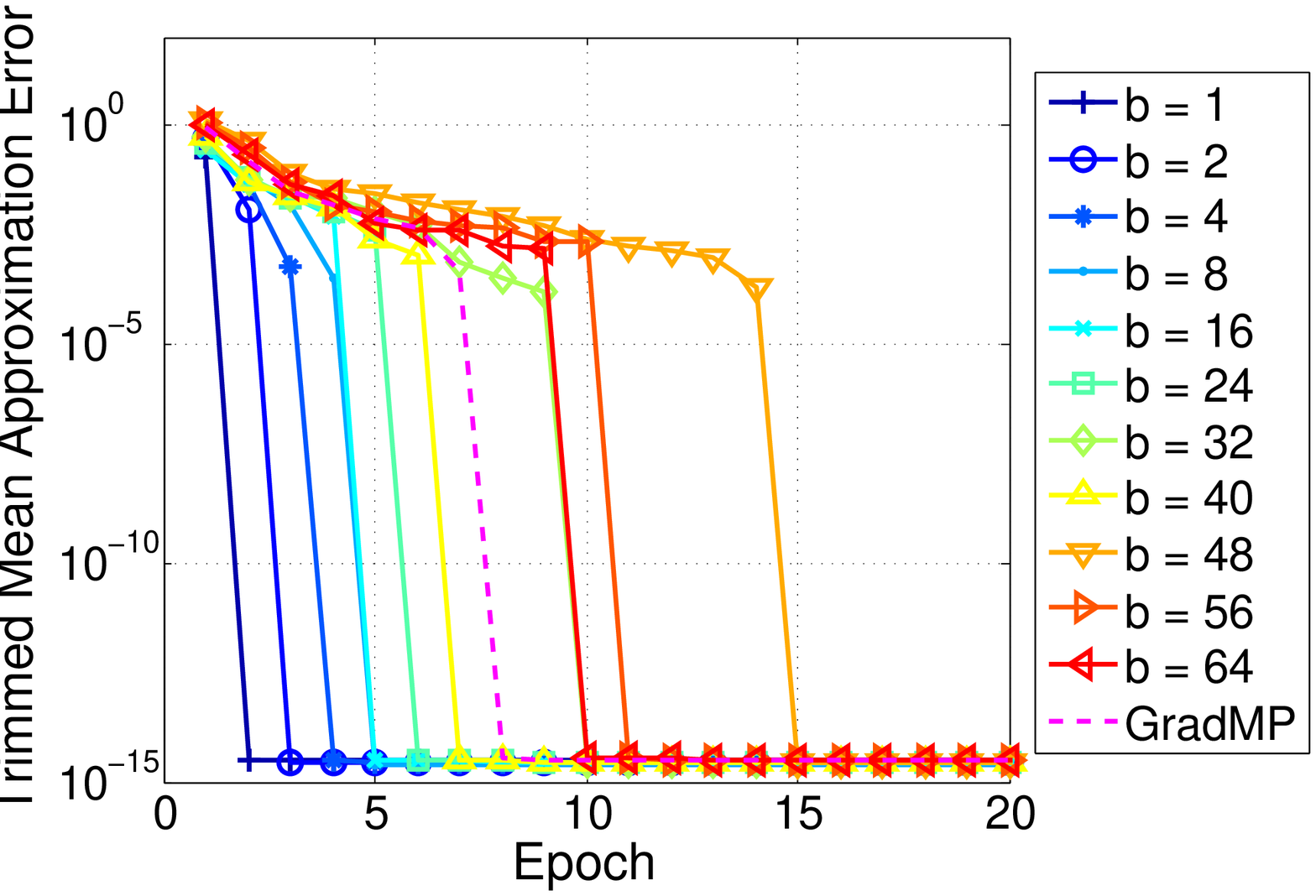}
\end{tabular}
\caption{Sparse Vector Recovery: Recovery error as a function of epochs and various block sizes $b$ for HT methods (left) and GradMP methods (right).\label{fig3}}
\end{figure} 

Next we repeat the same experiments but examine the recovery error as a function of the number of measurements for various block sizes (note that if the block size exceeds the number of measurements, we simply use the entire matrix as one block).  Figure~\ref{fig4} shows these results.  Because the methods exhibit graceful decrease in recovery error, here we plot the number of measurements (as a function of block size) required in order for the estimation error $\|w-\hat{w}\|_2$ to drop and remain below $10^{-6}$.  Although block size is not a parameter for the deterministic methods IHT and GradMP, a red horizontal line at the number of measurements required is included for comparison. We see that the fewest measurements are required when the block sizes are about $10$ (recall the signal dimension is $256$). We also note that StoIHT requires fewer measurements than IHT for large blocks, whereas StoGradMP requires the same as GradMP for large blocks, which is not surprising.  However, we see that both methods offer improvements over their deterministic counterparts if the block sizes are chosen correctly.

\begin{figure}[ht]
\begin{tabular}{cc}
\includegraphics[height=2.25in,width=3.05in]{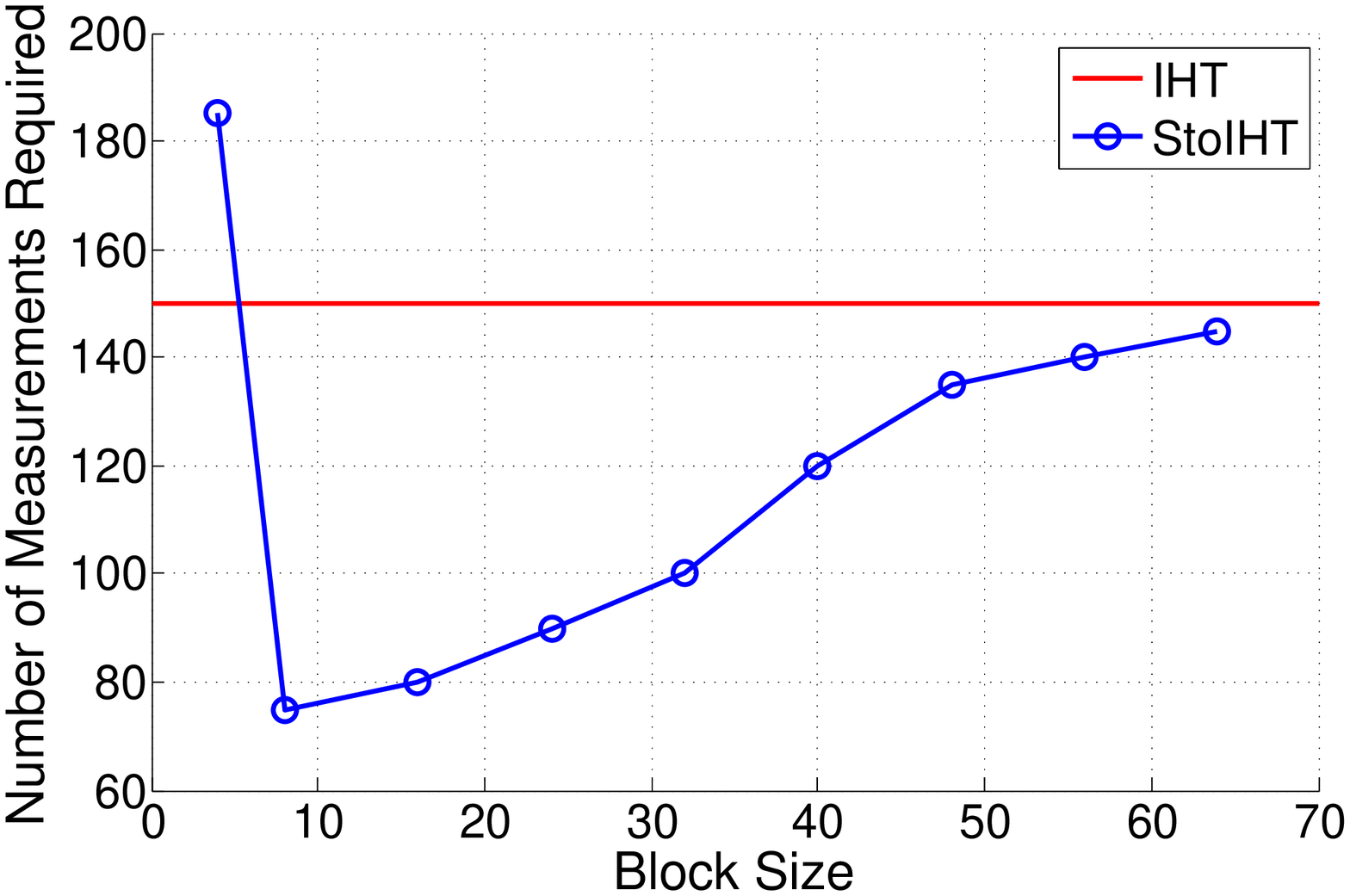} & 
\includegraphics[height=2.25in,width=3.05in]{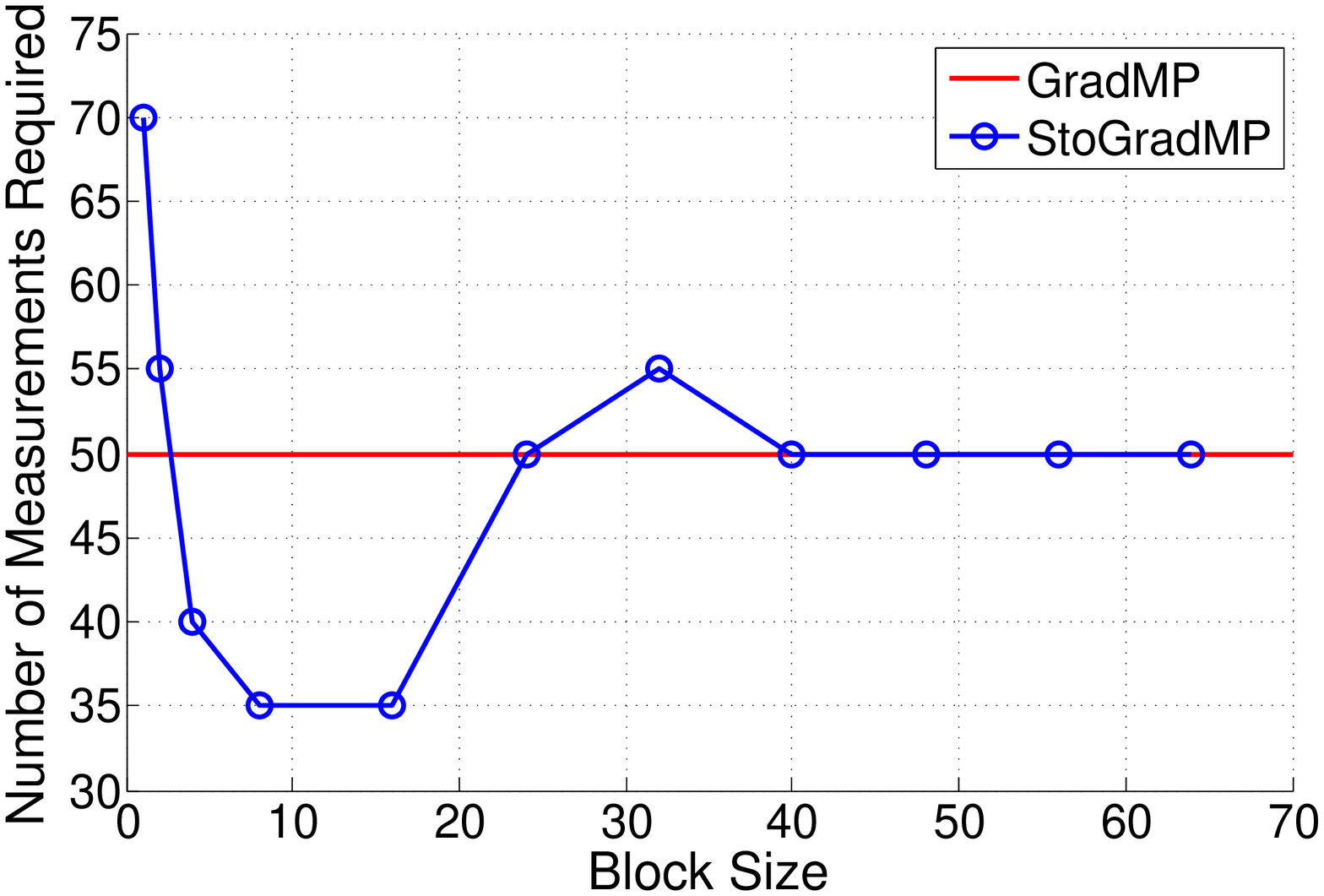}
\end{tabular}
\caption{Sparse Vector Recovery: Number of measurements required for signal recovery as a function of block size (blue marker) for StoIHT (left) and StoGradMP (right).  Number of measurements required for deterministic method shown as red solid line.\label{fig4}}
\end{figure} 

\subsection{Robustness to measurement noise}
We next repeat the above sparse vector recovery experiments in the presence of noise in the measurements.  All experiment parameters remain as in the previous setup, but a vector $e$ of Gaussian noise with $\|e\|_2 = 0.5$ is added to the measurement vector.  We again compare the recovery error against the number of epochs and measurements needed.  The results are shown in Figures~\ref{fig5} and \ref{fig6} for the IHT and GradMP algorithms, respectively.  The right hand plots show the number of measurements required for the error to drop below the noise level $0.5$ as a function of block size.  Overall, the methods are robust to noise and demonstrate the same improvements and heuristics as in the noiseless experiments.  

\begin{figure}[ht]
\begin{tabular}{cc}
\includegraphics[height=2.25in,width=3.22in]{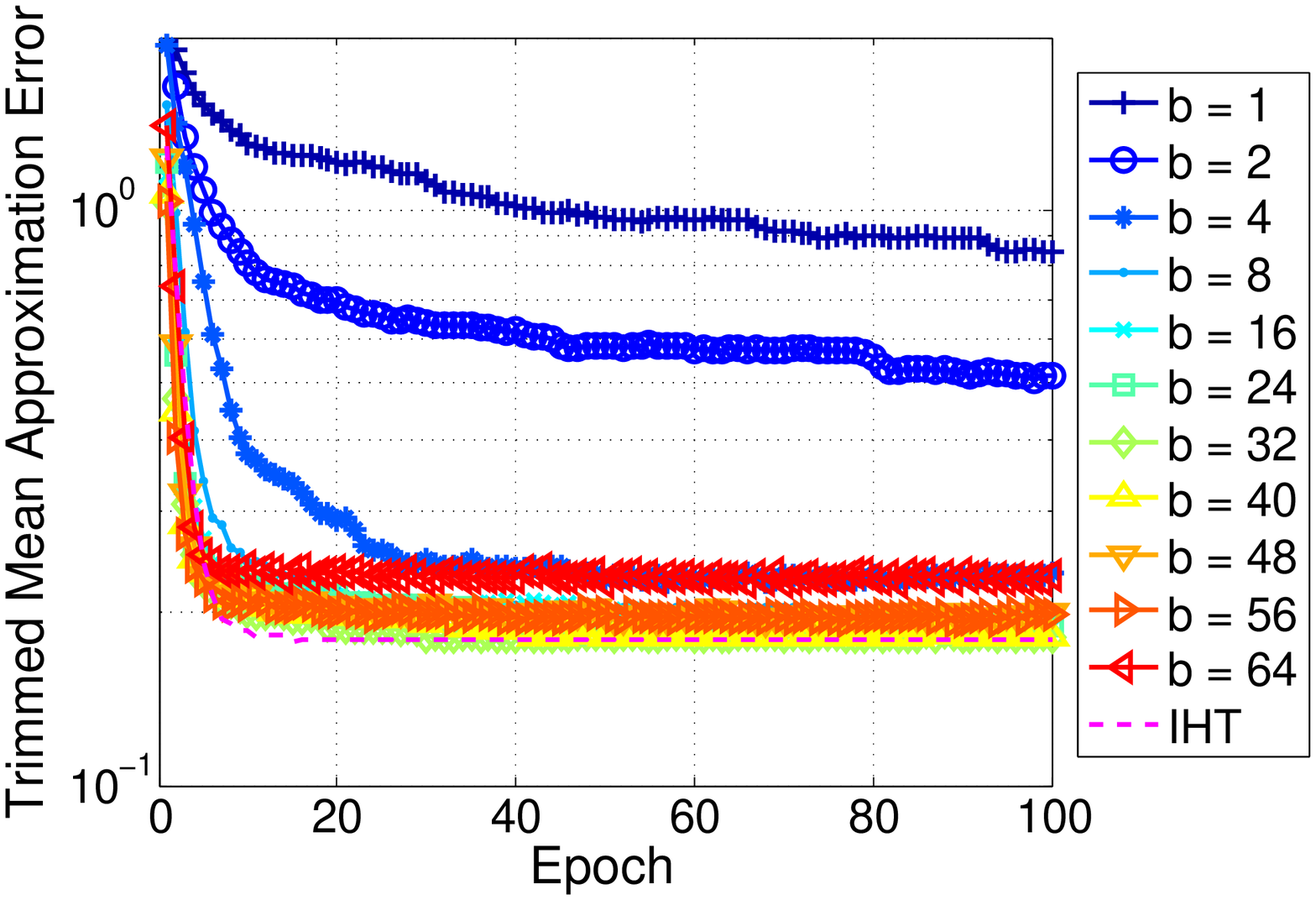} &
\includegraphics[height=2.25in,width=2.9in]{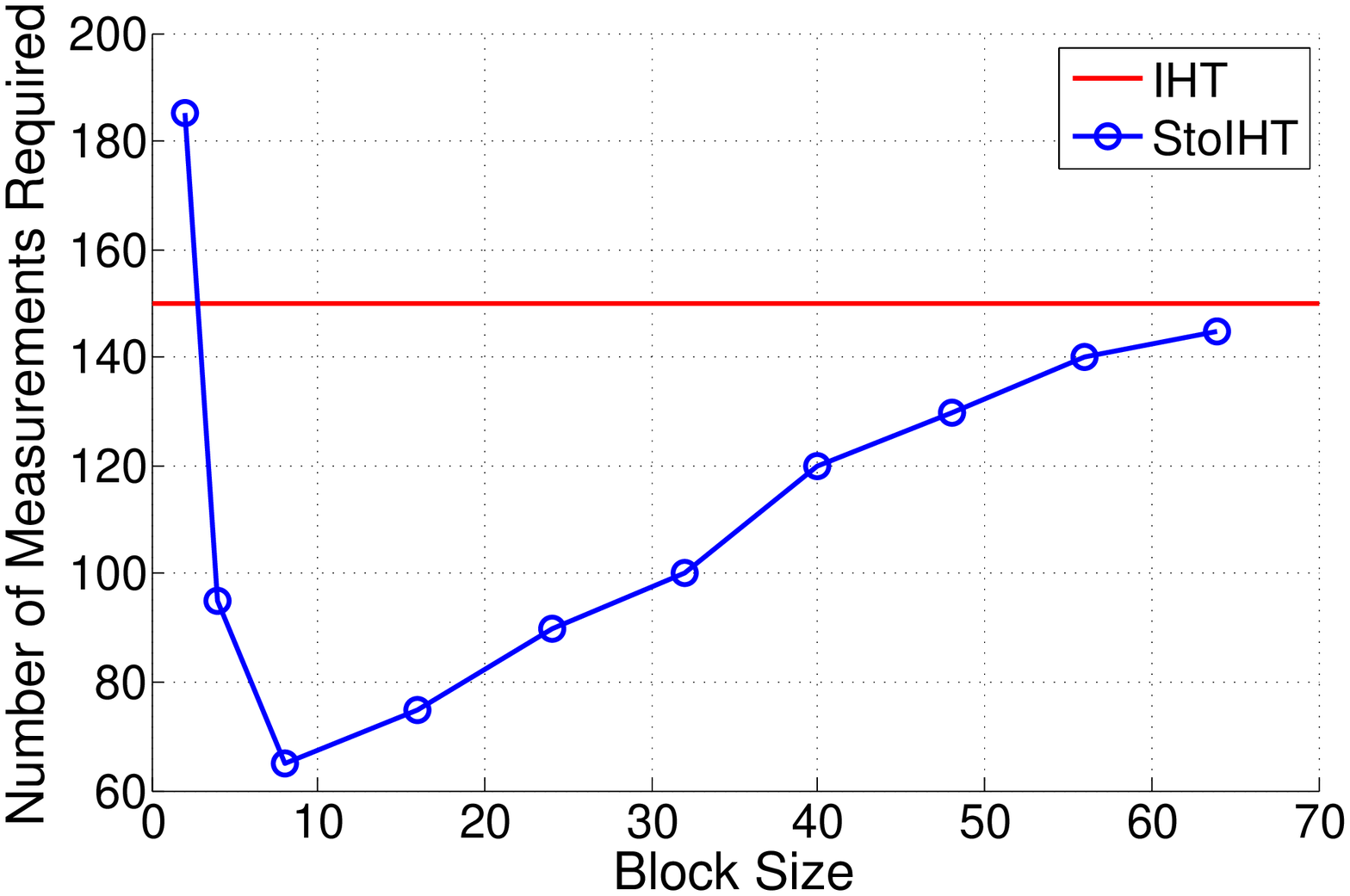} 
\end{tabular}
\caption{Sparse Vector Recovery: A comparison of IHT and StoIHT in the presence of noise.  Recovery error versus epoch (left) and measurements required versus block size (right).\label{fig5}}
\end{figure} 

\begin{figure}[ht]
\begin{tabular}{cc}
 \includegraphics[height=2.25in,width=3.22in]{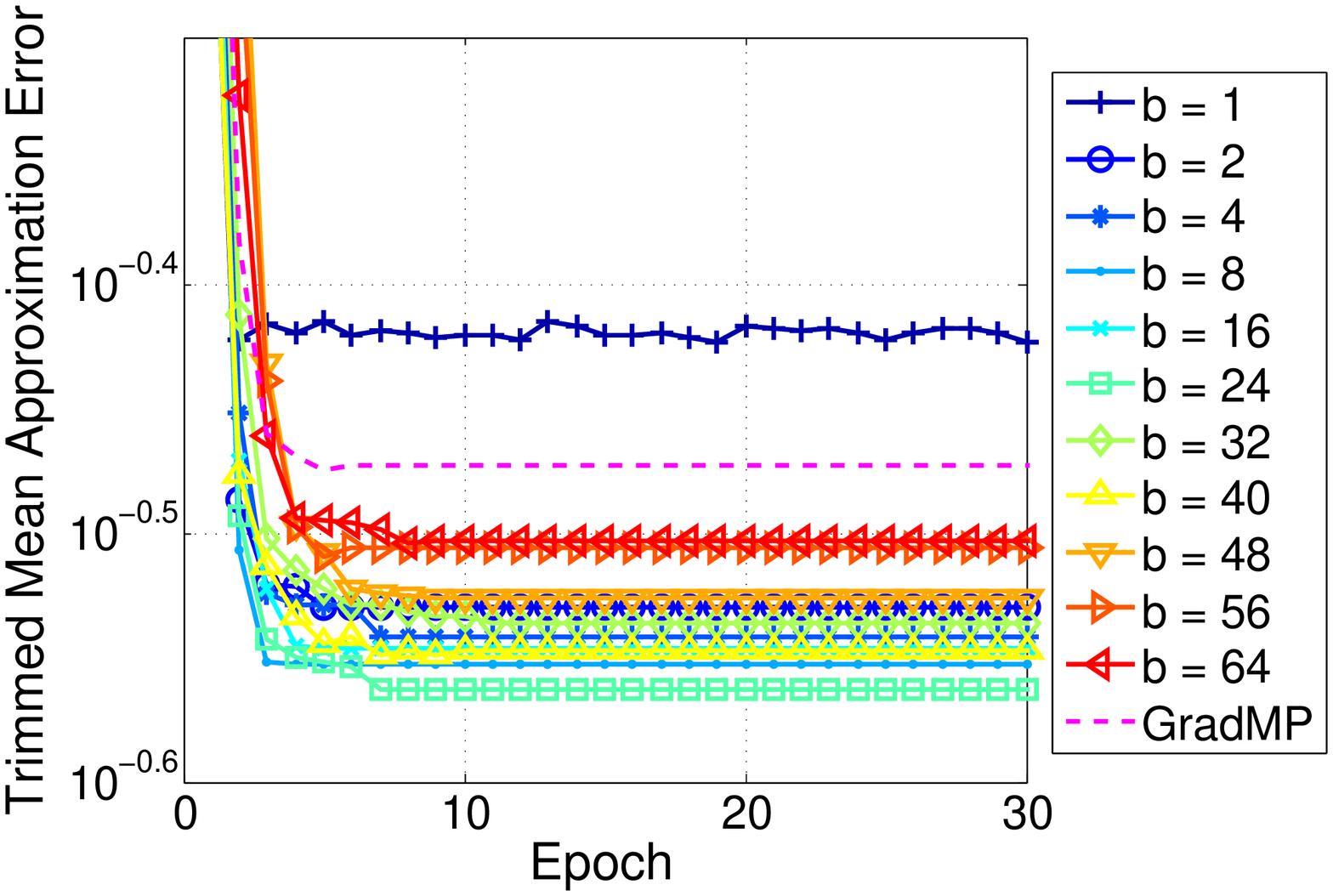} & 
 \includegraphics[height=2.25in,width=2.9in]{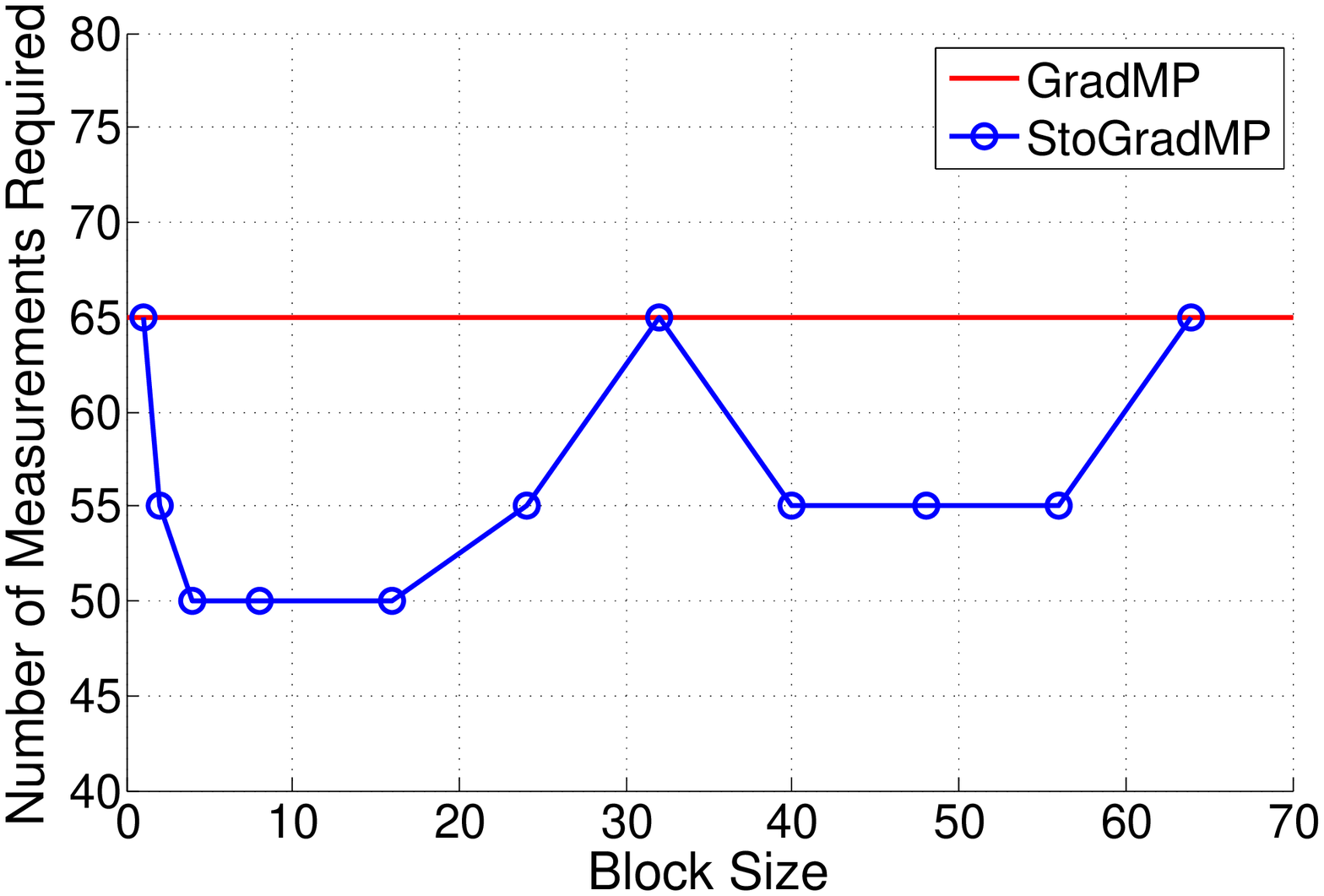} 
\end{tabular}
\caption{Sparse Vector Recovery: A comparison of GradMP and StoGradMP in the presence of noise. Recovery error versus epoch (left) and measurements required versus block size (right).\label{fig6}}
\end{figure} 

\subsection{The choice of step size in StoIHT}

Our last experiment in the sparse vector recovery setting explores the role of the step size $\gamma$ in StoIHT.  Keeping the dimension of the signal at $256$, the sparsity $k_0=8$, the number of measurements $m=80$, no noise, and fixing the block size $b=8$, we test the algorithm using various values of the step size $\gamma$.  The results are shown in Figure~\ref{fig7}.  We see that the value of $\gamma$ clearly plays a role, but the range of successful values is quite large.  Not surprisingly, too small of a step size leads to extremely slow convergence, and too large of one leads to divergence (at least initially).

\begin{figure}[ht]
\begin{center}
 \includegraphics[height=3in,]{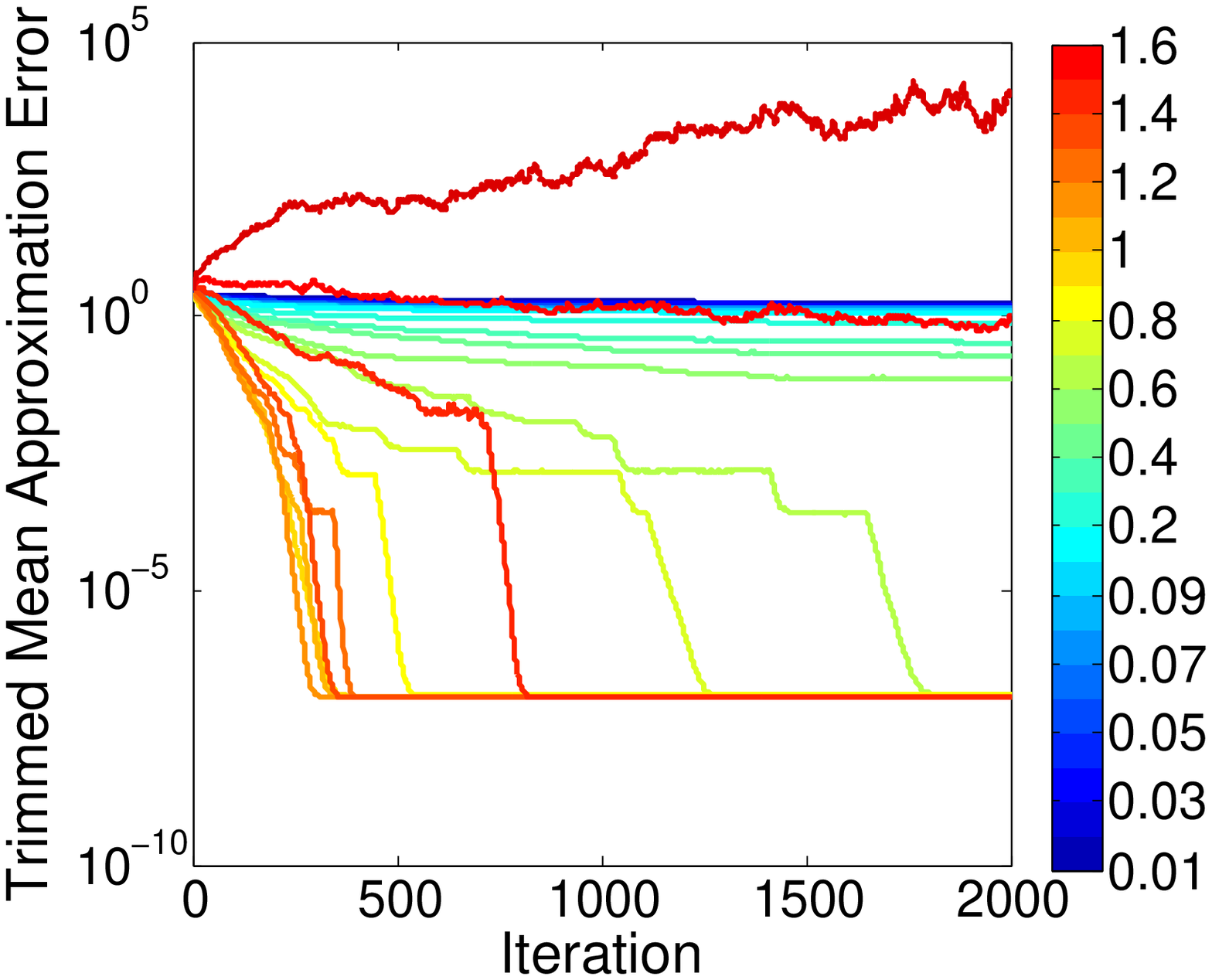} 
\end{center}
\caption{Sparse Vector Recovery: A comparison of StoIHT for various values of the step size $\gamma$ (shown in the colorbar).\label{fig7}}
\end{figure} 

\subsection{Low-Rank Matrix Recovery}

We now turn to the setting where we wish to recover a low-rank matrix $W_0$ from $m$ linear measurements as studied in Subsection \ref{subsec::Low-rank matrix recovery}. Here $W_0$ is the $10\times 10$ matrix with rank $k_0$ and we take $m$ linear Gaussian measurements of the form $y_i = \langle A_i, W_0 \rangle$, where each $A_i$ is a $10\times 10$ matrix with i.i.d. standard Gaussian entries.  As before, we first compare the percentage of exact recovery (where again we deem the signal is recovered exactly when the error $\|W_0-\hat{W}\|_F$ is below $10^{-6}$) against the number of measurements required, for various rank levels. For the matrix case, we use a step size of $\gamma=0.5$ for both the IHT and StoIHT methods, which seems to work well in this setting. The results for IHT and StoIHT are shown in Figure~\ref{fig8} and for GradMP and StoGradMP in Figure~\ref{fig9}. For this choice of parameters, we see that both StoIHT and StoGradMP tend to require fewer measurements to recover the signal.

\begin{figure}[ht]
\begin{tabular}{cc}
 \includegraphics[height=2.25in,width=2.9in]{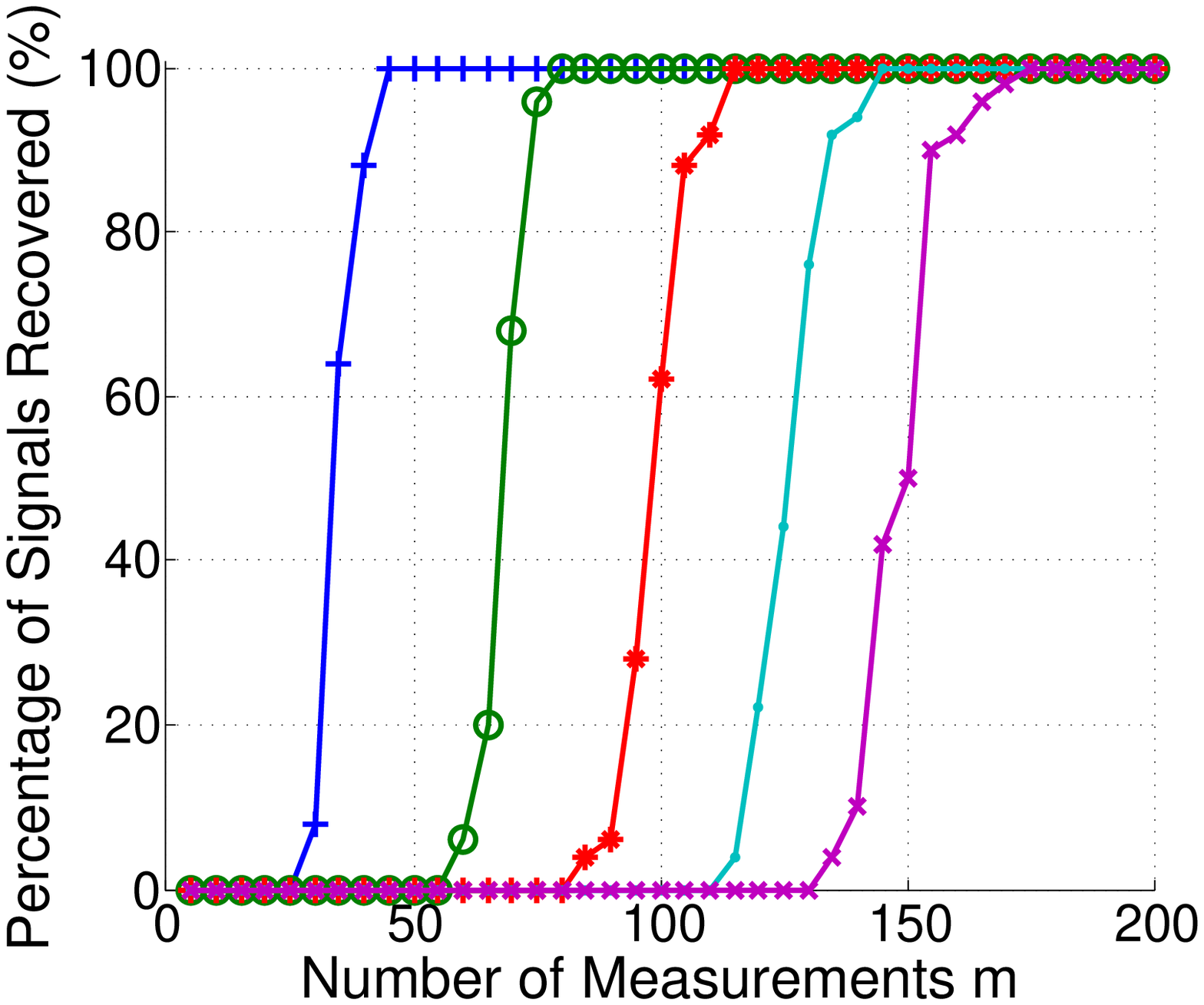} & 
 \includegraphics[height=2.25in,width=3.22in]{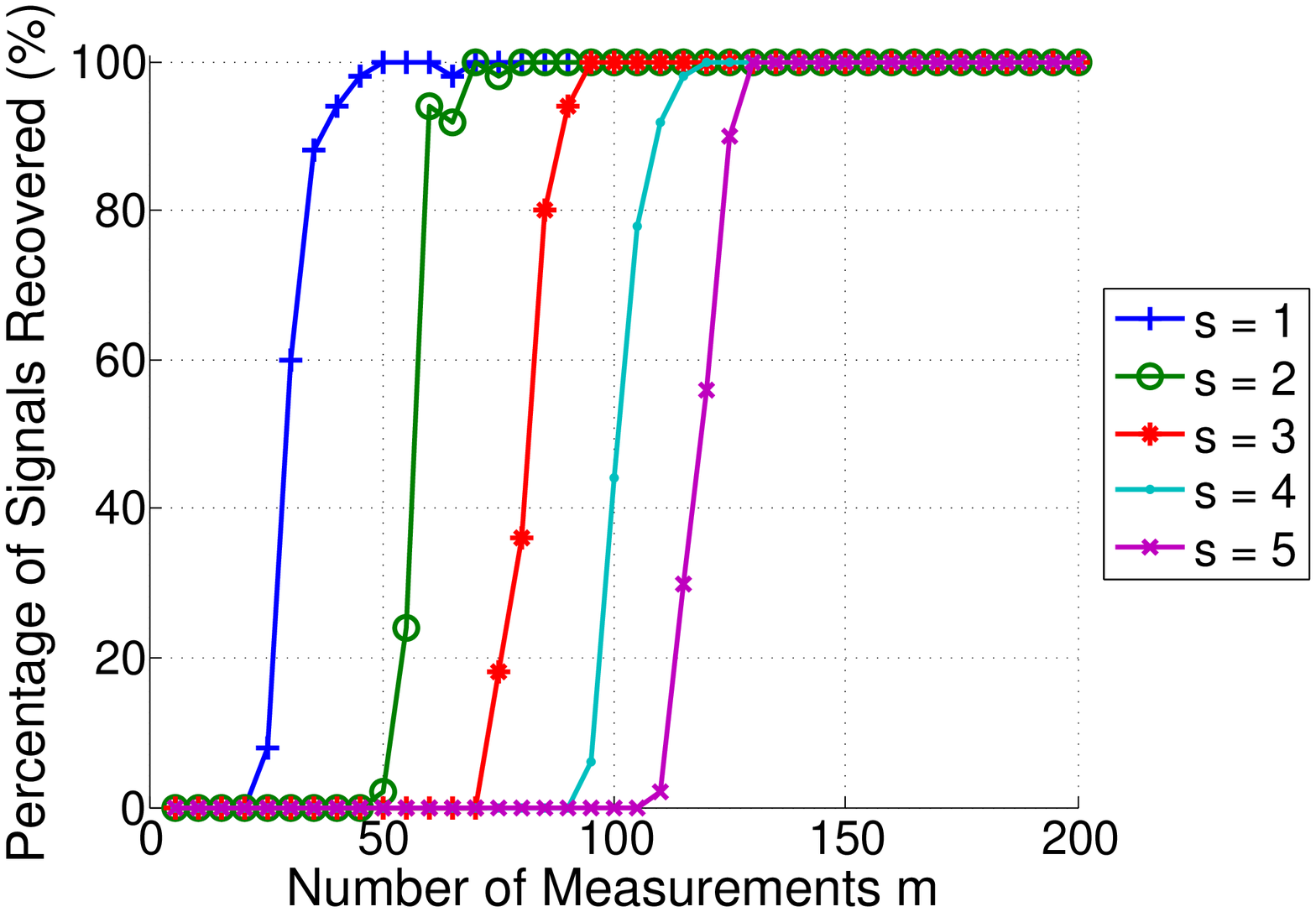} 
\end{tabular}
\caption{Low-Rank Matrix Recovery: Percent recovery as a function of the number of measurements for IHT (left) and StoIHT (right) for various rank levels $s$.\label{fig8}}
\end{figure} 

\begin{figure}[ht]
\begin{tabular}{cc}
 \includegraphics[height=2.25in,width=2.9in]{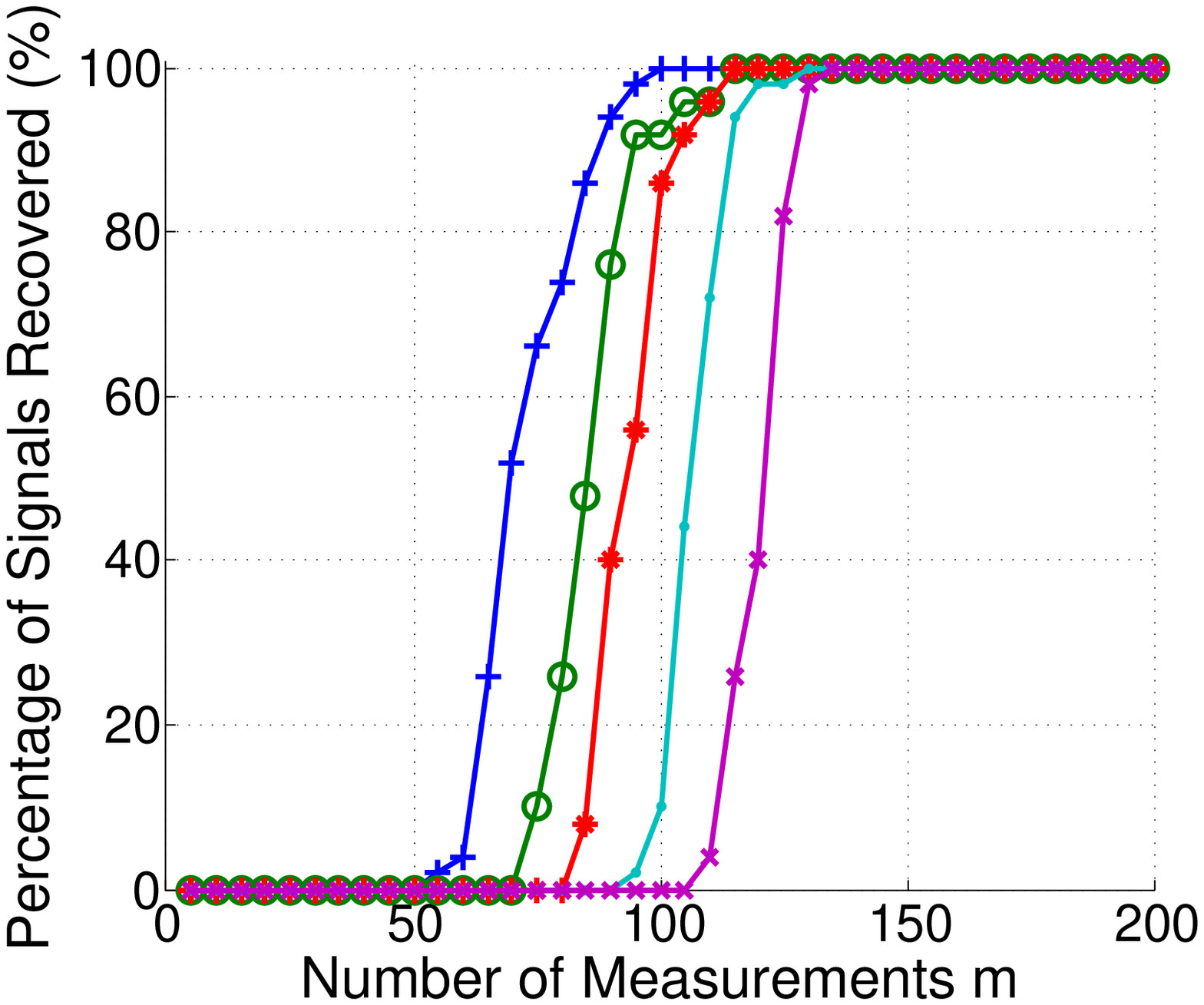} & 
 \includegraphics[height=2.25in,width=3.22in]{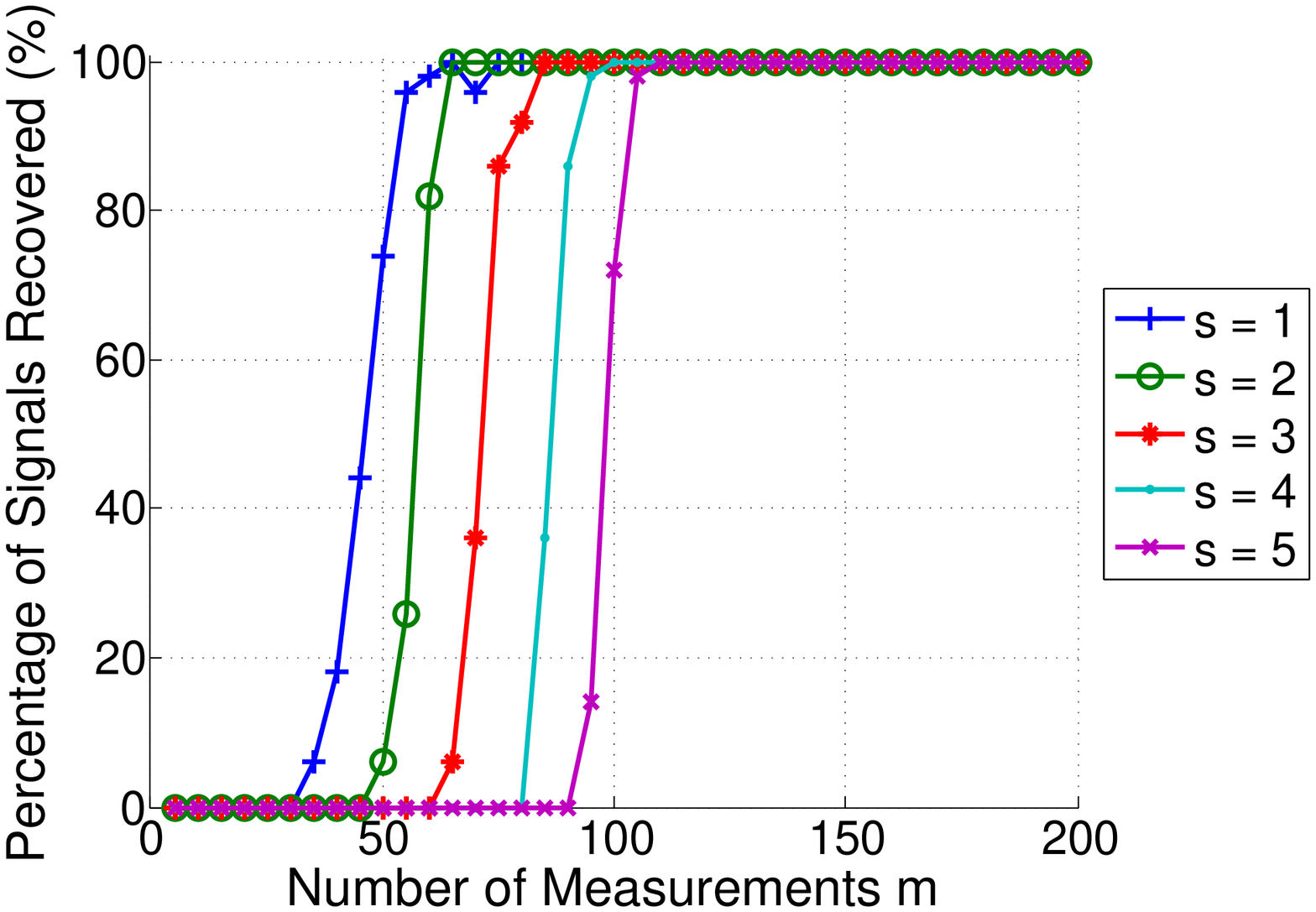} 
\end{tabular}
\caption{Low-Rank Matrix Recovery: Percent recovery as a function of the number of measurements for GradMP (left) and StoGradMP (right) for various rank levels $k_0$.\label{fig9}}
\end{figure} 

Next we examine the signal recovery error as a function of epoch, for various block sizes and against the deterministic methods.  We fix the rank to be $k_0=2$ in these experiments.  Because both block size and number of measurements affect the convergence, we see different behavior in the low measurement regime and the high measurement regime.  This is apparent in Figure~\ref{fig10}, where $m=90$ measurements are used in the plot on the left and $m=140$ measurements are used in the plot on the right, which shows the convergence of the IHT methods per epoch for various block sizes.  We again see that for proper choices of block sizes, the StoIHT method outperforms IHT.  It is also interesting to note that IHT seems to reach a higher noise floor than StoIHT.  Of course we again point out that we have not optimized any of the algorithm parameters for either method.  Results for the GradMP methods are shown in Figure~\ref{fig11}, again where $m=90$ measurements are used in the plot on the left and $m=140$ measurements are used in the plot on the right. Similar to the IHT results, proper choices of block sizes allows StoGradMP to require much fewer epochs than GradMP to achieve convergence.

\begin{figure}[ht]
\begin{tabular}{cc}
 \includegraphics[height=2.25in,width=2.9in]{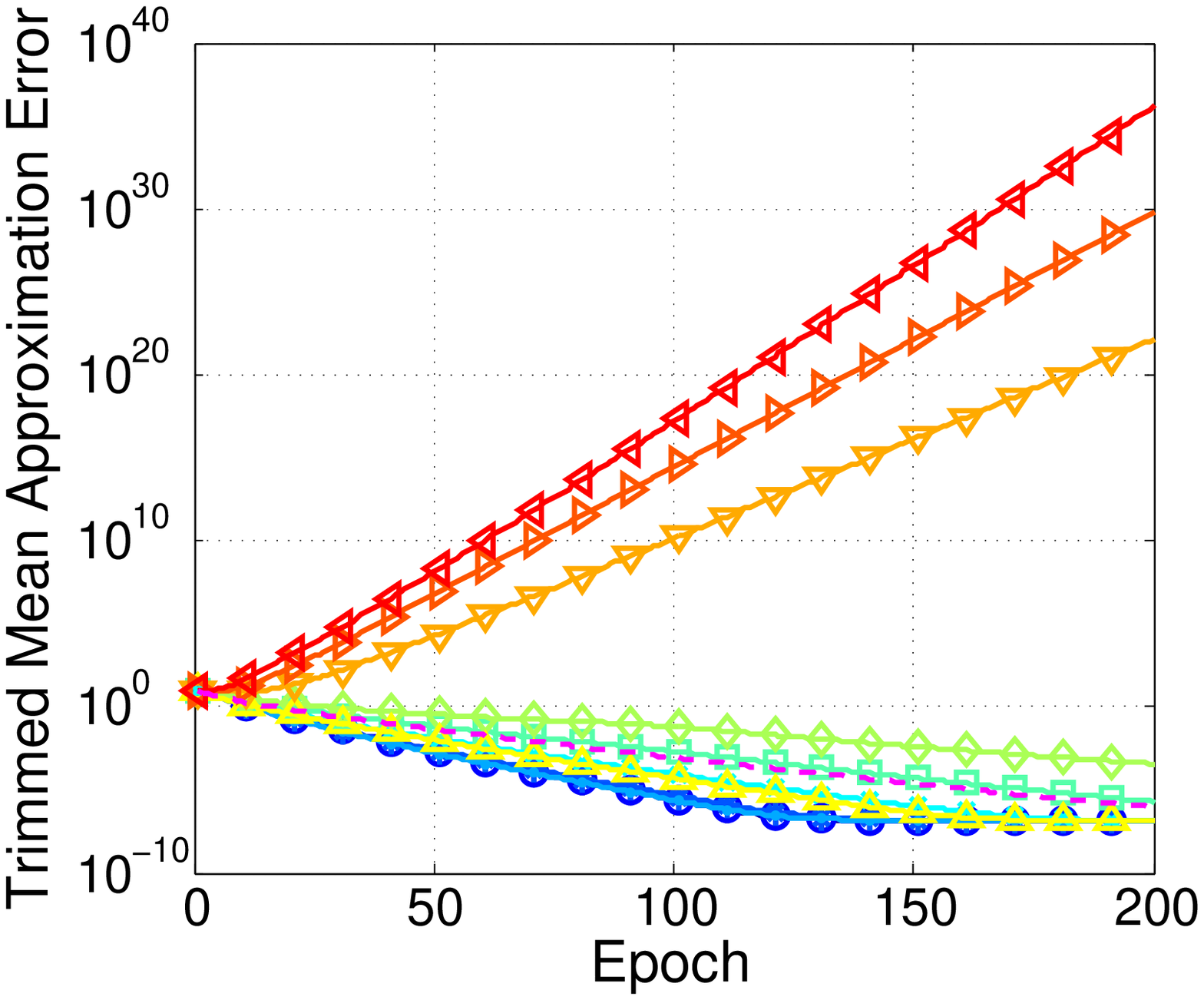} & 
 \includegraphics[height=2.25in,width=3.22in]{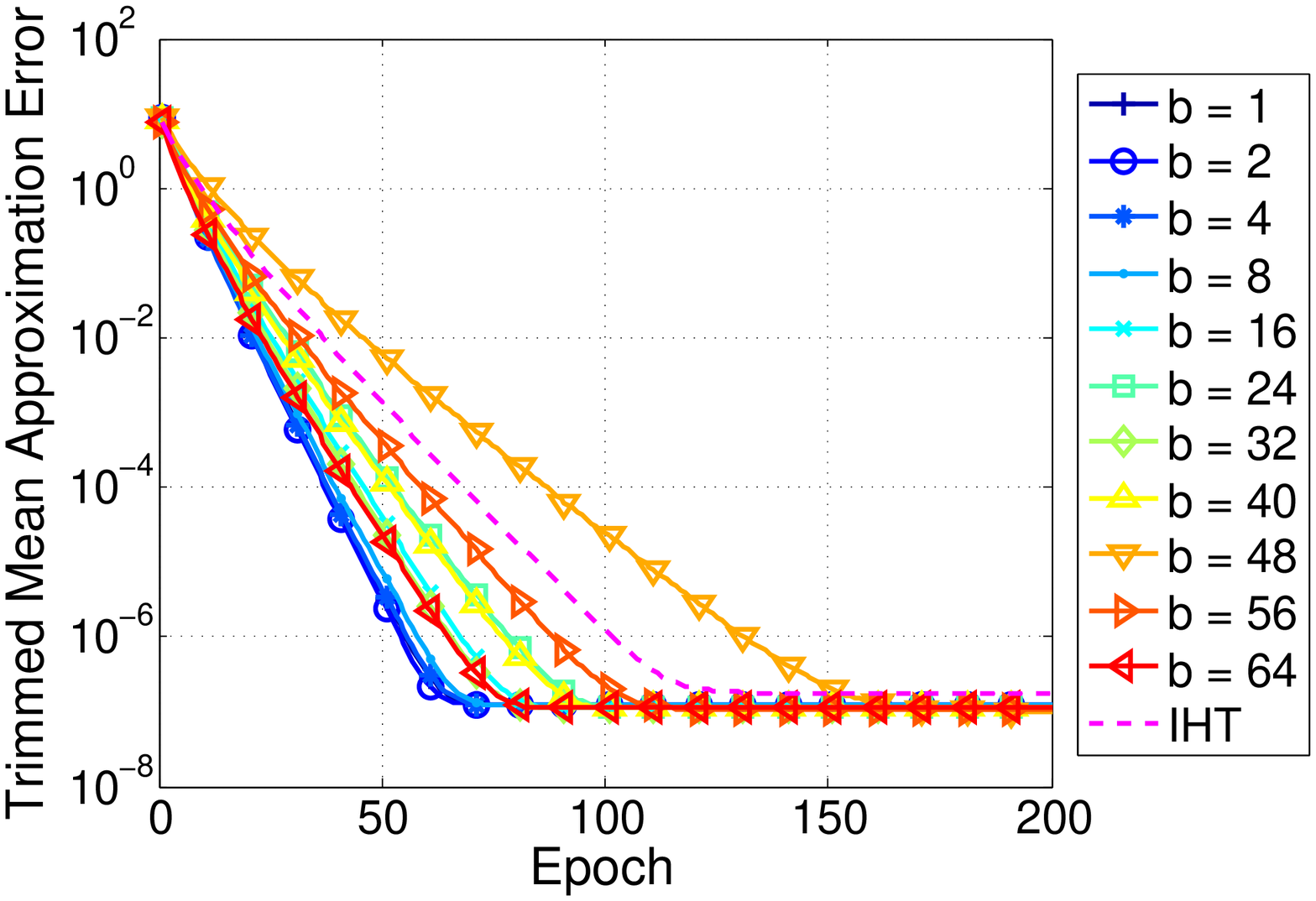} 
\end{tabular}
\caption{Low-Rank Matrix Recovery: Recovery error as a function of the number of epochs for StoIHT methods using $m=90$ (left) and $m=140$ (right) for various block sizes $b$.\label{fig10}}
\end{figure} 

\begin{figure}[ht]
\begin{tabular}{cc}
 \includegraphics[height=2.25in,width=2.9in]{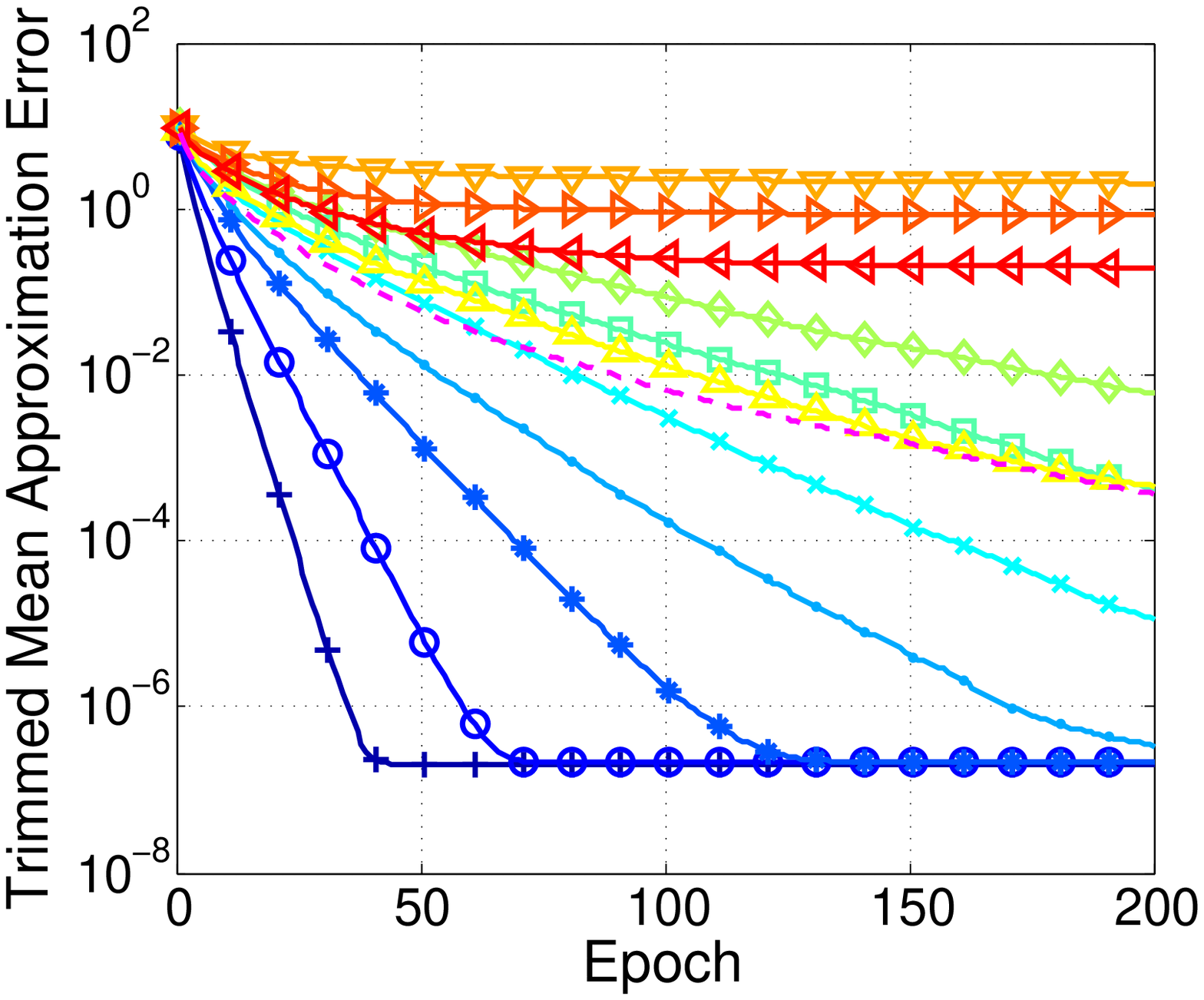} & 
 \includegraphics[height=2.25in,width=3.22in]{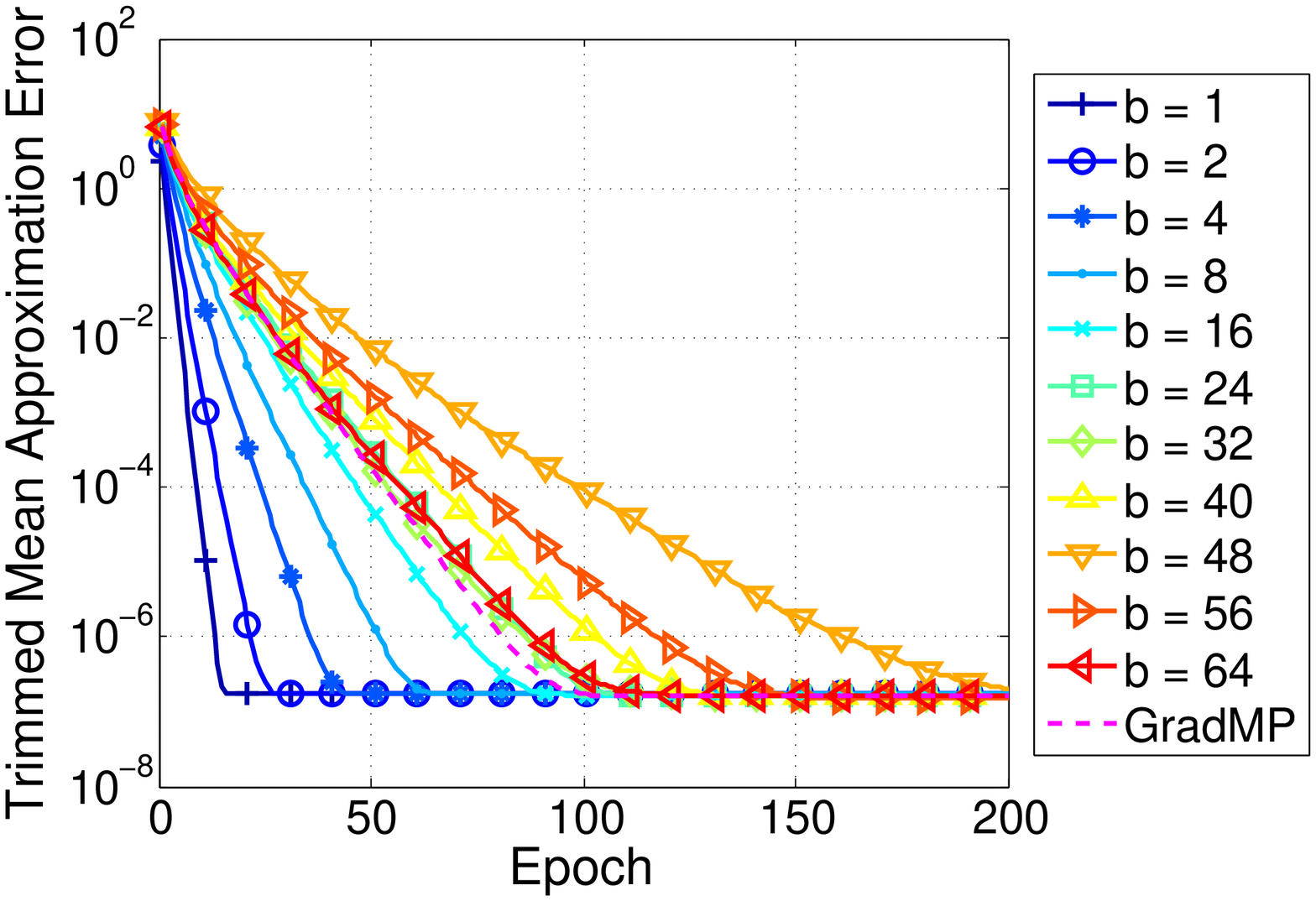} 
\end{tabular}
\caption{Low-Rank Matrix Recovery: Recovery error as a function of the number of epochs for the StoGradMP algorithm using $m=90$ (left) and $m=140$ (right) for various block sizes $b$.\label{fig11}}
\end{figure} 

Figure~\ref{fig12} compares the block size and the number of measurements required for exact signal recovery for the IHT methods and the GradMP methods, again for a fixed rank of $k_0=2$.  We again see that StoIHT and StoGradMP prefer small block sizes.

\begin{figure}[ht]
\begin{tabular}{cc}
 \includegraphics[height=2.25in,width=3.05in]{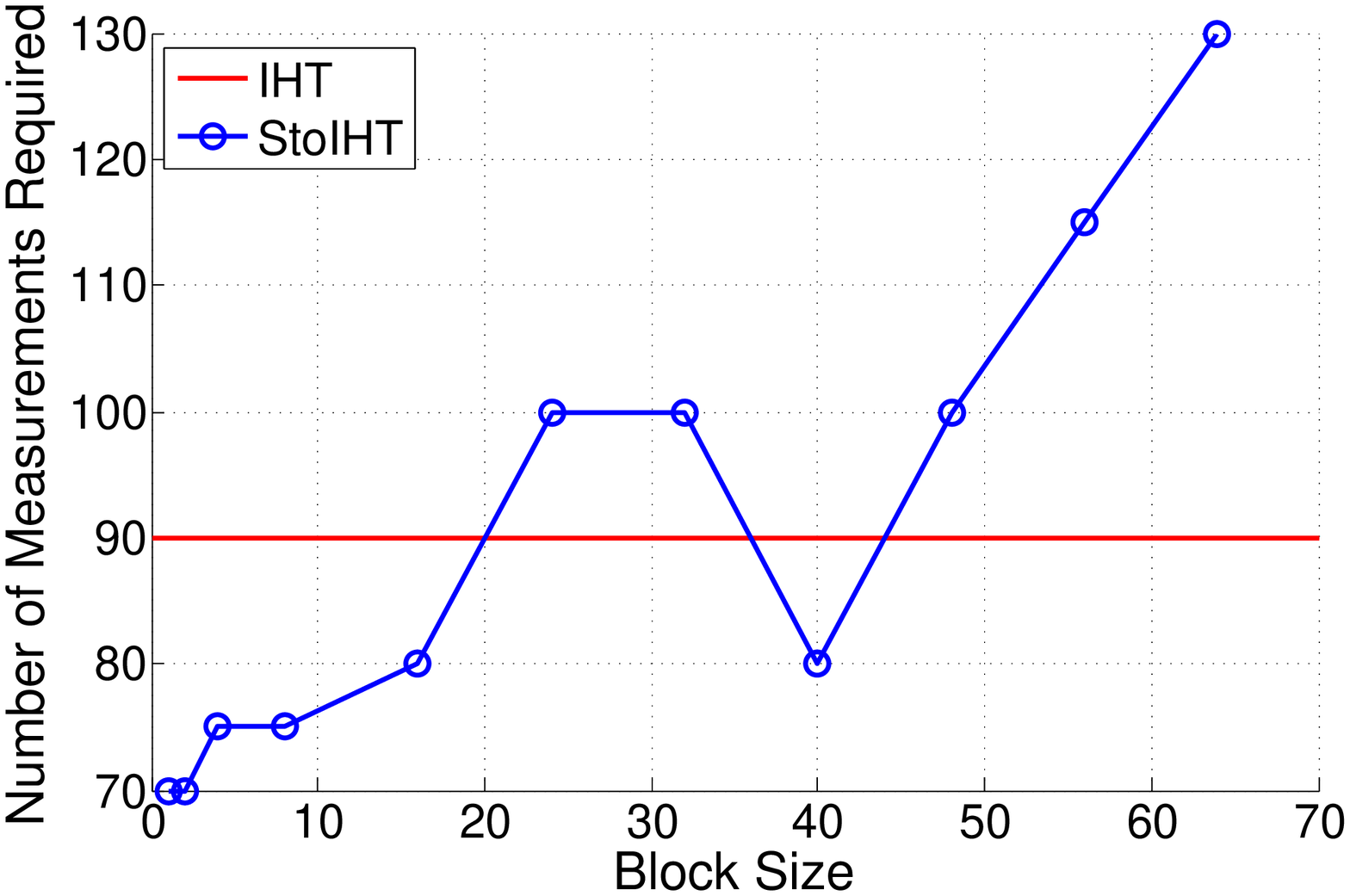} & 
 \includegraphics[height=2.25in,width=3.05in]{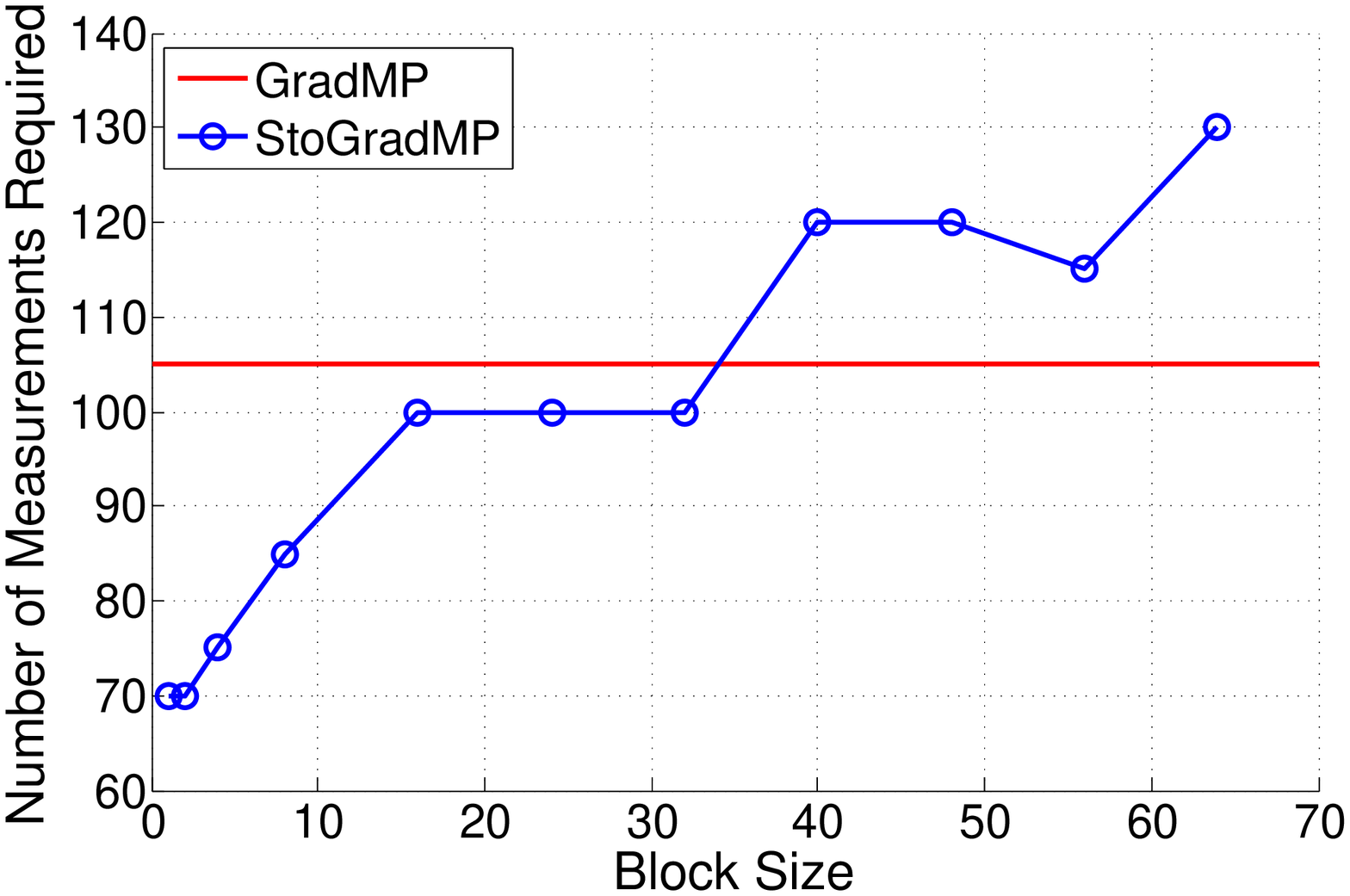} 
\end{tabular}
\caption{Low-Rank Matrix Recovery: Number of measurements required for signal recovery as a function of block size (blue marker) for StoIHT (left) and StoGradMP (right).  Number of measurements required for deterministic method shown as red solid line.  \label{fig12}}
\end{figure} 

\subsection{Recovery with approximations}
Finally, we consider the important case where the identification, estimation, and pruning steps can only be performed approximately.  In particular, we consider the case of low-rank matrix recovery in which these steps utilize only an approximate Singular Value Decomposition (SVD) of the matrix.  This may be something that is unavoidable in certain applications, or may be desirable in others for computational speedup.  For our first experiments of this kind, we use $N=1024$ and generate a $N\times N$ rank $k_0=40$ matrix.  We take $m$ permuted rows of the $N\times N$ discrete Fourier transform as the measurement operator, use $2$ blocks in the stochastic algorithms, and run $40$ trials. In the StoIHT experiments we take $m=0.3N^2$, and in the StoGradMP experiments we take $m=0.35N^2$; these values for $m$ empirically seemed to work well with the two algorithms. For each trial of the approximate SVD, we also run $5$ sub-trials, to account for the randomness used in the approximate SVD algorithm.  Here we use the randomized method described in~\cite{halko2011finding} to compute the approximate SVD of a matrix.  Briefly, to obtain a rank-$s$ approximation of a matrix X and compute its approximated SVD, one applies the matrix to a randomly generated $N\times (s+d)$ matrix $\Omega$ to obtain the product $Y=X\Omega$ and constructs an orthonormal basis $Q$ for the column space of $Y$.  Here, $d$ is an \textit{over-sampling factor} that can be tuned to balance the tradeoff between accuracy and computation time.  Using this basis, one computes the SVD of the product $B=Q^*X = U\Sigma V^*$, and approximates the SVD of $X$ by $X\approx (QU)\Sigma V^*$.  Because $(s+d)$ is typically much less than $N$, significant speedup can be gained. In addition, \cite{halko2011finding} proves that the approximation error is bounded by
$$
\norm{X - X_s}_F \leq \left(1 + \sqrt{\frac{s}{s+d}} \right) \norm{X - X^{\text{best}}_s}_F,
$$
where $X^{\text{best}}_s$ is the best rank-$s$ approximation of $X$ and $X_s$ is the approximate rank-$s$ matrix produced from the above procedure. Here, the multiplicative error is associated with the quantity $\eta$ in the approximation operator $\text{approx}_s(w,\eta)$ defined in (\ref{eqt::approximation definition}).

Figure~\ref{fig13} shows the approximation error as a function of epoch and runtime for the StoIHT algorithm, for various over-sampling factors $d$ as well as the full SVD computation for comparison.  We again use a $10\%$ trimmed mean over all the trials, and a step-size of $\gamma=0.5$.  We see that in terms of epochs, for reasonably sized over-sampling factors, the convergence using the SVD approximation is very similar to that of using the full SVD.  In terms of runtime, we see a significant speedup for moderate choices of over-sampling factor, as expected. Recall that $2$ blocks were used for this experiment, but we have observed a very similar relationship between the curves when increasing the number of blocks to $10$. 

The analogous results for StoGradMP are very similar, and are shown in Figure ~\ref{fig14}. We again see that for certain over-sampling factors, the convergence of the approximation error as a function of epoch is similar when using the approximate SVD and the full SVD. We also see a very significant speedup when using the approximate SVD; in this case, all the over-sampling factors used in this experiment offer an improved runtime over the full SVD computation. 

\begin{figure}[ht]
\begin{tabular}{cc}
 \includegraphics[height=2.25in,width=2.9in]{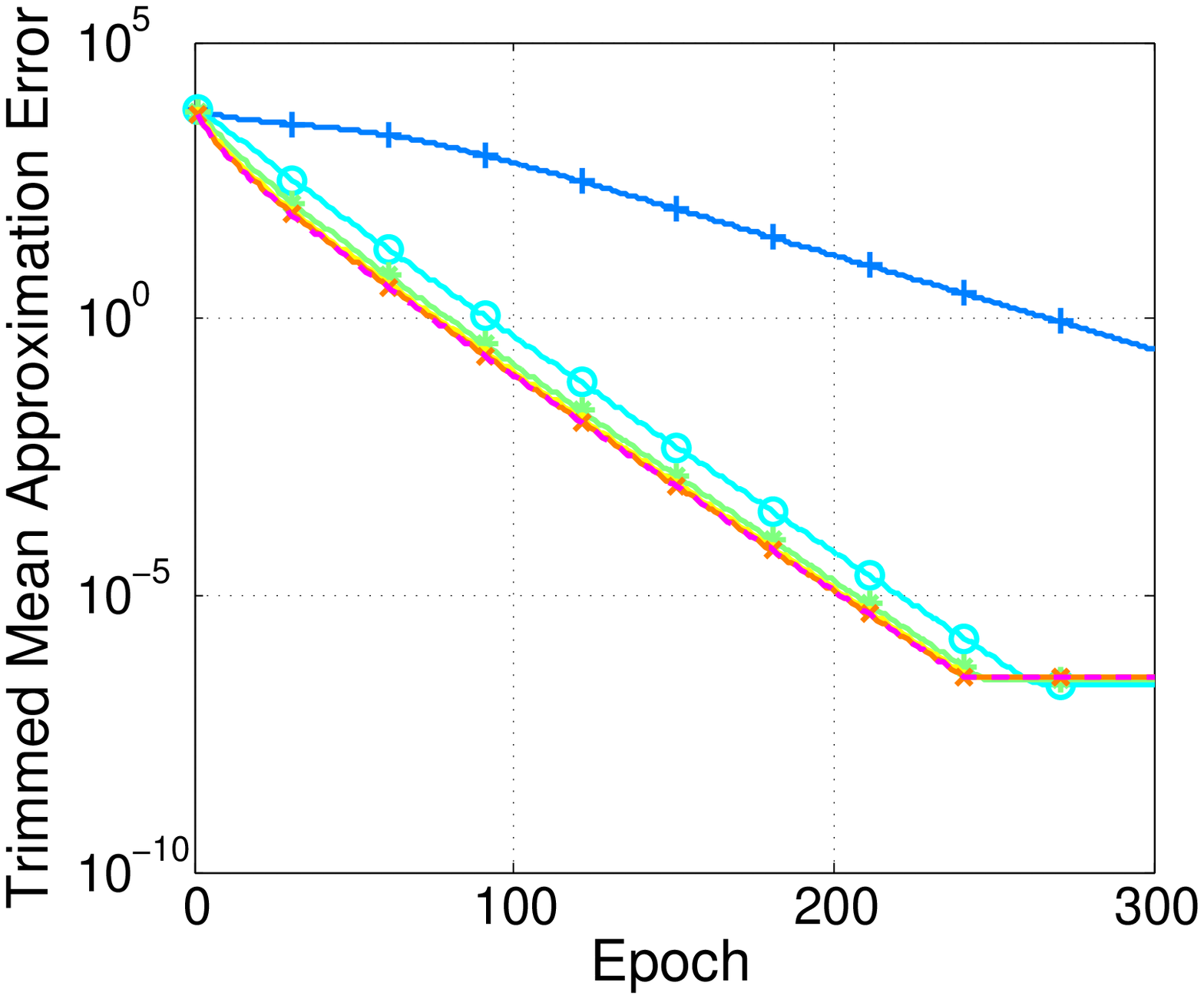} & 
 \includegraphics[height=2.25in,width=3.22in]{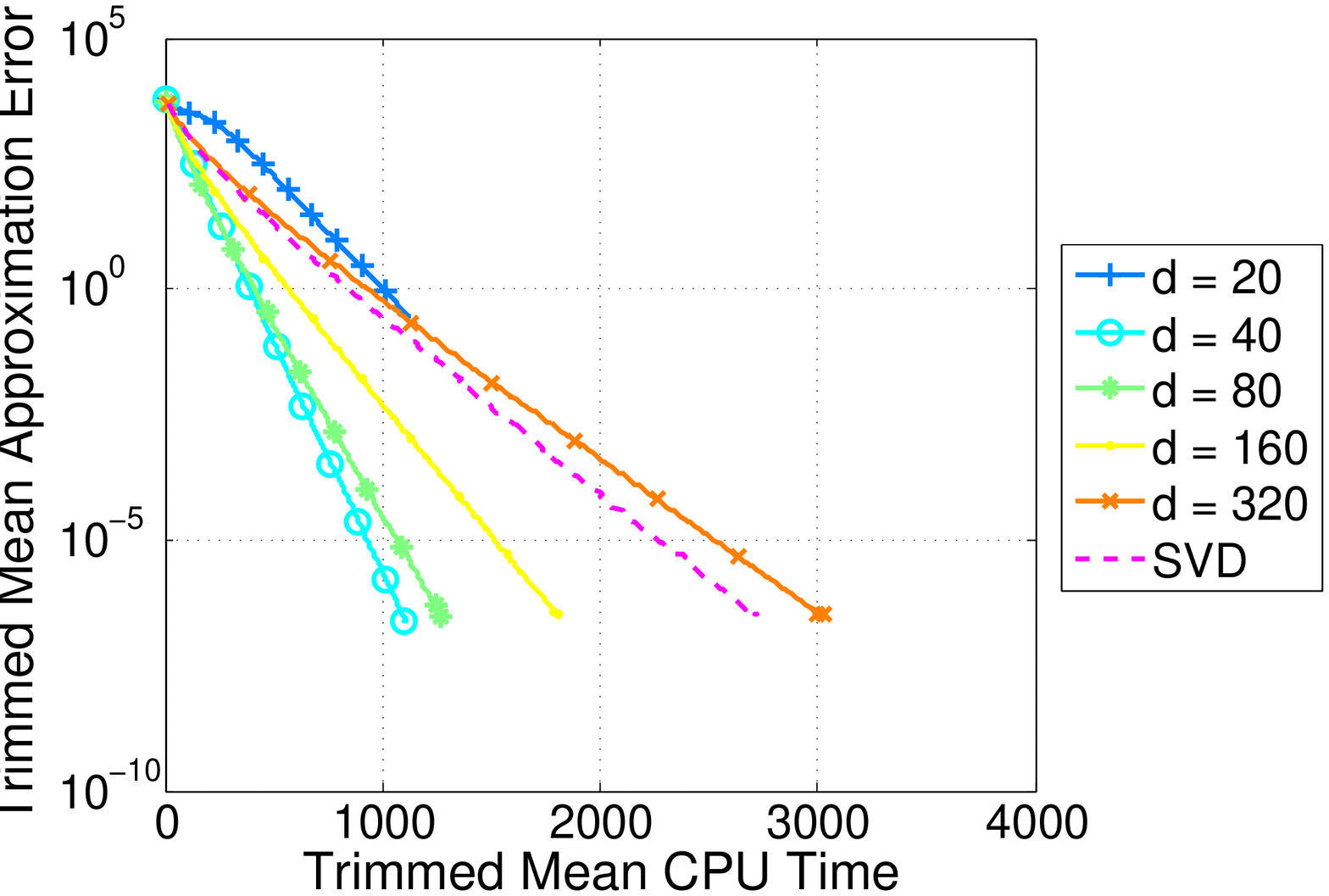} 
\end{tabular}
\caption{Low-Rank Matrix Recovery with Approximations: Trimmed mean recovery error as a function of epochs (left) and runtime (right) for various over-sampling factors $d$ using the StoIHT algorithm.  Performance using full SVD computation shown as dashed line.  \label{fig13}}
\end{figure} 

\begin{figure}[ht]
\begin{tabular}{cc}
 \includegraphics[height=2.25in,width=2.9in]{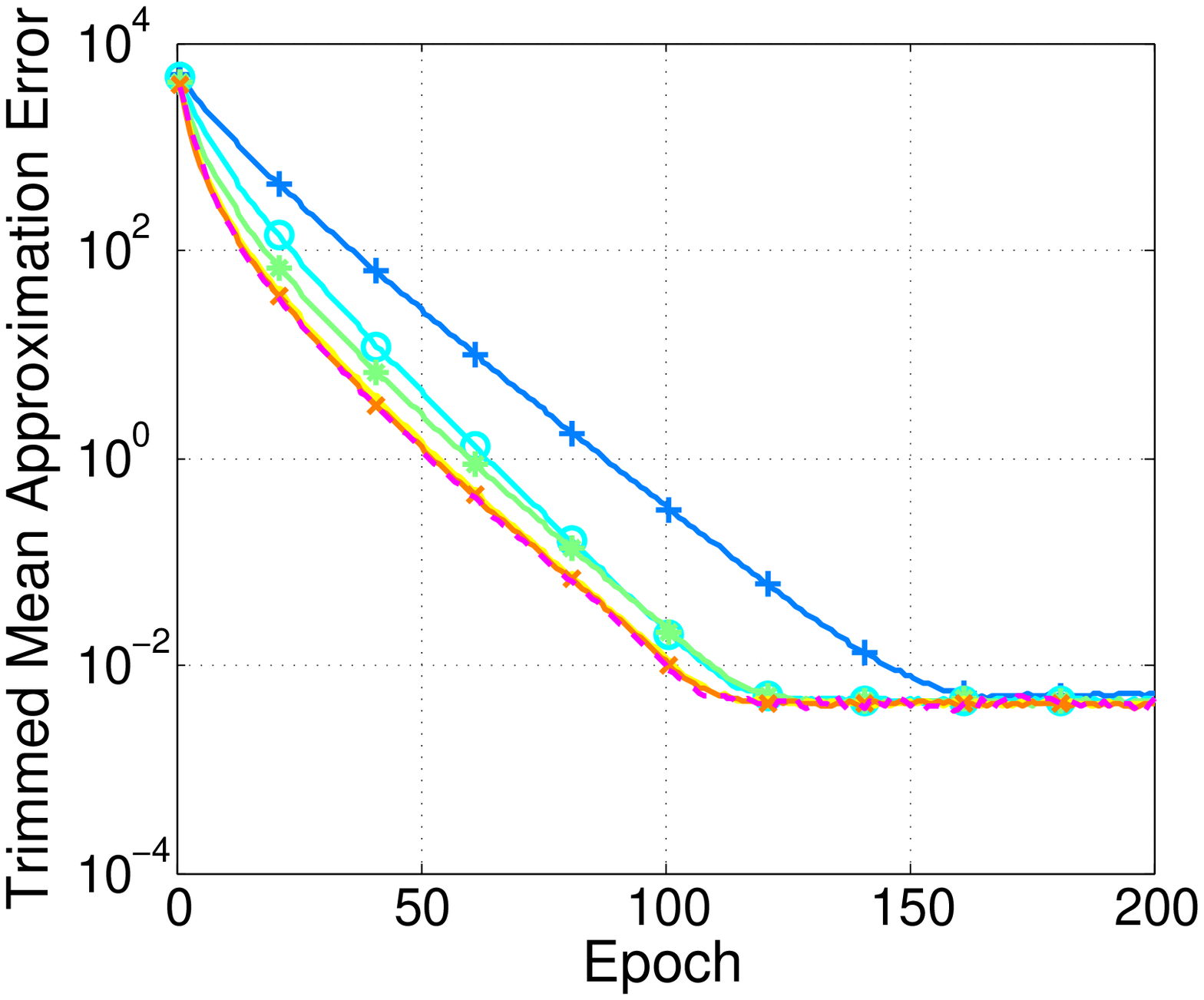} & 
 \includegraphics[height=2.25in,width=3.22in]{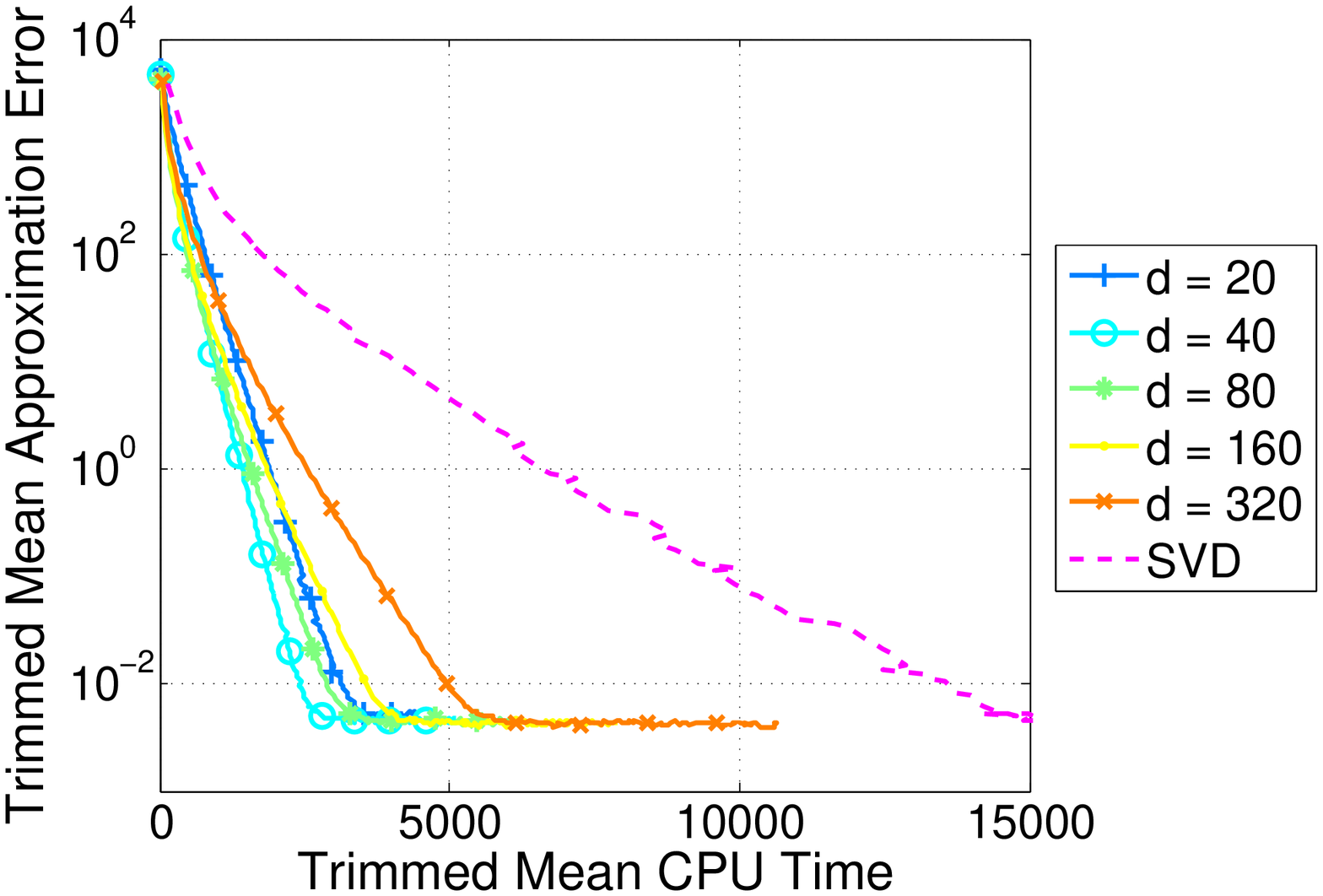} 
\end{tabular}
\caption{Low-Rank Matrix Recovery with Approximations: Trimmed mean recovery error as a function of epochs (left) and runtime (right) for various over-sampling factors $d$ using the StoGradMP algorithm.  Performance using full SVD computation shown as dashed line.  \label{fig14}}
\end{figure}

\section{Conclusion}
\label{sec::conclusion}

We study in this paper two stochastic algorithms to solve a possibly non-convex optimization problem with the constraint that the solution has a simple representation with respect to a predefined atom set. This type of optimization has found tremendous applications in signal processing, machine learning, and statistics such as sparse signal recovery and low-rank matrix estimation. Our proposed algorithms, called StoIHT and StoGradMP, have their roots back to the celebrated IHT and CoSaMP algorithms, from which we have made several significant extensions. The first extension is to transfer algorithmic ideas of IHT and CoSaMP to the stochastic setting and the second extension is the allowance of approximate projections at each iteration of the algorithms. More importantly, we theoretically prove that the stochastic versions with inexact projections enjoy the same linear convergence rate as their deterministic counterparts. We also show that the algorithms behave predictably even when the gradients are contaminated by noise. Experimentally, stochastic approaches have shown particular advantages over the deterministic counterparts in many problems of interest such as linear regression and matrix recovery.

\section{Proofs}
\label{sec::proofs}

\subsection{Consequences of the $\oper D$-RSC and $\oper D$-RSS}

\noindent The first corollary provides a useful upper bound for the gradient, which we call co-coercivity.
\begin{cor}
\label{cor::consequences of RSS}
Assume the function $f(w)$ satisfies the $\oper D$-RSS property, then
\begin{equation}
\label{inq::33}
\inner{w'-w,\nabla f(w') - \nabla f(w)} \leq \rho^+_s \norm{w'-w}_2^2
\end{equation}
for all vectors $w$ and $w'$ of size $n$ such that $|\supp_{\oper D}(w) \cup \supp_{\oper D}(w')| \leq s $. In addition, let $\Omega = \supp_{\oper D} (w) \cup \supp_{\oper D}(w')$; then we have
\begin{equation}
\label{inq::coercivity}
[\textbf{Co-coercivity}] \quad\quad \norm{\oper P_{\Omega}(\oper \nabla f(w') - \nabla f(w))}^2_2 \leq \rho^+_s \inner{w'-w,\nabla f(w') - \nabla f(w)}.
\end{equation}

\end{cor}
\begin{proof}
From the definition of $\oper D$-RSS, we can show that
\begin{equation}
\label{inq::D-RSS variant}
f(w') - f(w) - \inner{\nabla f(w), w'-w} \leq \frac{\rho^+_s}{2} \norm{w'-w}_2^2.
\end{equation}
Similarly, interchanging the role of $w$ and $w'$, we have
$$
f(w) - f(w') - \inner{\nabla f(w'), w-w'} \leq \frac{\rho^+_s}{2} \norm{w'-w}_2^2.
$$
Taking the summation of these two inequalities leads to the first claim. 

To prove the second claim, we define a function $G(x) \triangleq f(x) - \inner{\nabla f(w), x}$, it is easy to see that for any $x$ and $y$ with $\supp_{\oper D}(x) \cup \supp_{\oper D}(y) \in \Omega$, we have
$$
\norm{\nabla G(x) - \nabla G(y)}_2 = \norm{\nabla f(x) - \nabla f(y)}_2 \leq \rho^+_s \norm{x-y}_2.
$$
This implies that $G(x)$ has $\oper D$-RSS with constant $\rho^+_s$. In particular, we get a similar inequality as in (\ref{inq::D-RSS variant})
\begin{equation}
\label{inq::D-RSS for G(x)}
G(x) - G(y) - \inner{\nabla G(y), x-y} \leq \frac{\rho^+_s}{2} \norm{x-y}_2^2.
\end{equation}

\noindent We also observe that
$$
G(x) - G(w) = f(x) - f(w) - \inner{\nabla f(w), x-w} \geq 0
$$ 
for all $x$ such that $\supp_{\oper D}(x) \in \Omega$. Let $x \triangleq w' - \frac{1}{\rho^+_s} \oper P_{\Omega} \nabla G(w')$, then it is clear that $\supp_{\oper D}(x) \in \Omega$. Thus, by $\oper D$-RSS property of $G(x)$, we have
\begin{equation}
\begin{split}
\nonumber
G(w) &\leq G \left(w' - \frac{1}{\rho^+_s} \oper P_{\Omega} \nabla G(w') \right) \\
&\leq G(w') + \inner{\nabla G(w'), - \frac{1}{\rho^+_s} \oper P_{\Omega} \nabla G(w')} + \frac{1}{2\rho^+_s} \norm{\oper P_{\Omega} \nabla G(w')}_2^2 \\
&= G(w') - \frac{1}{2\rho^+_s} \norm{\oper P_{\Omega} \nabla G(w')}_2^2.
\end{split}
\end{equation}
Replacing the function $G(x)$ in this equality we get
$$
\frac{1}{2\rho^+_s} \norm{\oper P_{\Omega} (\nabla f(w') - \nabla f(w))}_2^2 \leq f(w') - f(w) - \inner{\nabla f(w), w'-w}.
$$
The claim follows by adding the two inequalities with $w$ and $w'$ interchanged.
\end{proof}

\noindent The following corollary provides the lower bound for the gradient.
\begin{cor}
\label{cor::consequences of RSC}
Assume the function $F(w)$ satisfies the $\oper D$-RSC, then
\begin{equation}
\label{inq::37}
\rho^-_{s} \norm{w'-w}_2^2 \leq \inner{w'-w,\nabla F(w') - \nabla F(w)}
\end{equation}
for all $w$ and $w'$ such that $|\supp_{\oper D}(w) \cup \supp_{\oper D}(w')| \leq s$.
\end{cor}
\begin{proof}
From the $\oper D$-RSC assumption, we can write
$$
F(w') - F(w) - \inner{\nabla F(w), w'-w} \geq \frac{\rho^-_s}{2} \norm{w'-w}_2^2.
$$
Swapping $w$ and $w'$, we also have
$$
F(w) - F(w') - \inner{\nabla F(w'),w-w'} \geq \frac{\rho^-_s}{2} \norm{w'-w}_2^2.
$$
The result follows by adding the two inequalities. 
\end{proof}

\noindent The next corollary provides key estimates for our convergence analysis. Recall that we assume $\{f_i(w) \}_{i=1}^M$ satisfies the $\oper D$-RSS and $F(w) = \sum_{i=1}^M f_i(w)$ satisfies the $\oper D$-RSC. 
\begin{cor}
\label{cor::StoIHT 1st corollary}
Let $i$ be an index selected with probability $p(i)$ from the set $[M]$. For any fixed sparse vectors $w$ and $w'$, let $\Omega$ be a set such that $\supp_{\oper D}(w) \cup \supp_{\oper D}(w') \in \Omega$ and denote $s = |\Omega|$. We have
\begin{equation}
\label{inq::1st key observation}
\E_i \norm{w' - w - \frac{\gamma}{Mp(i)} \oper P_{\Omega} \left(\nabla f_i(w') - \nabla f_i(w) \right)}_2 \leq \sqrt{1- (2 - \gamma \alpha_s) \gamma \rho^-_s } \norm{w'-w}_2
\end{equation}
where we define $\alpha_s \triangleq \max_i \frac{\rho^+_s(i)}{Mp(i)}$. In addition, we have
\begin{equation}
\label{inq::2nd key observation}
\E_i \norm{w' - w - \frac{\gamma}{Mp(i)} \left(\nabla f_i(w') - \nabla f_i(w) \right)}_2 \leq \sqrt{1 + \gamma^2 \alpha_s \overline{\rho}^+_s - 2\gamma \rho^-_s } \norm{w'-w}_2,
\end{equation}
where $\overline{\rho}^+_s \triangleq \frac{1}{M} \sum_i \rho^+_s(i)$.
\end{cor}
The difference between the estimates (\ref{inq::1st key observation}) and (\ref{inq::2nd key observation}) is the additional term $\norm{\oper P_{\Omega^c} \left(\nabla f_i(w') - \nabla f_i(w) \right)}_2$ with $\Omega^c = [n]\backslash \Omega$ appearing in (\ref{inq::2nd key observation}). 
\begin{proof}
We will use the co-coercivity property that appeared in inequality (\ref{inq::coercivity}) in Corollary \ref{cor::consequences of RSS}. We have
\begin{equation*}
\begin{split}
\nonumber
&\E_i\norm{w' - w - \frac{\gamma}{Mp(i)} \oper P_{\Omega} \left(\nabla f_i(w') - \nabla f_i(w) \right)}^2_2 \\
&= \norm{w' - w}_2^2 +  \E_i \frac{\gamma^2}{(Mp(i))^2} \norm{\oper P_{\Omega} \left( \nabla f_i(w') - \nabla f_i(w) \right)}_2^2 \\
&\quad- 2 \gamma \E_i\inner{w' - w,   \oper P_{\Omega} \frac{1}{Mp(i)} \left(\nabla f_i(w') - \nabla f_i(w) \right)} \\
&\leq \norm{w' - w}_2^2 + \gamma^2 \E_i \frac{\rho^+_{s}(i)}{(Mp(i))^2} \inner{w'-w, \nabla f_i(w') - \nabla f_i(w) } \\
&\quad- 2 \gamma \E_i\inner{w' - w,   \frac{1}{Mp(i)} \left(\nabla f_i(w') - \nabla f_i(w) \right) } \\
&\leq \norm{w' - w}_2^2 + \gamma^2 \max_i \frac{\rho^+_{s}(i)}{Mp(i)} \E_i \inner{w'-w, \frac{1}{Mp(i)} \left(\nabla f_i(w') - \nabla f_i(w) \right)} \\
&\quad- 2 \gamma \E_i\inner{w' - w,   \frac{1}{Mp(i)} \left(\nabla f_i(w') - \nabla f_i(w) \right) } \\
&= \norm{w' - w}_2^2 - \left(2\gamma - \gamma^2 \max_i \frac{\rho^+_{s}(i)}{Mp(i)} \right) \E_i\inner{w' - w,   \frac{1}{Mp(i)} \left(\nabla f_i(w') - \nabla f_i(w) \right) } \\
&= \norm{w' - w}_2^2 - (2\gamma - \gamma^2 \alpha_s) \inner{w' - w,   \nabla F(w') - \nabla F(w) }  \\
&\leq \norm{w' - w}_2^2 - (2\gamma - \gamma^2 \alpha_s) \rho^-_s \norm{w' - w}_2^2,
\end{split}
\end{equation*}
where the first inequality follows from (\ref{inq::coercivity}) and the last inequality follows from (\ref{inq::37}). Applying the known result $(\E Z)^2 \leq \E Z^2$ completes the proof of (\ref{inq::1st key observation}).

The proof of (\ref{inq::2nd key observation}) is similar to that of (\ref{inq::1st key observation}), except now we are not able to apply the co-coercivity inequality. Expanding the left hand side and applying the definition of $\oper D$-RSS together with the inequality (\ref{inq::37}), we derive
\begin{equation}
\begin{split}
\nonumber
&\E_i\norm{w' - w - \frac{\gamma}{Mp(i)} \left(\nabla f_i(w') - \nabla f_i(w) \right)}^2_2 \\
&= \norm{w' - w}_2^2 +  \E_i \frac{\gamma^2}{(Mp(i))^2} \norm{\nabla f_i(w') - \nabla f_i(w) }_2^2 \\
&\quad- 2 \gamma \E_i\inner{w' - w, \frac{1}{Mp(i)} \left(\nabla f_i(w') - \nabla f_i(w) \right)} \\
&= \norm{w' - w}_2^2 + \E_i \frac{\gamma^2}{(Mp(i))^2} \norm{\nabla f_i(w') - \nabla f_i(w) }_2^2\\
&\quad- 2 \gamma \inner{w' - w, \nabla F(w') - \nabla F_i(w) } \\
&\leq \norm{w' - w}_2^2 + \E_i \frac{\gamma^2}{(Mp(i))^2} (\rho^+_s(i))^2 \norm{w'-w}_2^2 - 2 \gamma \rho^-_{s}  \norm{w'-w}_2^2. 
\end{split}
\end{equation}
 We further have
$$
 \E_i \frac{\gamma^2}{(Mp(i))^2} (\rho^+_s(i))^2 \leq \gamma^2 \max_i \frac{\rho^+_s(i)}{Mp(i)} \E_i \frac{\rho^+_s(i)}{Mp(i)} = \gamma^2 \alpha_s \sum_i \frac{\rho^+_s(i)}{Mp(i)} p(i) = \gamma^2 \alpha_s \overline{\rho}^+_s,
$$
where we recall that $\alpha_s = \max_i \frac{\rho^+_s(i)}{Mp(i)}$ and $\overline{\rho}^+_s = \frac{1}{M} \sum_i \rho^+_s(i)$. Substitute this result into the above inequality and then use the inequality $(\E Z)^2 \leq \E Z^2$ to complete the proof.
\end{proof}

\subsection{Proof of Theorem \ref{thm::StoIHT}}
\label{sub::proof of theorem 1}

\begin{proof} [Proof of Theorem \ref{thm::StoIHT}]
We notice that $\norm{w^{t+1} - b^t}_2 \leq \eta \norm{b^t_k - b^t}_2 \leq \eta \norm{w^{\star} - b^t}_2$ where $b^t_k$ is the best $k$-sparse approximation of $b^t$ with respect to $\oper D$. Thus,
$$
\norm{w^{t+1} - w^{\star} + w^{\star} - b^t}_2^2 \leq \eta^2 \norm{w^{\star} - b^t}_2^2.
$$

\noindent Expanding the left hand side of this inequality leads to
\begin{equation}
\begin{split}
\nonumber
\norm{w^{t+1} - w^{\star}}_2^2 &\leq 2 \inner{w^{t+1} - w^{\star}, b^t - w^{\star}} + (\eta^2-1) \norm{b^t-w^{\star}}_2^2 \\
&= 2 \inner{w^{t+1} - w^{\star}, w^t - w^{\star} - \frac{\gamma}{Mp(i_t)} \nabla f_{i_t}(w^t)} \\
&\quad+ (\eta^2-1) \norm{w^t - w^{\star} - \frac{\gamma}{Mp(i_t)} \nabla f_{i_t}(w^t)}_2^2 \\
&= 2 \inner{w^{t+1} - w^{\star}, w^t - w^{\star} - \frac{\gamma}{Mp(i_t)} \left( \nabla f_{i_t}(w^t) - \nabla f_{i_t}(w^{\star}) \right)} \\
&\quad- 2\inner{w^{t+1} - w^{\star}, \frac{\gamma}{Mp(i_t)} \nabla f_{i_t} (w^{\star})} + (\eta^2-1) \norm{w^t - w^{\star} - \frac{\gamma}{Mp(i_t)} \nabla f_{i_t}(w^t)}_2^2.
\end{split}
\end{equation}

\noindent Denote $\Omega = \supp_{\oper D}(w^{t+1}) \cup \supp_{\oper D}(w^t) \cup \supp_{\oper D}(w^{\star})$ and notice that $|\Omega| \leq 3k$, we get
\begin{equation}
\begin{split}
\nonumber
\norm{w^{t+1} - w^{\star}}_2^2 &\leq 2 \inner{w^{t+1} - w^{\star}, w^t - w^{\star} - \frac{\gamma}{Mp(i_t)}\oper P_{\Omega} \left( \nabla f_{i_t}(w^t) - \nabla f_{i_t}(w^{\star}) \right)} \\
&\quad- 2\inner{w^{t+1} - w^{\star}, \frac{\gamma}{Mp(i_t)} \oper P_{\Omega} \nabla f_{i_t} (w^{\star})} + (\eta^2-1) \norm{w^t - w^{\star} - \frac{\gamma}{Mp(i_t)} \nabla f_{i_t}(w^t)}_2^2 \\
&\leq 2 \norm{w^{t+1} - w^{\star}}_2 \\
&\quad\times \underbrace{\left( \norm{w^t - w^{\star} - \frac{\gamma}{Mp(i_t)} \oper P_{\Omega} \left( \nabla f_{i_t}(w^t) - \nabla f_{i_t}(w^{\star}) \right)}_2 + \norm{\frac{\gamma}{Mp(i_t)} \oper P_{\Omega} \nabla f_{i_t}(w^{\star})}_2 \right)}_u \\
&\quad+ \underbrace{(\eta^2-1) \norm{w^t - w^{\star} - \frac{\gamma}{Mp(i_t)} \nabla f_{i_t}(w^t)}_2^2 }_v.
\end{split}
\end{equation}

\noindent Solving this quadratic inequality $x^2 - 2ux - v \leq 0$ with $x = \norm{w^{t+1} - w^{\star}}_2$, we get $x \leq u + \sqrt{u^2+v} \leq 2u+ \sqrt{v}$. Substituting the expressions for $u$ and $v$ above, we arrive at
\begin{equation}
\begin{split}
\nonumber
\norm{w^{t+1} - w^{\star}}_2 &\leq 2\left( \norm{w^t - w^{\star} - \frac{\gamma}{Mp(i_t)} \oper P_{\Omega} \left( \nabla f_{i_t}(w^t) - \nabla f_{i_t}(w^{\star}) \right)}_2 + \norm{\frac{\gamma}{Mp(i_t)} \oper P_{\Omega} \nabla f_{i_t}(w^{\star})}_2 \right) \\
&\quad+ \sqrt{\eta^2-1} \norm{w^t - w^{\star} - \frac{\gamma}{Mp(i_t)} \nabla f_{i_t}(w^t)}_2 \\
&\leq 2\left( \norm{w^t - w^{\star} - \frac{\gamma}{Mp(i_t)} \oper P_{\Omega} \left( \nabla f_{i_t}(w^t) - \nabla f_{i_t}(w^{\star}) \right)}_2 + \norm{\frac{\gamma}{Mp(i_t)} \oper P_{\Omega} \nabla f_{i_t}(w^{\star})}_2 \right) \\
&\quad+ \sqrt{\eta^2-1} \left( \norm{w^t - w^{\star} - \frac{\gamma}{Mp(i_t)} \left( \nabla f_{i_t}(w^t) - \nabla f_{i_t}(w^{\star}) \right)}_2 + \norm{\frac{\gamma}{Mp(i_t)} \nabla f_{i_t}(w^{\star})}_2 \right).
\end{split}
\end{equation}

\noindent Denote $I_t$ as the set containing all indices $i_1, i_2,..., i_t$ randomly selected at or before step $t$ of the algorithm: $I_t = \{i_1,...,i_t\}$. It is clear that $I_t$ determines the solutions $w^1,...,w^{t+1}$. We also denote the conditional expectation $\E_{i_t | I_{t-1}} \norm{w^{t+1} - w^{\star}}_2 \triangleq \E_{i_t} (\norm{w^{t+1} - w^{\star}}_2| I_{t-1})$. Now taking the conditional expectation on both sides of the above inequality  we obtain
\begin{equation}
\begin{split}
\nonumber
&\E_{i_t | I_{t-1}} \norm{w^{t+1} - w^{\star}}_2 \\
&\leq 2\left( \E_{i_t | I_{t-1}} \norm{w^t - w^{\star} - \frac{\gamma}{Mp(i_t)} \oper P_{\Omega} \left( \nabla f_{i_t}(w^t) - \nabla f_{i_t}(w^{\star}) \right)}_2 + \E_{i_t | I_{t-1}} \norm{\frac{\gamma}{Mp(i_t)} \oper P_{\Omega} \nabla f_{i_t} (w^{\star})}_2 \right)  \\
&\quad+ \sqrt{\eta^2-1} \left( \E_{i_t | I_{t-1}} \norm{w^t - w^{\star} - \frac{\gamma}{Mp(i_t)} \left( \nabla f_{i_t} (w^t) - \nabla f_{i_t} (w^{\star}) \right)}_2 + \E_{i_t | I_{t-1}} \norm{\frac{\gamma}{Mp(i_t)} \nabla f_{i_t} (w^{\star})}_2 \right) .
\end{split}
\end{equation}

\noindent Conditioning on $I_{t-1}$, $w^t$ can be seen as a fixed vector. We apply the inequality (\ref{inq::1st key observation}) of Corollary \ref{cor::StoIHT 1st corollary} for the first term and (\ref{inq::2nd key observation}) for the third term, we get 
\begin{equation}
\begin{split}
\nonumber
&\E_{i_t | I_{t-1}} \norm{w^{t+1} - w^{\star}}_2 \\
&\leq 2 \sqrt{\left(1- (2\gamma - \gamma^2 \alpha_{3k}) \rho^-_{3k} \right)} \norm{w^t-w^{\star}}_2 + 2 \frac{\gamma}{\min_{i_t} Mp(i_t)} \E_{i_t}\norm{\oper P_{\Omega} \nabla f_{i_t} (w^{\star})}_2 \\
&\quad+ \sqrt{\eta^2-1}  \sqrt{1 + \gamma^2 \alpha_{3k} \overline{\rho}^+_{3k} - 2\gamma \rho^-_{3k} } \norm{w^t-w^{\star}}_2 + \sqrt{\eta^2-1}  \frac{\gamma}{\min_{i_t} Mp(i_t)} \E_{i_t} \norm{\nabla f_{i_t} (w^{\star})}_2 \\
&= \left( 2 \sqrt{\left(1- (2\gamma - \gamma^2 \alpha_{3k}) \rho^-_{3k} \right)} + \sqrt{(\eta^2-1)\left(1 + \gamma^2 \alpha_{3k} \overline{\rho}^+_{3k} - 2\gamma \rho^-_{3k} \right)} \right) \norm{w^t - w^{\star}}_2 \\
&\quad+ \frac{\gamma}{\min_{i_t} Mp(i_t)} \left( 2 \E_{i_t} \norm{\oper P_{\Omega} \nabla f_{i_t} (w^{\star})}_2 + \sqrt{\eta^2 - 1} \E_{i_t} \norm{\nabla f_{i_t} (w^{\star})}_2 \right) \\
&\leq \kappa \norm{w^t -w^{\star}}_2 +  \sigma,
\end{split}
\end{equation}
where $\kappa$ and $\sigma_{w^{\star}}$ are defined in Theorem \ref{thm::StoIHT}. Taking the expectation on both sides with respect to $I_{t-1}$ yields
$$
\E_{I_t} \norm{w^{t+1}-w^{\star}}_2 \leq \kappa \E_{I_{t-1}} \norm{w^t-w^{\star}} + \sigma.
$$

\noindent Applying this result recursively over $t$ iterations yields the desired result:
\begin{equation}
\begin{split}
\nonumber
\E_{I_t} \norm{w^{t+1}-w^{\star}}_2 &\leq \kappa^{t+1} \norm{w^0 - w^{\star}}_2 + \sum_{j=0}^t \kappa^j \sigma \\
&\leq \kappa^{t+1} \norm{w^0 - w^{\star}}_2  + \frac{1}{1-\kappa} \sigma.
\end{split}
\end{equation}
\end{proof}

\subsection{Proof of Theorem \ref{thm::StoGradMP}}
\label{subsec::Proof of StoGradMP theorem}

The proof of Theorem \ref{thm::StoGradMP} is a consequence of the following three lemmas. Denote $I_t$ as the set containing all indices $i_1, i_2,..., i_t$ randomly selected at or before step $t$ of the algorithm: $I_t = \{i_1,...,i_t\}$ and denote the conditional expectation $\E_{i_t | I_{t-1}} \norm{w^{t+1} - w^{\star}}_2 \triangleq \E_{i_t} (\norm{w^{t+1} - w^{\star}}_2| I_{t-1})$.

\begin{lem}
\label{lem::bound l2 w^(t+1)-w*}
The recovery error at the $(t+1)$-th iteration is upper bounded by
$$
\norm{w^{t+1} - w^{\star}}_2 \leq (1+\eta_2) \norm{b^t - w^{\star}}_2.
$$
\end{lem}

\begin{lem}
\label{lem::bound l2 b^t - w*}
Denote $\widehat{\Gamma}$ as the set obtained from the $t$-th iteration and $i$ as the index selected randomly from $[M]$ with probability $p(i)$. We have,
\begin{equation}
\begin{split}
\nonumber
\E_{I_t} \norm{b^t - w^{\star}}_2 &\leq \sqrt{\frac{\alpha_{4k}}{\rho^-_{4k}}} \E_{I_t} \norm{P_{\widehat{\Gamma}^c} (b^t - w^{\star})}_2 + \sigma_1
\end{split}
\end{equation}
where $\alpha_k = \max_i \frac{\rho^+_k(i)}{Mp(i)}$ and 
$$
\sigma_1 \triangleq \frac{3}{\rho^-_{4k}} \frac{1}{\min_i Mp(i)} \max_{|\Omega| \leq 3k, i \in [M]}\norm{\oper P_{\Omega} \nabla f_i(w^{\star})}_2.
$$
\end{lem}

\begin{lem}
\label{lem::bound L2 P_Gamma^c b^t-w*}
Denote $\widehat{\Gamma}$ as the set obtained from the $t$-th iteration. Then,
\begin{equation}
\begin{split}
\E_{i_t} \norm{\oper P_{\widehat{\Gamma}^c} (b^t - w^{\star})}_2  &\leq \left( \max_i \sqrt{Mp(i)} \sqrt{\frac{\frac{2\eta_1^2-1}{\eta^2_1}\rho^+_{4k} - \rho^-_{4k}}{\rho^-_{4k}}}  + \frac{\sqrt{\eta^2_1-1}}{\eta_1} \right) \norm{w^t-w^{\star}}_2 + \sigma_2,
\end{split}
\end{equation}
where
$$
\sigma_2 \triangleq \frac{2 \max_{i \in [M]} p(i)}{\rho^-_{4k} \min_{i \in [M]} p(i)} \max_{|\Omega| \leq 4k, i \in [M]}\norm{\oper P_{\Omega} \nabla f_{i} (w^{\star})}_2.
$$
\end{lem}

We are now able to prove Theorem \ref{thm::StoGradMP}. We have a series of inequalities that follow from the above lemmas, 
\begin{equation}
\begin{split}
&\E_{I_t} \norm{w^{t+1} - w^{\star}}_2 \\
&\leq (1+\eta_2) \E_{I_t} \norm{b^t-w^{\star}}_2  \quad\quad \text{(Lemma \ref{lem::bound l2 w^(t+1)-w*})}\\
&\leq (1+\eta_2) \sqrt{\frac{\alpha_{4k}}{\rho^-_{4k}}} \E_{I_t} \norm{P_{\widehat{\Gamma}^c} (b^t - w^{\star})}_2 + (1+\eta_2) \sigma_1  \quad\quad \text{(Lemma \ref{lem::bound l2 b^t - w*}) }\\
&\leq (1+\eta_2) \sqrt{\frac{\alpha_{4k}}{\rho^-_{4k}}} \left( \max_i \sqrt{Mp(i)} \sqrt{\frac{ \frac{2\eta^2_1-1}{\eta^2_1}\rho^+_{4k} - \rho^-_{4k}}{\rho^-_{4k}}} + \frac{\sqrt{\eta^2_1-1}}{\eta_1} \right)  \E_{I_{t-1}} \norm{w^t-w^{\star}}_2 \\
&+ (1+\eta_2)\left( \sqrt{\frac{\alpha_{4k}}{\rho^-_{4k}}} \sigma_2 + \sigma_1 \right),
\end{split}
\end{equation}
where the last inequality follows from Lemma \ref{lem::bound L2 P_Gamma^c b^t-w*}. Replacing the definition of $\kappa$ in (\ref{eqt::kappa of StoGradMP}) and noticing that $\sigma_{w^{\star}}$ defined in (\ref{eqt::sigma of StoGradMP}) is greater than the second term of the last equation (it is due to $\max_{|\Omega| \leq 4k, i \in [M]}\norm{\oper P_{\Omega} \nabla f_{i} (w^{\star})}_2 \geq \max_{|\Omega| \leq 3k, i \in [M]}\norm{\oper P_{\Omega} \nabla f_{i} (w^{\star})}_2$), we arrive at
\begin{equation}
\E_{I_t} \norm{w^{t+1} - w^{\star}}_2 \leq \kappa \E_{I_{t-1}}\norm{w^t-w^{\star}}_2 + \sigma_{w^\star}.
\end{equation}
Applying this inequality recursively $t$ times will complete the proof.

\vspace{3mm}
\noindent To the end of this section, we prove three lemmas stated above.

\begin{proof} [Proof of Lemma \ref{lem::bound l2 w^(t+1)-w*}]
Recall that $b^t$ is the vector obtained from the $t$-th iteration. From the algorithm, we have
$$
\norm{w^{t+1} - b^t}_2 \leq \eta_2 \norm{b^t_k - b^t}_2 \leq \eta_2 \norm{w^{\star} - b^t}_2
$$
where $b^t_k$ is the best $k$-sparse approximation of $b^t$ with respect to the set $\oper D$. We thus have
\begin{equation}
\begin{split}
\norm{w^{t+1} - w^{\star}}_2 &\leq \norm{(w^{t+1} - b^t) + (b^t -w^{\star})}_2  \\
&\leq \norm{w^{t+1} - b^t}_2 + \norm{b^t - w^{\star}}_2 \leq (1+\eta_2) \norm{b^t - w^{\star}}_2.
\end{split}
\end{equation}
\end{proof}

\begin{proof} [Proof of Lemma \ref{lem::bound l2 b^t - w*}]
Denote the set $\oper C_{\widehat{\Gamma}} \triangleq \{ w: w = \sum_{j \in \widehat{\Gamma}} \alpha_j d_j \}$. It is clear that $\oper C_{\widehat{\Gamma}}$ is a convex set, so the estimation step can be written as
$$
b^t = \argmin_w F(w)  \quad\text{such that} \quad w \in \oper C_{\widehat{\Gamma}}.
$$

\noindent Optimization theory states that (Proposition 4.7.1 of \cite{DNO_optimization_2003_B}) 
\begin{equation}
\nonumber
\inner{\nabla F(b^t), b^t - z} \leq 0 \quad \text{for all } z \in \oper C_{\widehat{\Gamma}}.
\end{equation}

\noindent Put differently, we have
\begin{equation}
\nonumber
\inner{\nabla F(b^t), \oper P_{\widehat{\Gamma}}(b^t - z)} \leq 0 \quad \text{for all } z.
\end{equation}
Denote by $i$ an index selected randomly from $[M]$ with probability $p(i)$ and independent from  all the random indices $i_t$ and recall that $\nabla F(b^t) = \E_i \frac{1}{Mp(i)} \nabla f_i (b^t)$. The above inequality can be read as
\begin{equation}
\label{inq::optimization condition}
0 \geq \inner{\E_i \frac{1}{Mp(i)} \nabla f_i (b^t), \oper P_{\widehat{\Gamma}}(b^t - z)} = \E_i \inner{\frac{1}{Mp(i)} \oper P_{\widehat{\Gamma}} \nabla f_i (b^t), \oper P_{\widehat{\Gamma}} (b^t - z)} \quad \text{for all } z.
\end{equation}

\noindent We first derive the upper bound of $\norm{\oper P_{\widehat{\Gamma}}(b^t - w^{\star})}_2$. For any $\gamma > 0$, we have
\begin{equation}
\begin{split}
\nonumber
\norm{\oper P_{\widehat{\Gamma}}(b^t - w^{\star})}_2^2 &=  \inner{\oper P_{\widehat{\Gamma}}(b^t - w^{\star}), b^t - w^{\star} }\\
&= \inner{\oper P_{\widehat{\Gamma}} (b^t - w^{\star}), b^t - w^{\star} - \E_i \frac{\gamma}{Mp(i)}\oper P_{\widehat{\Gamma}}(\nabla f_i(b^t) - \nabla f_i(w^{\star}))} \\
&\quad+ \inner{\oper P_{\widehat{\Gamma}} (b^t - w^{\star}), \E_i \frac{\gamma}{Mp(i)} \oper P_{\widehat{\Gamma}} \nabla f_i(b^t)} -  \inner{\oper P_{\widehat{\Gamma}}(b^t -  w^{\star}), \E_i \frac{\gamma}{Mp(i)}\oper P_{\widehat{\Gamma}} \nabla f_i(w^{\star}))} \\
&= \E_i \inner{\oper P_{\widehat{\Gamma}} (b^t - w^{\star}), b^t - w^{\star} - \frac{\gamma}{Mp(i)}\oper P_{\widehat{\Gamma}}(\nabla f_i(b^t) - \nabla f_i(w^{\star}))} \\
&\quad+ \E_i \inner{\oper P_{\widehat{\Gamma}} (b^t - w^{\star}),  \frac{\gamma}{Mp(i)} \oper P_{\widehat{\Gamma}} \nabla f_i(b^t)} -  \E_i \inner{\oper P_{\widehat{\Gamma}}(b^t -  w^{\star}), \frac{\gamma}{Mp(i)}\oper P_{\widehat{\Gamma}} \nabla f_i(w^{\star}))} \\
&\leq \norm{\oper P_{\widehat{\Gamma}} (b^t - w^{\star})}_2 \E_i \norm{b^t - w^{\star} - \frac{\gamma}{Mp(i)} \oper P_{\widehat{\Gamma}}(\nabla f_i(b^t) - \nabla f_i(w^{\star}))}_2 \\
&\quad+ \norm{\oper P_{\widehat{\Gamma}}(b^t -  w^{\star})}_2 \E_i \frac{\gamma}{Mp(i)} \norm{\oper P_{\widehat{\Gamma}} \nabla f_i(w^{\star}))}_2
\end{split}
\end{equation} 
where the inequality follows from (\ref{inq::optimization condition}) and the Cauchy-Schwarz inequality. Canceling the common term in both sides, we derive
\begin{equation}
\nonumber
\norm{\oper P_{\widehat{\Gamma}}(b^t - w^{\star})}_2 \leq \E_i \norm{b^t - w^{\star} - \frac{\gamma}{Mp(i)} \oper P_{\widehat{\Gamma}}(\nabla f_i(b^t) - \nabla f_i(w^{\star}))}_2 + \E_i \frac{\gamma}{Mp(i)} \norm{\oper P_{\widehat{\Gamma}} \nabla f_i(w^{\star}))}_2.
\end{equation}

\noindent We bound the first term of the right-hand side. For a fixed realization of the random vector $b^t$, we apply Corollary \ref{cor::StoIHT 1st corollary} to obtain
\begin{equation}
\begin{split}
\E_{i} \norm{b^t - w^{\star} - \frac{\gamma}{Mp(i)} \oper P_{\widehat{\Gamma}}(\nabla f_i(b^t) - \nabla f_i(w^{\star}))}_2 \leq \sqrt{\left(1- (2\gamma - \gamma^2 \alpha_{4k}) \rho^-_{4k} \right)} \norm{b^t-w^{\star}}_2 .
\end{split}
\end{equation}

\noindent Applying this result to the above inequality and taking the expectation with respect to $i_t$ yields
\begin{equation}
\begin{split}
\norm{\oper P_{\widehat{\Gamma}}(b^t - w^{\star})}_2  \leq \sqrt{\left(1- (2\gamma - \gamma^2 \alpha_{4k}) \rho^-_{4k} \right)} \norm{b^t-w^{\star}}_2 + \frac{\gamma}{\min_i Mp(i)} \E_i \norm{\oper P_{\widehat{\Gamma}} \nabla f_i(w^{\star}))}_2,
\end{split}
\end{equation}

\noindent We now apply this inequality to get
\begin{equation}
\begin{split}
\norm{b^t - w^{\star}}_2^2 &= \norm{\oper P_{\widehat{\Gamma}}(b^t - w^{\star})}_2^2 + \norm{P_{\widehat{\Gamma}^c} (b^t - w^{\star}) }_2^2 \\
&\leq  \left( \sqrt{\left(1- (2\gamma - \gamma^2 \alpha_{4k}) \rho^-_{4k} \right)} \norm{b^t-w^{\star}}_2 + \frac{\gamma}{\min_i Mp(i)} \E_i \norm{\oper P_{\widehat{\Gamma}} \nabla f_i(w^{\star}))}_2 \right)^2 \\
&\quad+ \norm{\oper P_{\widehat{\Gamma}^c} (b^t - w^{\star})}_2^2.
\end{split}
\end{equation}

\noindent Solving the quadratic polynomial $a x^2 - 2bx - c \leq 0$ with $x = \norm{b^t - w^{\star}}_2$, 
$a = (2\gamma - \gamma^2 \alpha_{4k})\rho^-_{4k}$, $b = \sqrt{\left(1- (2\gamma - \gamma^2 \alpha_{4k}) \rho^-_{4k} \right)}  \frac{\gamma}{\min_i Mp(i)} \E_i \norm{\oper P_{\widehat{\Gamma}} \nabla f_i(w^{\star}))}_2$, and $c = (\frac{\gamma}{\min_i Mp(i)} \E_i \norm{\oper P_{\widehat{\Gamma}} \nabla f_i(w^{\star}))}_2)^2 + \norm{\oper P_{\widehat{\Gamma}^c} (b^t - w^{\star})}_2^2$, we get
$$
\norm{b^t - w^{\star}}_2 \leq \frac{b + \sqrt{b^2+ac}}{a} \leq \sqrt{\frac{c}{a}} + \frac{2b}{a}.
$$
Replacing these quantities $a$, $b$, and $c$ yields 
\begin{equation}
\begin{split}
\nonumber
\norm{b^t - w^{\star}}_2 &\leq \frac{1}{\sqrt{ (2\gamma - \gamma^2 \alpha_{4k}) \rho^-_{4k}}}  \norm{P_{\widehat{\Gamma}^c} (b^t - w^{\star})}_2 \\
&\quad+ \left( \frac{1}{\sqrt{ (2\gamma - \gamma^2 \alpha_{4k}) \rho^-_{4k}}} + \frac{2\sqrt{\left(1- (2\gamma - \gamma^2 \alpha_{4k}) \rho^-_{4k} \right)}}{ (2\gamma - \gamma^2 \alpha_{4k}) \rho^-_{4k}}\right)  \frac{\gamma}{\min_i Mp(i)} \E_i \norm{\oper P_{\widehat{\Gamma}} \nabla f_i(w^{\star}))}_2 \\
&\leq \frac{1}{\sqrt{ (2\gamma - \gamma^2 \alpha_{4k}) \rho^-_{4k}}}  \norm{P_{\widehat{\Gamma}^c} (b^t - w^{\star})}_2 + \frac{3}{ (2\gamma - \gamma^2 \alpha_{4k}) \rho^-_{4k}}  \frac{\gamma}{\min_i Mp(i)} \E_i \norm{\oper P_{\widehat{\Gamma}} \nabla f_i(w^{\star}))}_2.
\end{split}
\end{equation}
Optimizing $\gamma$ that maximizes ($2\gamma - \gamma^2 \alpha_{4k}$), we get $\gamma = \frac{1}{\alpha_{4k}}$. Plugging this value into the above inequality and taking the expectation with respect to $I_t$ (notice that the random variable $b^t$ is determined by random indices $i_0,...,i_t$)
\begin{equation}
\begin{split}
\nonumber
\E_{I_t} \norm{b^t - w^{\star}}_2 &\leq \sqrt{\frac{\alpha_{4k}}{\rho^-_{4k}}} \E_{I_t} \norm{P_{\widehat{\Gamma}^c} (b^t - w^{\star})}_2 + \frac{3}{\rho^-_{4k}} \frac{1}{\min_i Mp(i)} \E_{i, I_t} \norm{\oper P_{\widehat{\Gamma}} \nabla f_i(w^{\star}))}_2.
\end{split}
\end{equation}
The proof is completed. 

\end{proof}

\begin{proof} [Proof of Lemma \ref{lem::bound L2 P_Gamma^c b^t-w*}]
Since $b^t$ and $w^t$ are in $\text{span}({\oper D_{\widehat{\Gamma}}})$, we have $\oper P_{\widehat{\Gamma}^c} b^t = 0$ and $\oper P_{\widehat{\Gamma}^c} w^t = 0$. Therefore,
\begin{equation}
\label{inq::1st intemediate result}
\norm{\oper P_{\widehat{\Gamma}^c} (b^t - w^{\star})}_2 = \norm{\oper P_{\widehat{\Gamma}^c} (w^t - w^{\star})}_2 \leq \norm{\oper P_{{\Gamma}^c} (w^t - w^{\star})}_2 = \norm{\Delta - \oper P_{\Gamma} \Delta}_2,
\end{equation}
where we denote $\Delta \triangleq w^{\star} - w^t$. The goal is to estimate $\norm{\Delta - \oper P_{\Gamma} \Delta}_2$. Let $R \triangleq \supp_{\oper D} (\Delta)$ and apply the $\oper D$-RSC, we have
\begin{equation}
\begin{split}
\label{inq::proof Lemma 3-1st inequality}
F(w^{\star}) - F(w^t) - \frac{\rho^-_{4k}}{2} \norm{w^{\star} - w^t}_2^2 &\geq \inner{\nabla F(w^t),w^{\star} - w^t} \\
&= \E_{i_t} \inner{\frac{1}{Mp(i_t)} \nabla f_{i_t}(w^t),\Delta} \\
&= \E_{i_t} \inner{\frac{1}{Mp(i_t)}  \oper P_R \nabla f_{i_t}(w^t),\Delta} \\
&\geq - \E_{i_t} \norm{\frac{1}{Mp(i_t)}  \oper P_R \nabla f_{i_t}(w^t)}_2 \norm{\Delta}_2.
\end{split} 
\end{equation}
The right-hand side can be lower bounded by applying inequality (\ref{eqt::approximation consequence 2}), which yields $\norm{\oper P_R \nabla f_{i_t}(w^t)}_2 \leq \norm{\oper P_{\Gamma} \nabla f_{i_t}(w^t)}_2 + \frac{\sqrt{\eta^2_1-1}}{\eta_1} \norm{\oper P_{\Gamma^c} \nabla f_{i_t}(w^t)}_2$. We now apply this observation to the above inequality. Denote $z \triangleq -\frac{\oper P_{\Gamma} \nabla f_{i_t}(w^t)}{\norm{\oper P_{\Gamma} \nabla f_{i_t}(w^t)}_2} \norm{\Delta}_2$ and $x \triangleq \frac{\oper P_{\Gamma^c} \nabla f_{i_t}(w^t)}{\norm{\oper P_{\Gamma^c} \nabla f_{i_t}(w^t)}_2} \norm{\Delta}_2$, we have
\begin{equation}
\begin{split}
\label{inq::proof Lemma 3-2nd inequality}
&- \E_{i_t} \norm{\frac{1}{Mp(i_t)}  \oper P_R \nabla f_{i_t}(w^t)}_2 \norm{\Delta}_2 \\
&\geq -\E_{i_t} \norm{\frac{1}{Mp(i_t)}  \oper P_{\Gamma} \nabla f_{i_t}(w^t)}_2 \norm{\Delta}_2 -  \frac{\sqrt{\eta^2_1-1}}{\eta_1} \E_{i_t} \norm{\frac{1}{Mp(i_t)} \oper P_{\Gamma^c} \nabla f_{i_t}(w^t)}_2 \norm{\Delta}_2 \\
&=  \E_{i_t} \inner{\frac{1}{Mp(i_t)}  \oper P_{\Gamma} \nabla f_{i_t}(w^t), z} - \frac{\sqrt{\eta^2_1-1}}{\eta_1} \E_{i_t} \inner{\frac{1}{Mp(i_t)} \oper P_{\Gamma^c} \nabla f_{i_t}(w^t), x} \\
&=  \E_{i_t} \inner{\frac{1}{Mp(i_t)}  \nabla f_{i_t}(w^t), z} - \E_{i_t} \inner{\frac{1}{Mp(i_t)} \nabla f_{i_t}(w^t), \frac{\sqrt{\eta^2_1-1}}{\eta_1} x} \\
&= \E_{i_t} \inner{\frac{1}{Mp(i_t)} \nabla f_{i_t}(w^t), z - \frac{\sqrt{\eta^2_1-1}}{\eta_1} x},
\end{split} 
\end{equation}
where the second equality follows from $\supp_{\oper D}(z) = \Gamma$ and $\inner{\oper P_{\Gamma}r ,z} = \inner{r,\oper P_{\Gamma}z} = \inner{r,z}$. Denote $y \triangleq z - \frac{\sqrt{\eta^2_1-1}}{\eta_1} x$ and combine (\ref{inq::proof Lemma 3-1st inequality}) and (\ref{inq::proof Lemma 3-2nd inequality}) to arrive at
\begin{equation}
\label{inq::proof Lemma 3-3rd inequality}
F(w^{\star}) - F(w^t) - \frac{\rho^-_{4k}}{2} \norm{\Delta}_2^2 \geq \E_{i_t} \inner{\frac{1}{Mp(i_t)} \nabla f_{i_t}(w^t), y}.
\end{equation}

\noindent We now use the $\oper D$-RSS property to lower bound the right-hand side of the above inequality. Recall that from the definition of $\oper D$-RSS, we can show that
\begin{equation}
\begin{split}
\nonumber
\inner{\nabla f_{i_t}(w^t), y} &\geq f_{i_t}(w^t+ y) - f_{i_t}(w^t) - \frac{\rho^+_{4k}(i_t)}{2} \norm{y}_2^2.
\end{split}
\end{equation}
Multiply both sides with $\frac{1}{Mp(i_t)}$ and take the expectation  with respect to the index $i_t$ and recall that $\E_{i_t} \frac{1}{Mp(i_t)} f_{i_t} (w^t) = F(w^t)$, we have
$$
\E_{i_t} \inner{\frac{1}{Mp(i_t)}  \nabla f_{i_t}(w^t), y} \geq \E_{i_t} \frac{1}{Mp(i_t)}  f_{i_t }(w^t + y) - F(w^t) -  \frac{1}{2} \E_{i_t} \frac{\rho^+_{4k}(i_t)}{Mp(i_t)} \norm{y}_2^2.
$$

\noindent Combining with inequality (\ref{inq::proof Lemma 3-3rd inequality}) and removing the common terms yields
\begin{equation}
\begin{split}
\nonumber
\frac{1}{2} \E_{i_t} \frac{\rho^+_{4k}(i_t)}{Mp(i_t)} \norm{y}_2^2 - \frac{\rho^-_{4k}}{2} \norm{\Delta}_2^2 &\geq \E_{i_t} \frac{1}{Mp(i_t)}  f_{i_t }(w^t+ y) - F(w^{\star}) \\
&= \E_{i_t} \frac{1}{Mp(i_t)} \left( f_{i_t }(w^t+ y) - f_{i_t }(w^{\star}) \right) 
\end{split}
\end{equation}
where the equality follows from $F(w^{\star}) = \E_{i_t} \frac{1}{Mp(i_t)} f_{i_t }(w^{\star})$. Applying the $\oper D$-RSC one more time to the right-hand side and then taking the expectation, we get
\begin{equation}
\label{inq::proof Lemma 3-quadratic inequality}
\begin{split}
&\frac{1}{2} \E_{i_t} \frac{\rho^+_{4k}(i_t)}{Mp(i_t)} \norm{y}_2^2 - \frac{\rho^-_{4k}}{2} \norm{\Delta}_2^2 \\
&\geq \frac{\rho^-_{4k}}{2} \E_{i_t} \frac{1}{Mp(i_t)} \norm{w^t+ y-w^{\star}}_2^2 + \E_{i_t} \inner{\frac{1}{Mp(i_t)} \nabla f_{i_t}(w^{\star}), w^t+ y-w^{\star}} \\ 
&= \frac{\rho^-_{4k}}{2} \E_{i_t} \frac{1}{Mp(i_t)} \norm{\Delta - y}_2^2 + \E_{i_t} \inner{\frac{1}{Mp(i_t)} \oper P_{\Gamma \cup R} \nabla f_{i_t}(w^{\star}), y-\Delta}  \\
&\geq \frac{\rho^-_{4k}}{2} \E_{i_t} \frac{1}{Mp(i_t)} \norm{\Delta - y}_2^2 - \E_{i_t} \frac{1}{Mp(i_t)} \norm{\oper P_{\Gamma \cup R} \nabla f_{i_t} (w^{\star})}_2 \norm{\Delta - y}_2 \\
&\geq \frac{\rho^-_{4k}}{2\max_{i_t} Mp(i_t)} \E_{i_t} \norm{\Delta - y}_2^2 - \frac{\max_{i_t} \norm{\oper P_{\Gamma \cup R} \nabla f_{i_t} (w^{\star})}_2}{\min_{i_t} Mp(i_t)} \E_{i_t}\norm{\Delta - y}_2 \\
&\geq \frac{\rho^-_{4k}}{2\max_{i_t} Mp(i_t)} \left(\E_{i_t} \norm{\Delta - y}_2 \right)^2 - \frac{\max_{i_t} \norm{\oper P_{\Gamma \cup R} \nabla f_{i_t} (w^{\star})}_2}{\min_{i_t} Mp(i_t)} \E_{i_t}\norm{\Delta - y}_2 .
\end{split}
\end{equation}

\noindent Solving the quadratic inequality $au^2 - 2bu -c\leq 0$ with	 $u = \E_{i_t} \norm{\Delta - y}_2$, $a = \frac{\rho^-_{4k}}{\max_{i_t} Mp(i_t)}$, $b = \frac{\max_{i_t} \norm{\oper P_{\Gamma \cup R} \nabla f_{i_t} (w^{\star})}_2}{\min_{i_t} Mp(i_t)}$, and $c = \E_{i_t} \frac{\rho^+_{4k}(i_t) }{Mp(i_t)} \norm{y}_2^2 - \rho^-_{4k} \norm{\Delta}_2^2$, we obtain 
\begin{equation}
\label{inq::proof Lemma 3-4th inequality}
\E_{i_t} \norm{\Delta - y}_2 \leq \sqrt{\frac{c}{a}} + \frac{2b}{a}.
\end{equation}

\noindent Now plugging the definition of $y$ we can obtain the lower bound of the left-hand side. We have,
\begin{equation}
\begin{split}
\nonumber
\norm{\Delta - y }_2 &= \norm{\Delta - z + \frac{\sqrt{\eta^2_1-1}}{\eta_1} x}_2 \\
&\geq \norm{\Delta - z}_2 - \norm{\frac{\sqrt{\eta^2_1-1}}{\eta_1} x}_2 \\
&= \norm{\Delta + \frac{\norm{\Delta}_2}{\norm{\oper P_{\Gamma} \nabla f_{i_t}(w^t)}_2} \oper P_{\Gamma} \nabla f_{i_t}(w^t) }_2 - \norm{\frac{\sqrt{\eta^2_1-1}}{\eta_1}  \frac{\nabla f_{i_t}(w^t)}{\norm{\nabla f_{i_t}(w^t)}_2} \norm{\Delta}_2}_2 \\
&\geq \norm{\Delta - \oper P_{\Gamma} \Delta}_2 - \frac{\sqrt{\eta^2_1-1}}{\eta_1} \norm{\Delta}_2,
\end{split}
\end{equation}
where the first inequality follows from the triangular argument; the last inequality follows from the observation that for any vector $v$, $\norm{\Delta - \oper P_{\Gamma} v }_2 \geq \norm{\Delta - \oper P_{\Gamma} \Delta}_2$. Here, $v = -\frac{\norm{\Delta}_2}{ \norm{\oper P_{\Gamma} \nabla f_{i_t}(w^t)}_2} \nabla f_{i_t}(w^t)$. Therefore,
\begin{equation}
\nonumber
\E_{i_t} \norm{\Delta - y }_2 \geq \E_{i_t} \norm{\Delta - \oper P_{\Gamma} \Delta}_2 - \frac{\sqrt{\eta^2_1-1}}{\eta_1} \norm{\Delta}_2.
\end{equation}
Plugging this inequality into (\ref{inq::proof Lemma 3-4th inequality}), we get
\begin{equation}
\label{inq::proof Lemma 3-5th inequality}
\E_{i_t} \norm{\Delta - \oper P_{\Gamma} \Delta}_2 \leq  \sqrt{\frac{c}{a}} + \frac{2b}{a} + \frac{\sqrt{\eta^2_1-1}}{\eta_1} \norm{\Delta}_2.
\end{equation}

\noindent The last step is to substitute values of $a$, $b$, and $c$ defined above into this inequality. From the definition of $y$ together with the observation that $x$ is orthogonal with $z$, we have
\begin{equation}
\begin{split}
\nonumber
\norm{y}^2_2 &= \norm{z}^2_2 + \frac{\eta_1^2-1}{\eta^2_1} \norm{x}^2_2 \\
&= \norm{\frac{\oper P_{\Gamma} \nabla f_{i_t}(w^t)}{\norm{\oper P_{\Gamma} \nabla f_{i_t}(w^t)}_2} \norm{\Delta}_2}^2_2 + \frac{\eta_1^2-1}{\eta^2_1} \norm{\frac{\oper P_{\Gamma^c} \nabla f_{i_t}(w^t)}{\norm{\oper P_{\Gamma^c} \nabla f_{i_t}(w^t)}_2} \norm{\Delta}_2}^2_2 \\
&= \norm{\Delta}^2_2 + \frac{\eta_1^2-1}{\eta^2_1} \norm{\Delta}^2_2 = \frac{2\eta_1^2-1}{\eta^2_1} \norm{\Delta}^2_2.
\end{split}
\end{equation} 
Thus, the quantity $c$ defined above is bounded by 
\begin{equation}
\begin{split}
\nonumber
c &= \E_{i_t} \frac{\rho^+_{4k}(i_t)}{Mp(i_t)} \frac{2\eta_1^2-1}{\eta^2_1} \norm{\Delta}_2^2 - \rho^-_{4k} \norm{\Delta}_2^2 \\
&\leq \max_{i_t} \rho^+_{4k}(i_t) \E_{i_t} \frac{1}{Mp(i_t)} \frac{2\eta_1^2-1}{\eta^2_1} \norm{\Delta}_2^2 - \rho^-_{4k} \norm{\Delta}_2^2 \\
&= \left(\frac{2\eta_1^2-1}{\eta^2_1} \rho^+_{4k} - \rho^-_{4k} \right) \norm{\Delta}_2^2. 
\end{split}
\end{equation}

\noindent Now combine this inequality with (\ref{inq::proof Lemma 3-5th inequality}) and plug values of $a$ and $b$, we obtain
\begin{equation}
\begin{split}
\nonumber
\E_{i_t} \norm{\Delta - \oper P_{\Gamma} \Delta}_2 &\leq  \sqrt{\frac{c}{a}} + \frac{2b}{a} + \frac{\sqrt{\eta^2_1-1}}{\eta_1} \norm{\Delta}_2 \\
&\leq \max_{i_t} \sqrt{Mp(i_t)} \sqrt{\frac{\frac{2\eta_1^2-1}{\eta^2_1}\rho^+_{4k} - \rho^-_{4k}}{\rho^-_{4k}}} \norm{\Delta}_2 \\
&\quad+ \frac{2\max_{i_t} p(i_t)}{\rho^-_{4k} \min_{i_t} p(i_t)} \max_{|\Omega| \leq 4k, i_t \in [M]}\norm{\oper P_{\Omega} \nabla f_{i_t} (w^{\star})}_2 + \frac{\sqrt{\eta^2_1-1}}{\eta_1} \norm{\Delta}_2 \\
&= \left( \max_{i_t} \sqrt{Mp(i_t)} \sqrt{\frac{\frac{2\eta_1^2-1}{\eta^2_1} \rho^+_{4k} - \rho^-_{4k}}{\rho^-_{4k}}}  + \frac{\sqrt{\eta^2_1-1}}{\eta_1} \right) \norm{\Delta}_2 \\
&\quad+ \frac{2\max_{i_t} p(i_t)}{\rho^-_{4k} \min_{i_t} p(i_t)} \max_{|\Omega| \leq 4k, i_t \in [M]}\norm{\oper P_{\Omega} \nabla f_{i_t} (w^{\star})}_2
\end{split}
\end{equation}
The proof follows by combining this inequality with (\ref{inq::1st intemediate result}).
\end{proof}

\subsection{Proof of Theorem \ref{thm::StoIHT with inexact gradients}}
\label{subsection::proof of StoIHT with inexact gradients}
\begin{proof}
At the $t$-th iteration, denote $g_{i_t} (w) \triangleq \nabla f_{i_t}(w) + e^t$. The proof of Theorem \ref{thm::StoIHT with inexact gradients} is essentially the same as that of Theorem \ref{thm::StoIHT} with $\nabla f_{i_t} (w^t)$ and $\nabla f_{i_t} (w^{\star})$ replaced by $g_{i_t} (w^t)$ and $g_{i_t} (w^{\star})$, respectively. Following the same proof as in Theorem \ref{thm::StoIHT}, we arrive at
\begin{equation}
\begin{split}
\nonumber
&\E_{i_t | I_{t-1}} \norm{w^{t+1} - w^{\star}}_2 \\
&\leq 2\left( \E_{i_t | I_{t-1}} \norm{w^t - w^{\star} - \frac{\gamma}{Mp(i_t)} \oper P_{\Omega} \left( g_{i_t}(w^t) - g_{i_t}(w^{\star}) \right)}_2 + \E_{i_t | I_{t-1}} \norm{\frac{\gamma}{Mp(i_t)} \oper P_{\Omega} g_{i_t} (w^{\star})}_2 \right)  \\
&\quad+ \sqrt{\eta^2-1} \left( \E_{i_t | I_{t-1}} \norm{w^t - w^{\star} - \frac{\gamma}{Mp(i_t)} \left( g_{i_t} (w^t) - g_{i_t} (w^{\star}) \right)}_2 + \E_{i_t | I_{t-1}} \norm{\frac{\gamma}{Mp(i_t)} g_{i_t} (w^{\star})}_2 \right) .
\end{split}
\end{equation}
Notice that $g_{i_t}(w^t) - g_{i_t}(w^{\star}) = \nabla f_{i_t}(w^t) - \nabla f_{i_t}(w^{\star})$, we can apply the inequality (\ref{inq::1st key observation}) of Corollary \ref{cor::StoIHT 1st corollary} for the first term and (\ref{inq::2nd key observation}) for the third term of the summation to obtain
\begin{equation}
\begin{split}
\nonumber
&\E_{i_t | I_{t-1}} \norm{w^{t+1} - w^{\star}}_2 \\
&\leq 2 \sqrt{\left(1- (2\gamma - \gamma^2 \alpha_{3k}) \rho^-_{3k} \right)} \norm{w^t-w^{\star}}_2 + 2 \frac{\gamma}{\min_{i_t} Mp(i_t)} \E_{i_t}\norm{\oper P_{\Omega} \nabla g_{i_t} (w^{\star})}_2 \\
&\quad+ \sqrt{\eta^2-1}  \sqrt{1 + \gamma^2 \alpha_{3k} \overline{\rho}^+_{3k} - 2\gamma \rho^-_{3k} } \norm{w^t-w^{\star}}_2 + \sqrt{\eta^2-1}  \frac{\gamma}{\min_{i_t} Mp(i_t)} \E_{i_t} \norm{\nabla g_{i_t} (w^{\star})}_2 \\
&= \left( 2 \sqrt{\left(1- (2\gamma - \gamma^2 \alpha_{3k}) \rho^-_{3k} \right)} + \sqrt{(\eta^2-1)\left(1 + \gamma^2 \alpha_{3k} \overline{\rho}^+_{3k} - 2\gamma \rho^-_{3k} \right)} \right) \norm{w^t - w^{\star}}_2 \\
&\quad+ \frac{\gamma}{\min_{i_t} Mp(i_t)} \left( 2 \E_{i_t} \norm{\oper P_{\Omega} \nabla g_{i_t} (w^{\star})}_2 + \sqrt{\eta^2 - 1} \E_{i_t} \norm{\nabla g_{i_t} (w^{\star})}_2 \right) \\
&\leq \kappa \norm{w^t -w^{\star}}_2 +  (\sigma_{w^{\star}} + \sigma_{e^t} ),
\end{split}
\end{equation}
where $\kappa$ and $\sigma_{w^{\star}}$ are defined in (\ref{eqt::kappa of StoIHT}) and (\ref{eqt::sigma of StoIHT}) and 
$$
\sigma_{e^t} \triangleq \frac{\gamma}{\min_{i} Mp(i)} \left( 2 \max_{|\Omega| \leq 3k} \norm{\oper P_{\Omega} e^t}_2 + \sqrt{\eta^2 - 1} \norm{e^t}_2 \right).
$$
Taking the expectation on both sides with respect to $I_{t-1}$ yields
$$
\E_{I_t} \norm{w^{t+1}-w^{\star}}_2 \leq \kappa \E_{I_{t-1}} \norm{w^t - w^{\star}}_2 + (\sigma_{w^{\star}} + \sigma_{e}),
$$
where $\sigma_e = \max_{j \in [t]} \sigma_{e^j}$. Applying this result recursively completes the proof.

\end{proof}

\subsection{Proof of Theorem \ref{thm::StoGradMP with inexact gradients}}

\begin{proof}
The analysis of Theorem \ref{thm::StoGradMP with inexact gradients} follows closely to that of Theorem \ref{thm::StoGradMP}. In particular, we will apply Lemmas \ref{lem::bound l2 w^(t+1)-w*}, \ref{lem::bound l2 b^t - w*}, and the following lemma to obtain the proof. 
\begin{lem}
\label{lem::bound L2 P_Gamma^c b^t-w*-inexact gradient}
Denote $\widehat{\Gamma}$ as the set obtained from the $t$-th iteration. Then,
\begin{equation}
\begin{split}
\E_{i_t} \norm{\oper P_{\widehat{\Gamma}^c} (b^t - w^{\star})}_2  &\leq \left( \max_i \sqrt{Mp(i)} \sqrt{\frac{\frac{2\eta_1^2-1}{\eta^2_1}\rho^+_{4k} - \rho^-_{4k}}{\rho^-_{4k}}}  + \frac{\sqrt{\eta^2_1-1}}{\eta_1} \right) \norm{w^t-w^{\star}}_2 + \sigma_2,
\end{split}
\end{equation}
where 
$$
\sigma_2 \triangleq \frac{2 \max_{i_t} p(i_t)}{\rho^-_{4k} \min_{i_t} p(i_t)} \left( \max_{|\Omega| \leq 4k, i_t \in [M]}\norm{\oper P_{\Omega} \nabla f_{i_t} (w^{\star})}_2 + \max_t \norm{e^t}_2 \right).
$$
\end{lem}

Given this lemma, the proof is exactly the same as that of Theorem \ref{thm::StoGradMP}. Now, we proceed to prove Lemma \ref{lem::bound L2 P_Gamma^c b^t-w*-inexact gradient}. Denote $\Delta = w^{\star} - w^t$ and $g_{i_t}(w) \triangleq \nabla f_{i_t} (w) + e^t$. Similar to the analysis of Lemma \ref{lem::bound L2 P_Gamma^c b^t-w*}, we start by applying the $\oper D$-RSC,
\begin{equation}
\begin{split}
\label{inq::proof Lemma 3-1st inequality-inexact gradient}
F(w^{\star}) &- F(w^t) - \frac{\rho^-_{4k}}{2} \norm{w^{\star} - w^t}_2^2 \\
&\geq \inner{\nabla F(w^t),w^{\star} - w^t} \\
&= \E_{i_t} \inner{\frac{1}{Mp(i_t)} \nabla f_{i_t}(w^t),\Delta} \\
&= \E_{i_t} \inner{\frac{1}{Mp(i_t)}  \oper P_R (g_{i_t}(w^t)) - e^t,\Delta} \\
&\geq - \E_{i_t} \norm{\frac{1}{Mp(i_t)}  \oper P_R g_{i_t} (w^t)}_2 \norm{\Delta}_2 - \E_{i_t} \inner{\frac{1}{Mp(i_t)}  \oper P_R e^t ,\Delta} .
\end{split} 
\end{equation}
Again, applying inequality (\ref{eqt::approximation consequence 2}) allows us to write $\norm{\oper P_R g_{i_t}(w^t)}_2 \leq \norm{\oper P_{\Gamma} g_{i_t}(w^t)}_2 + \frac{\sqrt{\eta^2_1-1}}{\eta_1} \norm{\oper P_{\Gamma^c} g_{i_t}(w^t)}_2$. We now apply this observation to the above inequality. Denote $z \triangleq -\frac{\oper P_{\Gamma} g_{i_t}(w^t)}{\norm{\oper P_{\Gamma} g_{i_t}(w^t)}_2} \norm{\Delta}_2$ and $x \triangleq \frac{\oper P_{\Gamma^c} g_{i_t}(w^t)}{\norm{\oper P_{\Gamma^c} g_{i_t}(w^t)}_2} \norm{\Delta}_2$ and follow the same procedure in formula (\ref{inq::proof Lemma 3-2nd inequality}) with $\nabla f_{i_t}(w^t)$ replaced by $g_{i_t}(w^t)$, we arrive at
\begin{equation}
\begin{split}
\label{inq::proof Lemma 3-2nd inequality-inexact gradient}
&- \E_{i_t} \norm{\frac{1}{Mp(i_t)}  \oper P_R g_{i_t}(w^t)}_2 \norm{\Delta}_2 \\
&\geq -\E_{i_t} \norm{\frac{1}{Mp(i_t)}  \oper P_{\Gamma} g_{i_t}(w^t)}_2 \norm{\Delta}_2 -  \frac{\sqrt{\eta^2_1-1}}{\eta_1} \E_{i_t} \norm{\frac{1}{Mp(i_t)} \oper P_{\Gamma^c} g_{i_t}(w^t)}_2 \norm{\Delta}_2 \\
&= \E_{i_t} \inner{\frac{1}{Mp(i_t)} g_{i_t}(w^t), z - \frac{\sqrt{\eta^2_1-1}}{\eta_1} x} \\
&= \E_{i_t} \inner{\frac{1}{Mp(i_t)} \nabla f_{i_t}(w^t), z - \frac{\sqrt{\eta^2_1-1}}{\eta_1} x} + \E_{i_t} \inner{\frac{1}{Mp(i_t)} e^t, z - \frac{\sqrt{\eta^2_1-1}}{\eta_1} x}.
\end{split} 
\end{equation}
Denote $y \triangleq z - \frac{\sqrt{\eta^2_1-1}}{\eta_1} x$ and combine (\ref{inq::proof Lemma 3-1st inequality-inexact gradient}) and (\ref{inq::proof Lemma 3-2nd inequality-inexact gradient}), we get
\begin{equation}
\label{inq::proof Lemma 3-3rd inequality-inexact gradient}
\begin{split}
F(w^{\star}) - F(w^t) - \frac{\rho^-_{4k}}{2} \norm{\Delta}_2^2 &\geq \E_{i_t} \inner{\frac{1}{Mp(i_t)} \nabla f_{i_t}(w^t), y} - \E_{i_t} \inner{\frac{1}{Mp(i_t)} e^t, \Delta - y} \\
&\geq \E_{i_t} \inner{\frac{1}{Mp(i_t)} \nabla f_{i_t}(w^t), y} - \frac{1}{\min_{i_t} Mp(i_t)} \norm{e^t}_2 \E_{i_t} \norm{\Delta - y}_2,
\end{split}
\end{equation}
where the last argument follows from Cauchy-Schwarz inequality. We now use the $\oper D$-RSS property to lower bound the right-hand side of the above inequality. Recall that from the definition of $\oper D$-RSS, we can show that
\begin{equation}
\begin{split}
\nonumber
\inner{\nabla f_{i_t}(w^t), y} &\geq f_{i_t}(w^t+ y) - f_{i_t}(w^t) - \frac{\rho^+_{4k}(i_t)}{2} \norm{y}_2^2.
\end{split}
\end{equation}
Multiply both sides with $\frac{1}{Mp(i_t)}$ and take the expectation on both sides with respect to the index $i_t$ and recall that $\E_{i_t} \frac{1}{Mp(i_t)} f_{i_t} (w^t) = F(w^t)$, we have
$$
\E_{i_t} \inner{\frac{1}{Mp(i_t)}  \nabla f_{i_t}(w^t), y} \geq \E_{i_t} \frac{1}{Mp(i_t)}  f_{i_t }(w^t + y) - F(w^t) - \frac{1}{2} \E_{i_t} \frac{\rho^+_{4k}(i_t)}{Mp(i_t)} \norm{y}_2^2.
$$

\noindent Combining with inequality (\ref{inq::proof Lemma 3-3rd inequality-inexact gradient}) and removing the common terms yields
\begin{equation}
\begin{split}
\nonumber
 \frac{1}{2}\E_{i_t} &\frac{\rho^+_{4k}(i_t)}{Mp(i_t)} \norm{y}_2^2 - \frac{\rho^-_{4k}}{2} \norm{\Delta}_2^2 \\
&\geq \E_{i_t} \frac{1}{Mp(i_t)}  f_{i_t }(w^t+ y) - F(w^{\star}) - \frac{1}{\min_{i_t} Mp(i_t)} \norm{e^t}_2 \E_{i_t} \norm{\Delta - y}_2\\
&= \E_{i_t} \frac{1}{Mp(i_t)} \left( f_{i_t }(w^t+ y) - f_{i_t }(w^{\star}) \right) - \frac{1}{\min_{i_t} Mp(i_t)} \norm{e^t}_2 \E_{i_t} \norm{\Delta - y}_2.
\end{split}
\end{equation}
Apply the $\oper D$-RSC one more time to the right-hand side and follow a similar procedure as (\ref{inq::proof Lemma 3-quadratic inequality}) with an additional term involving the gradient noise, we get
\begin{equation}
\begin{split}
  \frac{1}{2}\E_{i_t} \frac{\rho^+_{4k}(i_t)}{Mp(i_t)} \norm{y}_2^2 & - \frac{\rho^-_{4k}}{2} \norm{\Delta}_2^2 \\
&\geq \frac{\rho^-_{4k}}{2\max_{i_t} Mp(i_t)} \left(\E_{i_t} \norm{\Delta - y}_2 \right)^2 - \frac{\max_{i_t} \norm{\oper P_{\Gamma \cup R} \nabla f_{i_t} (w^{\star})}_2}{\min_{i_t} Mp(i_t)} \E_{i_t}\norm{\Delta - y}_2 \\
&\quad- \frac{1}{\min_{i_t} Mp(i_t)} \norm{e^t}_2 \E_{i_t} \norm{\Delta - y}_2.
\end{split}
\end{equation}

\noindent Solving the quadratic inequality $au^2 - 2bu -c\leq 0$ where $u = \E_{i_t} \norm{\Delta - y}_2$, $a = \frac{\rho^-_{4k}}{\max_{i_t} Mp(i_t)}$, $b = \frac{\max_{i_t} \norm{\oper P_{\Gamma \cup R} \nabla f_{i_t} (w^{\star})}_2}{\min_{i_t} Mp(i_t)}+\frac{1}{\min_{i_t} Mp(i_t)} \norm{e^t}_2$, and $c =  \E_{i_t} \frac{\rho^+_{4k}(i_t)}{Mp(i_t)} \norm{y}_2^2 - \rho^-_{4k} \norm{\Delta}_2^2$, we obtain 
\begin{equation}
\label{inq::proof Lemma 3-4th inequality-inexact gradient}
\E_{i_t} \norm{\Delta - y}_2 \leq \sqrt{\frac{c}{a}} + \frac{2b}{a}.
\end{equation}

\noindent Following the same steps after inequality (\ref{inq::proof Lemma 3-4th inequality}), we arrive at 
\begin{equation}
\begin{split}
\nonumber
\E_{i_t} \norm{\Delta - \oper P_{\Gamma} \Delta}_2 &\leq  \sqrt{\frac{c}{a}} + \frac{2b}{a} + \frac{\sqrt{\eta^2_1-1}}{\eta_1} \norm{\Delta}_2 \\
&= \max_i \sqrt{Mp(i)} \sqrt{\frac{\frac{2\eta_1^2-1}{\eta^2_1}\rho^+_{4k} - \rho^-_{4k}}{\rho^-_{4k}}} \norm{\Delta}_2 \\
&\quad+ \frac{2\max_i p(i)}{\rho^-_{4k} \min_i p(i)} \left( \max_{|\Omega| \leq 4k, i \in [M]}\norm{\oper P_{\Omega} \nabla f_i (w^{\star})}_2 + \norm{e^t}_2 \right) + \frac{\sqrt{\eta^2_1-1}}{\eta_1} \norm{\Delta}_2 \\
&= \left( \max_i \sqrt{Mp(i)} \sqrt{\frac{\frac{2\eta_1^2-1}{\eta^2_1} \rho^+_{4k} - \rho^-_{4k}}{\rho^-_{4k}}}  + \frac{\sqrt{\eta^2_1-1}}{\eta_1} \right) \norm{\Delta}_2 \\
&\quad+ \frac{2 \max_i p(i)}{\rho^-_{4k} \min_i p(i)} \left(\max_{|\Omega| \leq 4k, i \in [M]}\norm{\oper P_{\Omega} \nabla f_i (w^{\star})}_2 + \norm{e^t}_2 \right).
\end{split}
\end{equation}
The proof follows by combining this inequality with (\ref{inq::1st intemediate result}).

\end{proof}

\subsection{Proof of Theorem \ref{thm::StoGradMP with inexact gradient and approximated estimation}}

\begin{proof}
The proof follows exactly the same steps as that of Theorem \ref{thm::StoGradMP} in Section \ref{subsec::Proof of StoGradMP theorem}. The only difference is Lemma \ref{lem::bound l2 b^t - w*}, which is now replaced by the following lemma.
\begin{lem}
\label{lem::bound l2 b^t - w* - approximated estimation}
Denote $\widehat{\Gamma}$ as the set obtained from the $t$-th iteration and $i$ as the index selected randomly from $[M]$ with probability $p(i)$. We have,
\begin{equation}
\begin{split}
\nonumber
\E_{I_t} \norm{b^t - w^{\star}}_2 &\leq \sqrt{\frac{\alpha_{4k}}{\rho^-_{4k}}} \E_{I_t} \norm{P_{\widehat{\Gamma}^c} (b^t - w^{\star})}_2 + \sigma_1 + \epsilon^t,
\end{split}
\end{equation}
where $\alpha_k = \max_i \frac{\rho^+_k(i)}{Mp(i)}$ and $
\sigma_1 \triangleq \frac{3}{\rho^-_{4k}} \frac{1}{\min_i Mp(i)} \max_{|\Omega| \leq 3k, i \in [M]}\norm{\oper P_{\Omega} \nabla f_i(w^{\star}))}_2$.
\end{lem}

From the triangular inequality, 
$$
\E_{i_t} \norm{b^t - w^{\star}}_2 \leq \E_{i_t} \norm{b^t_{\opt} - w^{\star}}_2 + \E_{i_t} \norm{b^t - b^t_{\opt}}_2 \leq \E_{i_t} \norm{b^t_{\opt} - w^{\star}}_2 + \epsilon^t.
$$
Applying Lemma \ref{lem::bound l2 b^t - w*} to get upper bound for $\E_{i_t} \norm{b^t_{\opt} - w^{\star}}_2$ will complete the proof of this lemma.

\end{proof}

\bibliographystyle{plain}
\bibliography{all_references}
\end{document}